\documentclass[11pt]{amsart}
\usepackage[T1]{fontenc}
\usepackage[utf8]{inputenc}
\usepackage{lmodern}
\usepackage{amsfonts}
\usepackage{amsmath}
\usepackage{amssymb}
\usepackage{breqn}
\usepackage{url}
\usepackage{graphicx}
\usepackage{tikz} 
\usepackage[numbers,square]{natbib}
\usepackage[pdfstartview=FitH,hypertexnames=false]{hyperref}
\AddToHook{env/thebibliography/begin}{\interlinepenalty=10000\relax}
\hypersetup{
  hidelinks,
  bookmarksdepth=3,
  pdfdisplaydoctitle=true,
  pdftitle={The property (T) threshold for random groups and Banach fixed point properties},
  pdfauthor={Izhar Oppenheim},
  pdfsubject={Geometric group theory and Banach fixed point properties},
  pdfkeywords={random groups, Kazhdan property (T), affine isometric actions, fixed points, uniformly curved Banach spaces}
}
\providecommand{\doi}[1]{\url{https://doi.org/#1}}
\usepackage{amsthm}

\usepackage{verbatim}
\usepackage{bbm}
\usepackage{enumitem}
\usepackage{needspace}
\usepackage[expansion=false]{microtype}
\allowdisplaybreaks[1]
\AddToHook{enddocument/afterlastpage}{\catcode`\^=7\relax}

\newcommand{\im}{\operatorname{Im}}
\newcommand{\Ker}{\operatorname{Ker}}

\newcommand{\B}{\operatorname{\mathbb{E}}}

\newcommand{\id}{\operatorname{id}}

\newcommand{\bipartite}{\operatorname{bipartite}}
\newcommand{\sides}{\operatorname{sides}}

\newcommand{\GL}{\operatorname{GL}}
\newcommand{\Isom}{\operatorname{Isom}}
\newcommand{\cpnum}{\operatorname{cpnum}}
\newcommand{\RelLen}{\operatorname{l}}
\newcommand{\disp}{\operatorname{disp}}
\newcommand{\ucurvfun}{\operatorname{f}}

\usepackage{aliascnt}
\usepackage{cleveref}

\theoremstyle{plain}
\newtheorem{theorem}{Theorem}[section]
\newaliascnt{lemma}{theorem}
\newtheorem{lemma}[lemma]{Lemma}
\aliascntresetthe{lemma}
\newaliascnt{claim}{theorem}

\aliascntresetthe{claim}
\newaliascnt{proposition}{theorem}
\newtheorem{proposition}[proposition]{Proposition}
\aliascntresetthe{proposition}
\newaliascnt{corollary}{theorem}
\newtheorem{corollary}[corollary]{Corollary}
\aliascntresetthe{corollary}
\theoremstyle{definition}
\newaliascnt{definition}{theorem}
\newtheorem{definition}[definition]{Definition}
\aliascntresetthe{definition}
\newaliascnt{observation}{theorem}

\aliascntresetthe{observation}
\newaliascnt{example}{theorem}

\aliascntresetthe{example}
\newaliascnt{remark}{theorem}
\newtheorem{remark}[remark]{Remark}
\aliascntresetthe{remark}

\makeatletter
\let\paper@deferred@thm@head\deferred@thm@head
\def\deferred@thm@head#1{%
  \paper@deferred@thm@head{\MakeLinkTarget[theorem]{}#1}}
\makeatother
\crefformat{section}{\S#2#1#3} 
\crefformat{subsection}{\S#2#1#3}
\crefformat{subsubsection}{\S#2#1#3}

\makeatletter
\newcommand{\nameditem}[2]{%
  \item[(#1)]%
  \phantomsection
  \def\@currentlabel{#1}%
  \label{#2}%
}
\makeatother

\begin{document}

\title[The property $(T)$ threshold for random groups]{The property $(T)$ threshold for random groups and Banach fixed point properties}

\author{Izhar Oppenheim}
\address{Department of Mathematics, Ben-Gurion University of the Negev, Be'er Sheva, Israel}
\email{izharo@bgu.ac.il}
\thanks{ORCID: \href{https://orcid.org/0000-0002-1849-1851}{0000-0002-1849-1851}.}
\thanks{Supported by the Israel Science Foundation (grant no. 242/24).}
\date{}

\begin{abstract}
We determine the threshold for Kazhdan's property~\((T)\) in Gromov's
density model. At every density \(d>\frac14\), random groups have
property~\((T)\) asymptotically almost surely. Together with Ashcroft's
result below \(\frac14\), this establishes the sharp threshold.

More generally, for every fixed number of generators \(n\ge2\) and
density \(d\in(\frac14,1)\), we prove that, asymptotically almost surely,
every affine isometric action on a real \(L^p\)-space has a fixed point,
simultaneously for
\[
  1\le p\le C_{n,d}\sqrt{\frac{\RelLen}{\log\RelLen}},
\]
where \(\RelLen\) is the relator length and \(C_{n,d}>0\) depends only on
\(n\) and \(d\). Thus \(\frac14\) is also the threshold for each fixed
\(L^p\) fixed point property, \(1\le p<\infty\).
We establish a simultaneous theorem for
uniformly smooth real Banach spaces with common quantitative convexity
and curvature bounds, and a theorem for each fixed real Banach space
isomorphic to a uniformly curved space whose curvature modulus satisfies
a specified decay condition.

Following the approach of \.Zuk’s criterion, we establish a criterion for property~\((T)\) based on a quadratic form associated with random walks on auxiliary graphs. Our criterion requires this form to be negative on maps arising from nonzero harmonic cocycles. The key insight is that the structure of these graphs, together with harmonicity, forces the required negativity for random groups above density \(\frac14\).
\end{abstract}

\maketitle

\section{Introduction}

We prove that random groups in Gromov's density model have Kazhdan's
property~\((T)\) asymptotically almost surely at every density above \(\frac14\), completing the
determination of its threshold. The proof uses the first-letter
structure of reduced words together with harmonic cocycles. This
approach also gives fixed point properties on a growing range of
\(L^p\)-spaces and on broader classes of Banach spaces.

We work throughout with Gromov's density model. Fix an integer \(n\ge2\)
and a density \(0<d<1\). For \(\RelLen\in\mathbb N\), a random group in
\(\mathcal{D}(n,\RelLen,d)\) is a group
\[
  \Gamma=\langle s_1,\ldots,s_n\mid R\rangle,
\]
where \(R\) is chosen uniformly among the subsets of cardinality
\(\lfloor(2n-1)^{d\RelLen}\rfloor\) of the set of cyclically reduced words of
length \(\RelLen\) in the alphabet
\(S=\{s_1^{\pm1},\ldots,s_n^{\pm1}\}\).
A property holds \emph{asymptotically almost surely} (a.a.s.) if its
probability tends to \(1\) as \(\RelLen\to\infty\), with \(n\) and \(d\) fixed.
Our results allow \(\RelLen\) to tend to infinity through all positive integers.

\subsection{The threshold for property \texorpdfstring{\((T)\)}{(T)}}

For finitely generated groups, property~\((T)\) is equivalent to the
fixed point property for affine isometric actions on Hilbert spaces.
Our main result is the following.

\begin{theorem}\label{thm:intro-property-t}
Fix \(n\ge2\) and \(d\in(\frac14,1)\). A random group
\(\Gamma\in\mathcal{D}(n,\RelLen,d)\) has property~\((T)\) a.a.s.
\end{theorem}

Together with Ashcroft's theorem that property~\((T)\) fails a.a.s.\
at every density \(d<\frac14\)~\cite{Ash2}, this identifies \(\frac14\)
as the exact threshold. The behavior at the critical density
\(d=\frac14\) remains open.

The threshold problem was explicitly raised in Ollivier's 2005
survey~\cite[Section~IV.c, p.~70]{Oll05}; see also~\cite{KK}.
\.Zuk~\cite{Zuk} established property~\((T)\) above density \(\frac13\),
with a complete proof given by Kotowski--Kotowski~\cite{KK}.
Their argument assumes that the relator length is divisible by three;
Ashcroft~\cite{Ash23} subsequently removed this restriction.
Below the threshold, the range in which property~\((T)\) fails was
successively extended from \(d<\frac15\) by
Ollivier--Wise~\cite{OW11}, to \(d<\frac5{24}\) by
Mackay--Przytycki~\cite{MP15}, to \(d<\frac3{14}\) by
Montee~\cite{Mon23}, and finally to \(d<\frac14\) by
Ashcroft~\cite{Ash2}, who constructed actions with unbounded orbits on
finite-dimensional CAT(0) cube complexes.
The conjectured threshold \(\frac14\) is recorded
in~\cite{Mun22}, where it is attributed to Przytycki; the author had
already heard it from Marcin Kotowski circa 2012.
Theorem~\ref{thm:intro-property-t} resolves the remaining range
\(\frac14<d\le\frac13\).

\subsection{Outline of the proof}

The proof starts from the fixed point characterization of
property~\((T)\). We divide each relator into four pieces and build
graphs recording how these pieces occur next to one another. The new
ingredient is to keep track of their first letters, which link this
graph construction to the original generators. An action of the group
determines the actions of all the pieces, since each piece is a word
in those generators. Grouping the pieces by their first letter lets
us exploit this dependence. We show that it supplies the additional
information needed to turn expansion of the graphs into a fixed point
theorem for actions on Hilbert spaces.

The use of four pieces explains the threshold \(\frac14\). The graph
vertices represent words of about one quarter of the relator length,
while the edges come from the relators. Above density \(\frac14\),
there are exponentially more relators than such words. We prove that
the resulting random graphs have the expansion required by the
argument. The same strategy extends to the Banach spaces in our
theorems, with additional estimates reflecting their geometry.
\Cref{section: Historical overview and main proof ideas} gives the
detailed explanation and the connection with earlier approaches.

\subsection{The \texorpdfstring{\(L^p\)}{Lp} fixed point spectrum}\label{sec:intro-lp}

For a real Banach space \(\B\), a group \(\Gamma\) has
property~\((F_{\B})\) if every affine isometric action of \(\Gamma\) on
\(\B\) has a fixed point. For \(1\le p<\infty\), write \((F_{L^p})\) for
this property on every real \(L^p\)-space, and set
\[
  \mathcal{F}_L(\Gamma)
  =\{p\in[1,\infty):\Gamma\text{ has property }(F_{L^p})\}.
\]
Dru\c{t}u's question whether this spectrum is an
interval~\cite[Question~1.8]{CDH10} was studied by
Lavy--Olivier~\cite{LO21} and settled in general by
Marrakchi--de la Salle~\cite{MdS23}. For a finitely generated group with
property~\((T)\), it has the form
\[
  \mathcal{F}_L(\Gamma)=[1,p_\Gamma),
  \qquad 2<p_\Gamma\le\infty.
\]
Our next result gives a quantitative lower bound for this endpoint.

\begin{theorem}\label{thm:intro-lp}
Fix \(n\ge2\) and \(d\in(\frac14,1)\). There exists a constant
\(C_{n,d}>0\) such that a random group \(\Gamma\in\mathcal{D}(n,\RelLen,d)\)
a.a.s.\ satisfies
\[
  \left[1,C_{n,d}\sqrt{\frac{\RelLen}{\log\RelLen}}\right]
  \subseteq\mathcal{F}_L(\Gamma).
\]
In particular, for every fixed \(1\le p<\infty\), a random group in
\(\mathcal{D}(n,\RelLen,d)\) has property~\((F_{L^p})\) a.a.s.
\end{theorem}

Since \((F_{L^p})\) implies property~\((T)\) for every
\(1\le p<\infty\)~\cite[Corollary~1.5]{CDH10}, Ashcroft's theorem gives
\(\mathcal{F}_L(\Gamma)=\varnothing\) a.a.s.\ at every density
\(d<\frac14\). Thus \(\frac14\) is the threshold for each fixed
property~\((F_{L^p})\).

At densities \(d>\frac13\), Dru\c{t}u--Mackay
\cite[Theorem~1.11]{DrutuM} proved that
\([2,p]\subseteq\mathcal{F}_L(\Gamma)\) a.a.s.\ for each fixed \(p\ge2\)
when \(n\ge10\cdot2^p\); de Laat--de la Salle~\cite{LaatSalle}
studied fixed point properties in a different random group model.
For fixed \(n\ge2\) and \(\frac13<d<\frac12\),
\cite[Theorems~1.5 and~4.17]{OppZukBan} subsequently gave the stronger
linear range
\[
  \left[1,\,
    a\left(d-\frac13\right)\log(2n-1)\,\RelLen-b
  \right]
  \subseteq\mathcal{F}_L(\Gamma),
\]
a.a.s.\ along lengths divisible by three, for universal constants
\(a,b>0\). Theorem~\ref{thm:intro-lp} extends a growing range of fixed
point properties to \(\frac14<d\le\frac13\), without a divisibility
restriction.

\paragraph{\textbf{Conformal dimension}.}
For a non-elementary hyperbolic group \(\Gamma\), write
\(\operatorname{Confdim}_{\mathrm{AR}}(\partial_\infty\Gamma)\) for the
Ahlfors regular conformal dimension of its boundary, a quasi-isometry
invariant. Bourdon~\cite{Bou16} constructed proper affine isometric
actions on \(\ell^p\) for \(p\ge1\) above this dimension; see also
Nica~\cite{Nic13} for a related construction. Hence \((F_{L^p})\), for
\(p\ge2\), forces this dimension to be at least \(p\).
Theorem~\ref{thm:intro-lp} therefore has the following consequence.

\begin{corollary}\label{cor:intro-confdim}
Fix \(n\ge2\) and \(d\in(\frac14,\frac12)\). A random group
\(\Gamma\in\mathcal{D}(n,\RelLen,d)\) is a.a.s.\ non-elementary hyperbolic and satisfies
\[
  \operatorname{Confdim}_{\mathrm{AR}}(\partial_\infty\Gamma)
  \ge C_{n,d}\sqrt{\frac{\RelLen}{\log\RelLen}},
\]
where \(C_{n,d}>0\) is the constant in Theorem~\ref{thm:intro-lp}.
\end{corollary}

In particular, a random group in this range a.a.s.\ avoids any fixed
finite collection of quasi-isometry classes. The corollary illustrates
the geometric consequence of our fixed point theorem; its growth rate
is weaker than the known linear bounds. Such lower bounds follow from
fixed point methods for \(\frac13<d<\frac12\), along lengths divisible
by three~\cite[Corollary~1.6 and Theorem~1.8]{OppZukBan}, and from
geometric constructions for \(0<d<\frac18\) by
Mackay~\cite[Theorem~1.3]{Mac16} and throughout \(0<d<\frac12\) by
Frost~\cite{Fro22}. Together with Mackay's upper
bound~\cite[Proposition~1.7]{Mac12}, the latter gives growth of order
\(\RelLen\) throughout the hyperbolic density range.

\subsection{Fixed points on uniformly curved Banach spaces}

The method also applies beyond \(L^p\)-spaces. Write \(\delta_{\B}\) for
the modulus of convexity of a Banach space \(\B\), and
\(\ucurvfun_{\B}\) for its uniform-curvature modulus. Uniform curvature,
introduced by Pisier~\cite{Pisier2}, controls how scalar operator bounds
pass to \(\B\)-valued functions. The precise definitions are recalled in
\Cref{section: Preliminaries}.

We use \(\ucurvfun\) for a common upper bound on the curvature moduli.
For a nondecreasing function \(\ucurvfun:[0,1]\to[0,1]\) with
\(\ucurvfun(0)=0\), the required decay condition is
\begin{equation}\label{eq:intro-curvature-decay}
  \lim_{t\downarrow0} \ucurvfun(t)\bigl(\log(1/t)\bigr)^\beta=0
  \qquad\text{for every }\beta>0.
\end{equation}
Thus \(\ucurvfun\) decays faster than every negative power of
\(\log(1/t)\); any bound \(\ucurvfun(t)=O(t^\theta)\) with
\(\theta>0\) suffices.

\begin{theorem}\label{thm:intro-banach}
Fix \(n\ge2\), \(d\in(\frac14,1)\), \(2\le p<\infty\), and \(0<C\le1\).
Let \(\ucurvfun:[0,1]\to[0,1]\) be nondecreasing, with \(\ucurvfun(0)=0\), and satisfy
\eqref{eq:intro-curvature-decay}.
Then a random group \(\Gamma\in\mathcal{D}(n,\RelLen,d)\) a.a.s.\ has
property~\((F_{\B})\) simultaneously for every uniformly smooth real Banach
space \(\B\) satisfying
\[
  \delta_{\B}(\varepsilon)\ge C\varepsilon^p
  \quad(0<\varepsilon\le2),
  \qquad
  \ucurvfun_{\B}(t)\le \ucurvfun(t)
  \quad(0<t\le1).
\]
\end{theorem}

Examples include the classical and noncommutative \(L^q\)-spaces for
each fixed \(1<q<\infty\); see~\cite{Pisier2,PX} and
\Cref{subsect Strictly theta-Hilbertian spaces}. For each fixed target
space, it suffices to be isomorphic to a space with the required
curvature decay.

\begin{theorem}
\label{thm:density-isomorphic-uniformly-curved}
Fix $n\geq2$ and $d\in(\frac14,1)$.
Let $\B$ be a fixed real Banach space isomorphic to a uniformly curved
real Banach space $\B_0$ whose curvature modulus satisfies
\[
\lim_{t\to0^+} \ucurvfun_{\B_0}(t)\bigl(\log(t^{-1})\bigr)^\beta=0
\qquad\text{for every }\beta>0.
\]
Then a random group $\Gamma\in\mathcal{D}(n,\RelLen,d)$ has property
$(F_{\B})$ a.a.s.
\end{theorem}

This includes every fixed real Banach space isomorphic to a subquotient
of a \(\theta\)-Hilbertian space for some \(0<\theta\le1\), as defined in
\Cref{subsect Strictly theta-Hilbertian spaces}. To the best of our
knowledge, this class contains all currently known examples of
super-reflexive Banach spaces. Whether every super-reflexive space has
such a representation remains open; counterexamples are conjectured
to exist~\cite[Remark~17]{Naor21}.

For comparison, at densities \(\frac13<d<\frac12\),
\cite[Theorem~1.5]{OppZukBan} proves \((F_{\B})\) for every fixed
uniformly curved space \(\B\), along lengths divisible by three,
without the additional convexity, smoothness, or curvature-rate
assumptions. The present results extend fixed point properties to
densities above \(\frac14\) under the stated geometric hypotheses.

\paragraph{\textbf{Organization}.}
\Cref{section: Historical overview and main proof ideas} explains the
proof ideas and their history in the Hilbert-space setting.
\Cref{section: Preliminaries} collects the background on Banach spaces,
affine actions, and random walks.
\Cref{section: Norm bounds for random walks on bipartite (multi)graphs}
establishes spectral estimates for bipartite multigraphs;
\Cref{section: Cocycle equalities and inequalities} develops the
harmonic-cocycle estimates; and
\Cref{section: Fixed point property criteria} proves the fixed point
criteria. Finally, \Cref{section: Fixed point properties for random groups}
verifies these criteria for random presentations and deduces the results
in the density model.

\section{Historical overview and main proof ideas}
\label{section: Historical overview and main proof ideas}

This section explains the mechanism behind the threshold $d=\frac14$.
We first recall the earlier approach through triangular presentations and
then describe the three main ideas of our proof. For the proof sketch,
we restrict to Hilbert spaces, assume that $\RelLen$ is divisible by $4$,
and impose an exact regularity condition on the auxiliary graph. The
adaptations needed for the full results are discussed at the end.

For readability, we use less formal notation in this overview than in
the later sections. For example, we write $M$ in place of
$M\otimes I_{\mathcal H}$ when the random walk operator acts on
$\mathcal H$-valued functions.
Throughout this section, graphs may have multiple edges, and degrees
are counted with multiplicity.

\paragraph*{\textbf{From triangular presentations to the density model.}}
In~\cite{Zuk}, \.Zuk gave a criterion for property~$(T)$ for a group
$\Gamma=\langle S\mid R\rangle$ whose relators have length $3$.
Assume that $S$ is symmetric and that $R$ is closed under inversion
and cyclic permutations. Consider the graph with vertex set $S$ in
which $s,t\in S$ are adjacent whenever $t^{-1}su\in R$ for some
$u\in S$, retaining the edge occurrences contributed by the relators.
This graph is the link of a vertex in the Cayley complex of $\Gamma$:
its vertices record the edge directions at that vertex, and its edges
record the corners of the triangular relator cells. Each relator thus
contributes three edges to the link.

\.Zuk's criterion states that if this graph is connected and the second
largest eigenvalue of its random walk is less than $\frac12$, then
$\Gamma$ has property~$(T)$. Here and below, eigenvalues are ordered
by their values, rather than their absolute values. The geometric
interpretation of the graph explains why this is an instance of
Garland's method, which deduces cohomological vanishing from expansion
of links~\cite{Garland1973Overview,BallmannSwiatkowski1997Overview}.
In the triangular random group model, the link behaves similarly to
an Erd\H{o}s--R\'{e}nyi random graph, and the required spectral estimates
hold a.a.s.\ when $d>\frac13$, as the number of generators tends to infinity.

For relators of length $\RelLen$ divisible by $3$, the analogous
geometric picture is to look at the sphere of radius $\RelLen/3$
around the identity in the Cayley graph and at how the relator cells
connect the directions leading to this sphere. Dividing each relator
into three blocks of length $\RelLen/3$ makes it a triangle at this
larger scale. At each corner, the two paths leaving the corner
suggest an edge in the corresponding link. This motivates what we
call the \emph{link at distance
$\RelLen/3$}. Explicitly, its vertices are the formal reduced words
of length $\RelLen/3$ that occur as prefixes of relators, and $w,w'$
are adjacent if $w^{-1}w'w''\in R$ for some word $w''$ of length
$\RelLen/3$. The sphere picture is a geometric motivation for this
combinatorial construction: formal words remain distinct vertices,
even if they represent the same group element or have shorter
representatives in the Cayley graph.

One would like to apply the same spectral criterion to this larger
link. Unlike an Erd\H{o}s--R\'{e}nyi graph, however, it has a restriction
coming from reduction: words with the same first letter cannot be
adjacent. Indeed, if $w=su$ and $w'=su'$, then $w^{-1}w'$ is not
reduced and cannot begin a cyclically reduced relator.

Kotowski and Kotowski~\cite[Section~3.3]{KK} overcome this restriction
and pass to the density model by treating suitable reduced words of
length $\RelLen/3$ as new generators.
They choose words whose first and last letters are positive in the
original alphabet, while allowing either sign in their interiors.
Concatenations of three such words are cyclically reduced and give a
positive triangular presentation in the new generators, whose link is
bipartite. The resulting triangular group maps onto a finite-index
subgroup of the original random group. Property~$(T)$ passes to
quotients and finite-index supergroups, giving the result for
$d>\frac13$ when $3$ divides $\RelLen$.

A direct replacement of the triangular model by the square model cannot
give the desired improvement: Odrzyg{\'o}{\'z}d{\'z} proved that random
groups in the square model do not have property~$(T)$ a.a.s.\ for
$d<\frac13$~\cite[Theorem~5.20]{OdrzygOzdz2016Overview}.
Our argument instead keeps track of the first letters inside the four
blocks and of the cocycles arising from the original group.

\paragraph*{\textbf{The auxiliary graph.}}
We now use the same geometric idea at distance $\RelLen/4$, dividing
each relator into four equal blocks and recording the corners of the
resulting square. Write $S=\{s_1^{\pm1},\ldots,s_n^{\pm1}\}$, and assume
for this sketch that $\RelLen\ge8$ is divisible by $4$ and that $R$
is closed under inversion and cyclic permutations. Let $V$ be the set
of all formal reduced words of length $\RelLen/4$ in $S$.
At each of the four corners of a relator, add an undirected edge
between the inverse of the preceding block and the following block.
Retain every edge occurrence, including repeated ones. By the
symmetrization assumption, two vertices $w,w'$ are adjacent precisely
when $w^{-1}w'w''\in R$ for some word $w''$ of length $\RelLen/2$.
We call this graph the \emph{link at distance $\RelLen/4$}, with the
same combinatorial interpretation of the sphere picture as above.
Figure~\ref{fig:quarter-link-cayley} shows how a relator cell
contributes an edge to this graph.

\begin{figure}[htbp]
\centering
\begin{tikzpicture}[x=1cm,y=1cm,
  every node/.style={font=\small},
  boundary/.style={draw=black!55,line width=.65pt},
  linkedge/.style={draw=black,line width=1.2pt},
  vertex/.style={circle,fill=black,inner sep=1.7pt}]
  \node[anchor=south] at (2.65,4.35) {(a) One corner of a relator cell};
  \node[anchor=south] at (8.95,4.35) {(b) The corresponding link edge};
  \fill[black!3] (2.65,0) -- (.85,1.8) -- (2.65,3.6)
    -- (4.45,1.8) -- cycle;
  \draw[boundary] (2.65,0) -- (.85,1.8) -- (2.65,3.6)
    -- (4.45,1.8) -- cycle;
  \draw[boundary,->] (2.20,.45) -- (1.75,.90);
  \draw[boundary,->] (3.10,.45) -- (3.55,.90);
  \draw[boundary,->] (4.00,2.25) -- (3.55,2.70);
  \draw[boundary,->] (2.20,3.15) -- (1.75,2.70);
  \node[rotate=45,font=\scriptsize] at (1.57,2.88)
    {$\RelLen/4$};
  \node[rotate=-45,font=\scriptsize] at (3.73,2.88)
    {$\RelLen/4$};
  \node[rotate=-45,font=\scriptsize] at (1.57,.72)
    {$\RelLen/4$};
  \node[rotate=45,font=\scriptsize] at (3.73,.72)
    {$\RelLen/4$};
  \node[vertex,label={[font=\small]below:$1$}] at (2.65,0) {};
  \node[vertex,label={[font=\small]left:$w$}] at (.85,1.8) {};
  \node[vertex,label={[font=\small]right:$w'$}] at (4.45,1.8) {};
  \fill[black!55] (2.65,3.6) circle[radius=1.1pt];
  \node[font=\small] at (2.65,3.08) {$w''$};
  \node[font=\scriptsize] at (2.65,3.87)
    {$\lvert w''\rvert=\RelLen/2$};
  \draw[linkedge] (.85,1.8) arc[start angle=135,end angle=45,radius=2.545584cm];
  \node[font=\scriptsize,fill=black!3,inner sep=1.5pt] at (2.65,1.78)
    {one link edge};
  \draw[black!60,line width=.4pt] (2.65,1.98)--(2.65,2.43);
  \draw[->,black!60,line width=.6pt] (5.15,2.05)--(6.45,2.05);
  \node[font=\scriptsize,anchor=south] at (5.80,2.16) {records};
  \draw[linkedge] (7.15,2.05)--(10.75,2.05);
  \node[vertex,label={[font=\small]above:$w$}] at (7.15,2.05) {};
  \node[vertex,label={[font=\small]above:$w'$}] at (10.75,2.05) {};
  \node at (8.95,1.42) {$w^{-1}w'w''\in R$};
  \node[font=\scriptsize] at (8.95,.94)
    {$\lvert w\rvert=\lvert w'\rvert=\RelLen/4$};
  \node[font=\scriptsize] at (8.95,.48) {vertices are formal reduced words};
\end{tikzpicture}
\caption{One relator corner in the link at distance $\RelLen/4$.
Each side in (a) represents a boundary path of length $\RelLen/4$.
The paths from $1$ read $w$ and $w'$, while the rest of the boundary,
traversed from the endpoint of $w'$ back to the endpoint of $w$,
reads $w''$. Thus $w^{-1}w'w''\in R$. The bold arc records the auxiliary
link edge in (b), not an edge of the Cayley graph. Repeating this
construction at every corner of every relator cell gives the auxiliary
graph. This is a schematic picture of the formal-word construction, not a claim that the arc
is the intersection of the cell with a metric sphere}
\label{fig:quarter-link-cayley}
\end{figure}

For $s\in S$, let $V_s$ be the set of words in $V$ whose first letter
is $s$. Each $V_s$ has cardinality $(2n-1)^{\RelLen/4-1}$. There are no edges within
$V_s$: if $w=su$ and $w'=su'$, then $w^{-1}w'$ is not reduced and cannot
begin a cyclically reduced relator. Thus the graph is $2n$-partite.
Figure~\ref{fig:first-letter-graph} illustrates the case
$S=\{s_1^{\pm1},s_2^{\pm1}\}$.

\begin{figure}[htbp]
\centering
\begin{tikzpicture}[x=1cm,y=1cm,
  vertex/.style={circle,fill=black,inner sep=1.5pt},
  mutededge/.style={draw=black!24,line width=.45pt},
  selectededge/.style={draw=black,line width=.85pt},
  class/.style={draw=black!55,fill=black!3,line width=.5pt},
  focusclass/.style={draw=black,fill=black!6,line width=.65pt},
  every node/.style={font=\small}]
  \node[anchor=south] at (2.4,2.30) {(a) Four first-letter classes};
  \node[anchor=south] at (8.75,2.30) {(b) The subgraph on $V_{s_1}\cup V_{s_2}$};
  \draw[focusclass] (1.00,1.10) ellipse [x radius=.85,y radius=.64];
  \draw[focusclass] (3.80,1.10) ellipse [x radius=.85,y radius=.64];
  \draw[class] (1.00,-1.00) ellipse [x radius=.85,y radius=.64];
  \draw[class] (3.80,-1.00) ellipse [x radius=.85,y radius=.64];
  \node at (1.00,1.96) {$V_{s_1}$};
  \node at (3.80,1.96) {$V_{s_2}$};
  \node at (1.00,-1.91) {$V_{s_1^{-1}}$};
  \node at (3.80,-1.91) {$V_{s_2^{-1}}$};
  \coordinate (a1) at (.66,1.30); \coordinate (a2) at (1.32,1.30);
  \coordinate (a3) at (1.00,.88);
  \coordinate (b1) at (3.46,1.30); \coordinate (b2) at (4.12,1.30);
  \coordinate (b3) at (3.80,.88);
  \coordinate (c1) at (.66,-.80); \coordinate (c2) at (1.32,-.80);
  \coordinate (c3) at (1.00,-1.22);
  \coordinate (d1) at (3.46,-.80); \coordinate (d2) at (4.12,-.80);
  \coordinate (d3) at (3.80,-1.22);
  \draw[mutededge] (a1)--(c1) (a3)--(c2) (a2)--(d1) (a3)--(d3)
    (b1)--(c3) (b3)--(c2) (b2)--(d2) (b3)--(d1)
    (c1)--(d1) (c3)--(d3);
  \draw[selectededge] (a1) to[bend left=9] (b1);
  \draw[selectededge] (a2) to[bend left=11] (b2);
  \draw[selectededge] (a3)--(b3);
  \foreach \v in {a1,a2,a3,b1,b2,b3,c1,c2,c3,d1,d2,d3}
    \node[vertex] at (\v) {};
  \node[font=\scriptsize] at (1.00,.64) {$\cdots$};
  \node[font=\scriptsize] at (3.80,.64) {$\cdots$};
  \node[font=\scriptsize] at (1.00,-1.46) {$\cdots$};
  \node[font=\scriptsize] at (3.80,-1.46) {$\cdots$};
  \draw[black!18] (5.15,-2.08)--(5.15,2.08);
  \draw[focusclass,rounded corners=9pt] (5.60,-1.22) rectangle (7.55,1.85);
  \draw[focusclass,rounded corners=9pt] (9.65,-1.22) rectangle (11.65,1.85);
  \node at (6.65,2.02) {$V_{s_1}$};
  \node at (10.55,2.02) {$V_{s_2}$};
  \coordinate (A1) at (7.20,1.25); \coordinate (A2) at (7.20,.35);
  \coordinate (A3) at (7.20,-.55);
  \coordinate (B1) at (10.00,1.25); \coordinate (B2) at (10.00,.35);
  \coordinate (B3) at (10.00,-.55);
  \draw[mutededge] (A1)--(B2) (A2)--(B2) (A2)--(B3)
    (A3)--(B1) (A3)--(B3);
  \draw[selectededge] (A1)--(B1);
  \foreach \v in {A1,A2,A3,B1,B2,B3}\node[vertex] at (\v) {};
  \node[anchor=east,font=\scriptsize] at (7.05,1.25) {$w=s_1u$};
  \node[anchor=west,font=\scriptsize] at (10.15,1.25) {$w'=s_2u'$};
  \node[font=\scriptsize] at (6.65,-.88) {$\vdots$};
  \node[font=\scriptsize] at (10.55,-.88) {$\vdots$};
  \node[align=center,font=\scriptsize] at (8.60,-1.83)
    {$w\sim w'\ \Longleftrightarrow\ w^{-1}w'w''\in R$\\[2pt]
     for some word $w''$ of length $\RelLen/2$};
\end{tikzpicture}

\caption{The graph for $S=\{s_1^{\pm1},s_2^{\pm1}\}$.
Each vertex is a formal reduced word of length $\RelLen/4$, and $V_s$ consists of those
whose first letter is $s$. There are no edges within a class.
The subgraph on $V_{s_1}\cup V_{s_2}$ is highlighted in (a) and enlarged
in (b). An edge records a relator of the form $w^{-1}w'w''$, where
$\lvert w''\rvert=\RelLen/2$. The drawing is schematic: only selected vertices
and edges are shown, and edge multiplicities are omitted}
\label{fig:first-letter-graph}
\end{figure}

We now make the simplifying assumption of \emph{local regularity}:
there is a positive integer $D$ such that, for every pair $s\ne s'$,
the subgraph on $V_s\cup V_{s'}$ is $D$-regular. In other words, each
vertex in $V_s$ has exactly $D$ incident edges to $V_{s'}$, counted with
multiplicity. The whole graph is then $(2n-1)D$-regular. Let $M$ be its
random walk operator. All norms and inner products in this sketch
use vertex degrees as weights, so that $M$ is self-adjoint. We use the same
notation for operators and their extensions to Hilbert-valued functions.

The proof compares two conclusions about the same quadratic form:
the four-piece relations force it to be nonnegative, whereas harmonicity
and expansion force it to be negative. The following three observations
explain how this contradiction arises.

\paragraph*{\textbf{First insight: the first-letter classes give negative eigenvectors.}}
The restriction imposed by reduced words now works to our advantage.
Local regularity means that a step of the random walk chooses each of
the other $2n-1$ first-letter classes with equal probability. On
functions that are constant on each $V_s$, the walk therefore acts as
the random walk on a complete graph with $2n$ vertices. In particular,
for $s\ne s'$, the function
\[
 \phi_{s,s'}(w)=
 \begin{cases}
  1 & w\in V_s,\\
  -1 & w\in V_{s'},\\
  0 & \text{otherwise}
 \end{cases}
\]
is an eigenfunction with eigenvalue $-1/(2n-1)$. These functions span
the space of mean-zero functions that are constant on each first-letter
class. The first-letter partition therefore supplies explicit negative
spectral directions.

\paragraph*{\textbf{Second insight: only harmonic cocycle maps need to be tested.}}
We cannot expect $M$ to have a negative quadratic form on every
nonconstant function. Indeed, a nonzero mean-zero function supported
on a single $V_s$ has quadratic form zero, since there are no edges
within that class. In particular, the second largest eigenvalue of
$M$ is at least zero. The relevant maps, however, have additional
structure coming from the group.

A cocycle $c:\Gamma\to\mathcal H$ is the translation part of an
affine isometric action of $\Gamma$ on a real Hilbert space
$\mathcal H$; see \Cref{section: Preliminaries} for the precise definition.
It is \emph{harmonic} if $\sum_{s\in S}c(s)=0$.
Shalom's characterization~\cite{Shalom} says that failure of
property~$(T)$ yields a nonzero harmonic cocycle for some orthogonal
representation; see \Cref{thm: exist non-zero harmonic cocycle} for
the version used in this paper. Such a cocycle defines a map
$\phi_c:V\to\mathcal H$ by $\phi_c(w)=c(w)$, where a formal word is
evaluated in $\Gamma$ before applying $c$.

The cocycle identity around a four-piece relator gives a nonnegative
sum of its corner pairings; this is the Hilbert-space case of
\Cref{prop: relation coupling}. Summing over all relators, with the
edge multiplicities used above, gives
\begin{equation}\label{eq: overview nonnegative}
 \langle M\phi_c,\phi_c\rangle\ge0.
\end{equation}
It therefore suffices to prove strict negativity for every nonzero
harmonic cocycle map. This is the mechanism behind the fixed point
criterion in \Cref{prop: general Zuk}.

For comparison, \.Zuk's spectral condition on the triangular link is
equivalent to
\[
 \langle M\phi,\phi\rangle<\frac12\|\phi\|^2
 \qquad\text{for every nonzero mean-zero function }\phi.
\]
Here $M$ denotes the random walk on the triangular link, and the mean,
norm, and inner product use vertex degrees as weights. In our
four-piece setting, the required inequality is
\[
 \langle M\phi_c,\phi_c\rangle<0
 \qquad\text{for every nonzero harmonic cocycle }c,
\]
with $M$ now acting on the auxiliary graph. The numerical threshold
drops from $\frac12$ to $0$, but the test class is reduced from all
mean-zero functions to maps arising from harmonic cocycles. These
maps have additional structure, which the next insight uses to obtain
the required negativity.

\paragraph*{\textbf{Third insight: harmonicity gives a large first-letter component.}}
Let $P$ be the orthogonal projection onto the functions that are
constant on each first-letter class, given by averaging within each
class. For a harmonic cocycle, these averages
can be computed explicitly using \Cref{prop: sum c (w)}:
\[
 (P\phi_c)(w)=\frac1{|V_s|}\sum_{w'\in V_s}c(w')
             =\alpha c(s)\qquad(w\in V_s),
\]
where
\[
 \alpha=\alpha(\RelLen,n)
 =1+\frac1{2n-1}+\cdots+\frac1{(2n-1)^{\RelLen/4-1}}\ge1.
\]
Thus averaging within a first-letter class retains a definite multiple
of the cocycle value on that letter. Since the classes have equal
cardinality and $c$ is harmonic, $\phi_c$ has mean zero and
$P\phi_c$ belongs to the eigenspace described above.
Moreover, the elementary word-norm estimate in
\Cref{prop: bound norm c (w)} gives
\[
 \|P\phi_c\|^2\ge\frac1{\RelLen^2}\|\phi_c\|^2.
\]
In particular, a nonzero harmonic cocycle yields a nonzero map with a
quantitatively large component in the explicit negative eigenspace.

Suppose now that the remaining spectrum is small: on functions with
mean zero on every first-letter class, the random walk has operator
norm less than $1/((2n-1)\RelLen^2)$. This is the part controlled by
expansion of the bipartite subgraphs on $V_s\cup V_{s'}$.
The map $\phi_c-P\phi_c$ has mean zero on each class and is orthogonal
to $P\phi_c$. Thus, for every nonzero harmonic cocycle, the negative
contribution from $P\phi_c$ dominates:
\begin{align*}
 \langle M\phi_c,\phi_c\rangle
 &< -\frac1{2n-1}\|P\phi_c\|^2
        +\frac1{(2n-1)\RelLen^2}\|\phi_c\|^2\\
 &\le0.
\end{align*}
This contradicts~\eqref{eq: overview nonnegative}.

\paragraph*{\textbf{Why the density threshold is $\frac14$.}}
There are $2n(2n-1)^{\RelLen/4-1}$ block vertices, while the number of
random relators is of order $(2n-1)^{d\RelLen}$. Up to polynomial factors arising
from symmetrization, the expected degrees therefore grow on the
exponential scale
\[
 (2n-1)^{(d-\frac14)\RelLen}.
\]
For $d>\frac14$, this growth allows the expansion and regularity errors
to decay exponentially, faster than the polynomial scale $\RelLen^{-2}$
required by the preceding argument. The probabilistic part of the
proof establishes these estimates for the bipartite subgraphs used
below; see \Cref{thm: local graphs bireg and spec}. We prove them in
the new range $\frac14<d\le\frac13$ and obtain the results at larger
densities by adding relators, since the fixed point properties under
consideration are preserved under quotients.

\paragraph*{\textbf{Adaptations for the full results.}}
The proof implements this mechanism with four additional ingredients.
\begin{itemize}
\item For arbitrary relator lengths, the four blocks need not have
equal lengths. We therefore work with four bipartite graphs, one for
each pair of consecutive blocks. Their construction also allows us
to use the original random relators without symmetrizing them; see
\Cref{subsection: Graphs associated to the group presentation}.
\item The random graphs satisfy approximate local biregularity.
This allows us to control the action of the random walk on first-letter
maps, with error terms measuring deviations from regularity. The
resulting estimates also accommodate unequal block lengths.
\item In Banach spaces, uniform convexity supplies inequalities for
duality pairings that replace the Hilbert-space projection argument;
see \Cref{prop: couplings sum ineq}. The power $p$ in
\Cref{thm:intro-banach} arises from this quantitative convexity
estimate, with $p=2$ in the Hilbert-space case.
\item Uniform curvature transfers scalar spectral estimates to
operator-norm bounds on Banach-valued functions; see
\Cref{prop: lambda ineq of uni-curved}.
\end{itemize}

\section{Preliminaries}
\label{section: Preliminaries}
\subsection{Notation for Banach spaces}

Throughout, $\B$ denotes a Banach space, and $\B^*$ denotes the dual of $\B$.  We also use $\Vert . \Vert_{\B},  \Vert . \Vert_{\B^*}$ for the norms of $\B,  \B^*$, respectively, and  $\langle .  ,  . \rangle_{\B} : \B \times \B^* \rightarrow \mathbb{R}$ for the pairing between $\B$ and $\B^*$.  When $\B$ and $\B^*$ are clear from the context,  we will sometimes omit the subscript $\B$ or $\B^*$ and simply write $\Vert . \Vert$ and $\langle . , . \rangle$.

\subsection{Affine isometric actions on Banach spaces}

Let $\B$ be a real Banach space and $\Gamma$ a discrete group.  A \textit{linear representation} of $\Gamma$ on $\B$ is a homomorphism $\pi : \Gamma \rightarrow \GL (\B)$.  A linear representation $\pi$ is called \textit{uniformly bounded} if $\sup_{g \in \Gamma} \Vert \pi (g) \Vert_{B (\B)} < \infty$.  A linear representation $\pi$ is called \textit{isometric} if  $\pi (g)$ is an isometry for every $g \in \Gamma$, or equivalently if $ \Vert \pi (g) \Vert_{B (\B)} =1$ for every $ g \in \Gamma$.

Given a representation $\pi$ of $\Gamma$ on $\B$, the \textit{dual representation of $\pi$} is the map $\overline{\pi} : \Gamma \rightarrow B(\B^*)$ satisfying, for every $\xi \in \B,  \eta \in \B^{*}$ and $g \in \Gamma$,
\[
\langle \pi (g) \xi,  \eta \rangle_{\B} = \langle \xi,  \overline{\pi} (g^{-1}) \eta \rangle_{\B}.
\]
The map $\overline{\pi} : \Gamma \rightarrow \GL (\B^*)$ is a homomorphism, so $\overline{\pi}$ is a linear representation of $\Gamma$ on $\B^*$. Moreover, for every $g \in \Gamma$,
\begin{align*}
\Vert \overline{\pi} (g) \Vert_{B (\B^{*})} &= \sup_{\xi \in \B, \eta \in \B^*, \Vert \xi \Vert_{\B} = \Vert \eta \Vert_{\B^*} =1} \langle \xi, \overline{\pi} (g) \eta \rangle_{\B} \\
&= \sup_{\xi \in \B, \eta \in \B^*, \Vert \xi \Vert_{\B} = \Vert \eta \Vert_{\B^*} =1} \langle \pi (g^{-1}) \xi, \eta \rangle_{\B} \\
&= \Vert \pi (g^{-1}) \Vert_{B (\B)}.
\end{align*}
Thus, if $\pi$ is uniformly bounded (respectively, isometric), then so is $\overline{\pi}$.

For a representation $\pi$ of $\Gamma$ on $\B$, a map $c: \Gamma \rightarrow \B$ is called \textit{a cocycle with respect to $\pi$} if, for every $g,h \in \Gamma$, $c (gh) = c (g) + \pi (g) c (h)$.  

An \textit{affine isometric action} of $\Gamma$ on $\B$ is a homomorphism $\rho : \Gamma \rightarrow \Isom (\B )$, where $\Isom (\B)$ is the group of affine isometries of $\B$.  Every such action $\rho$ has the form 
\[
\rho (g) . \xi = \pi (g) \xi + c (g),  \forall g \in \Gamma,  \forall \xi \in \B
\]
where $\pi$ is a linear isometric representation of $\Gamma$ on $\B$ and $c : \Gamma \rightarrow \B$ is a cocycle of $\Gamma$ with respect to $\pi$.

For a Banach space $\B$,  we say that $\Gamma$ has property $(F_{\B})$ if every affine isometric action of $\Gamma$ on $\B$ has a fixed point.  For a class $\mathcal{E}$ of Banach spaces,  we say that $\Gamma$ has property $(F_{\mathcal{E}} )$ if it has property $(F_{\B})$ for every $\B \in \mathcal{E}$.

\subsection{Ultraproducts of Banach spaces and group actions}

We recall that an \textit{ultrafilter} $\mathcal{U}$ on $\mathbb{N}$ is a nonempty collection of subsets of $\mathbb{N}$ satisfying the following conditions:
\begin{itemize}
\item $\emptyset \notin \mathcal{U}$. 
\item If $B \in \mathcal{U}$ and $B \subseteq A$, then $A \in \mathcal{U}$.
\item If $A, B \in \mathcal{U}$, then $A \cap B \in \mathcal{U}$.
\item If $A \notin \mathcal{U}$, then $\mathbb{N} \setminus A \in \mathcal{U}$.
\end{itemize}

An ultrafilter $\mathcal{U}$ on $\mathbb{N}$ is called \textit{nonprincipal} if, for every $n \in \mathbb{N}$,  $\lbrace n \rbrace \notin \mathcal{U}$.  The existence of nonprincipal ultrafilters follows from the axiom of choice.  In the discussion below, we fix a nonprincipal ultrafilter $\mathcal{U}$.  

Given a sequence of real numbers $\lbrace a_n \rbrace_{n \in \mathbb{N}}$,  the limit of the sequence along the ultrafilter $\mathcal{U}$ is denoted by $\lim_{n \rightarrow \mathcal{U}} a_n$ and is defined to be a number $a \in \mathbb{R}$ such that for every $\varepsilon >0$,  
\[
\lbrace n : \vert a - a_n \vert < \varepsilon \rbrace \in \mathcal{U}.
\]
One can show that if $\lbrace a_n \rbrace_{n \in \mathbb{N}}$ is a bounded sequence, then this limit exists and is unique.  

Given a sequence of Banach spaces $\lbrace \B_n \rbrace_{n \in \mathbb{N}}$,  the $\ell^{\infty}$ product $\ell^{\infty} (\mathbb{N},\allowbreak \lbrace \B_n \rbrace_{n \in \mathbb{N}})$ is the space of bounded sequences, i.e.,
\[
\ell^{\infty} (\mathbb{N}, \lbrace \B_n \rbrace_{n \in \mathbb{N}}) = \lbrace \lbrace \xi_n \in \B_n \rbrace_{n \in \mathbb{N}} : \sup_{n} \Vert \xi_n \Vert < \infty \rbrace.
\]
The \textit{ultraproduct} of the sequence $\lbrace \B_n \rbrace_{n \in \mathbb{N}}$ is defined as follows. Set
\[
N_{\mathcal{U}} = \lbrace \lbrace \xi_n \rbrace_{n \in \mathbb{N}} \in \ell^{\infty} (\mathbb{N}, \lbrace \B_n \rbrace_{n \in \mathbb{N}}) : \lim_{n \rightarrow \mathcal{U}} \Vert \xi_n \Vert = 0 \rbrace,
\]
and define the ultraproduct by 
\[
\lbrace \B_n \rbrace_{\mathcal{U}} = \ell^{\infty} (\mathbb{N}, \lbrace \B_n \rbrace_{n \in \mathbb{N}}) / N_{\mathcal{U}},
\]
with the norm
\[
\Vert \lbrace \xi_n \rbrace_{n \in \mathbb{N}} \Vert = \lim_{n \rightarrow \mathcal{U}} \Vert \xi_n \Vert .
\]

The space $\lbrace \B_n \rbrace_{\mathcal{U}}$ is a Banach space. For a class of Banach spaces $\mathcal{E}$, we say that $\mathcal{E}$ is \textit{closed under passing to ultraproducts} if, for every sequence $\lbrace \B_n \rbrace_{n \in \mathbb{N}} \subseteq \mathcal{E}$, $\lbrace \B_n \rbrace_{\mathcal{U}} \in \mathcal{E}$.  

Let $\Gamma$ be a finitely generated group and fix a finite generating set $S$.  For an affine isometric action $\rho$ of $\Gamma$ on a Banach space $\B$,  we define the displacement function $\disp_{\rho,  S} : \B \rightarrow [0, \infty)$ as 
\[
\disp_{\rho,  S} (\xi) = \max_{s \in S} \Vert \xi - \rho (s) \xi \Vert.
\]
We note that the action $\rho$ has a fixed point if and only if there is $\xi_0 \in \B$ such that $\disp_{\rho,  S} (\xi_0) =0$.  

Let $\lbrace \B_n \rbrace_{n \in \mathbb{N}}$ be a sequence of Banach spaces and let $\rho_n$ be a sequence of affine isometric actions of $\Gamma$ on $\B_n$ such that
\[
\sup_{n} \disp_{\rho_n,  S} (0_n) < \infty,
\]
where $0_n \in \B_n$ denotes the origin of $\B_n$.  In this setting,  there is an affine isometric action $\rho_{\mathcal{U}}$ of $\Gamma$ on the ultraproduct $\lbrace \B_n \rbrace_{\mathcal{U}}$ induced by the action of $\Gamma$ on $ \ell^{\infty} (\mathbb{N}, \lbrace \B_n \rbrace_{n \in \mathbb{N}})$.  

An affine isometric action $\rho$ will be called \textit{uniform with respect to $S$} (or \textit{$S$-uniform}) if
\[
\inf_{\xi \in \B} \disp_{\rho,  S} (\xi) \geq 1 .
\]

Given a class of Banach spaces $\mathcal{E}$, we define the class of uniform $S$-actions by 
\[
\mathcal{C}_{\mathcal{E}}^{S-\text{uniform}} = \lbrace (\B, \rho) : \B \in \mathcal{E},  \rho \text{ is uniform with respect to } S \rbrace.
\]

The class $\mathcal{C}_{\mathcal{E}}^{S-\text{uniform}}$ can be thought of as a class of actions that are ``uniformly far'' from having a fixed point.  For the discussion below, we need the fact that if $\mathcal{E}$ is closed under passing to ultraproducts, then so is $\mathcal{C}_{\mathcal{E}}^{S-\text{uniform}}$. More precisely, if $\mathcal{E}$ is closed under passing to ultraproducts and we are given a sequence $\lbrace (\B_n,  \rho_n) \rbrace_{n \in \mathbb{N}} \subseteq \mathcal{C}_{\mathcal{E}}^{S-\text{uniform}}$ with 
\[
\sup_{n} \disp_{\rho_n,  S} (0_n) < \infty,
\]
then $(\lbrace \B_n \rbrace_{\mathcal{U}},  \rho_{\mathcal{U}}) \in \mathcal{C}_{\mathcal{E}}^{S-\text{uniform}}$. 

The following result is proved (in much greater generality) in \cite{Stalder} (where it is attributed to Gromov):
\begin{theorem} \cite[Theorem 1.5]{Stalder}
\label{thm: S-uniform}
Let $\Gamma$ be a finitely generated group with a finite generating set $S$ and let $\mathcal{E}$ be a class of Banach spaces that is closed under passing to ultraproducts.  If $\Gamma$ does not have property $(F_{\mathcal{E}})$, then 
$\mathcal{C}_{\mathcal{E}}^{S-\text{uniform}} \neq \emptyset$.  

In other words,  if $\mathcal{E}$ is closed under passing to ultraproducts and there is an action of $\Gamma$ on some $\B \in \mathcal{E}$ that does not admit a fixed point,  then there is an $S$-uniform action of $\Gamma$ on a space in $\mathcal{E}$.
\end{theorem}

\subsection{Uniformly convex and uniformly smooth Banach spaces}

In this section, $\B$ is a real Banach space and $1 < p < \infty$.  

The space $\B$ is called \textit{strictly convex} if for any two unit vectors $\xi_1,  \xi_2 \in \B,  \xi_1 \neq \xi_2$, we have $\Vert \frac{\xi_1 + \xi_2}{2} \Vert < 1$.   The space $\B$ is called \textit{strictly smooth} if $\B^*$ is strictly convex. For a Banach space $\B$, define the \textit{modulus of convexity of $\B$} to be the function $\delta : (0,2] \rightarrow [0,1]$ defined as 
\[
\delta_{\B} (\varepsilon) = \inf \left\lbrace 1 - \left\Vert \frac{\xi_1 + \xi_2}{2} \right\Vert : \Vert \xi_1 \Vert = \Vert \xi_2 \Vert =1, \Vert \xi_1 - \xi_2 \Vert \geq \varepsilon \right\rbrace.
\]
The space $\B$ is called \textit{uniformly convex} if $\delta (\varepsilon) >0$ for every $0 < \varepsilon \leq 2$. The space $\B$ is called \textit{uniformly smooth} if $\B^*$ is uniformly convex.


For $\xi \in \B$,  we define $J_{\B, p} \xi$ to be the following subset of $\B^*$:
\[
J_{\B, p} \xi = \lbrace \eta \in \B^* : \Vert \eta \Vert_{\B^*} = \Vert \xi \Vert_{\B}^{p-1} \text{ and } \langle \xi,  \eta \rangle_{\B} = \Vert \xi \Vert^{p}_{\B} \rbrace.
\]
By the Hahn-Banach theorem,  $J_{\B, p} \xi$ is always nonempty.  When $\B$ is clear from the context,  we will sometimes write $J_p$ instead of $J_{\B, p}$. When $\B$ is strictly smooth, $J_{\B,p} \xi$ is a singleton for every $\xi \in \B$.  In that case,  we abuse notation and treat $J_{\B, p}$ as a function,  i.e.,  when $\B$ is strictly smooth,  we also use $J_{\B, p} : \B \rightarrow \B^*$ to denote the function mapping each $\xi \in \B$ to the unique $\eta \in \B^{*}$ such that $\Vert \eta \Vert_{\B^*} = \Vert \xi \Vert^{p-1}_{\B}$ and $\langle \xi,  \eta \rangle_{\B} = \Vert \xi \Vert^{p}_{\B}$.  

We recall that $q$ is called the \textit{H\"{o}lder conjugate of $p$} if $\frac{1}{p} + \frac{1}{q} =1$.

For a Banach space $\B$ and $2 \leq p < \infty, 0 < C \leq 1$, we say that $\B$ is \textit{$(p,C)$-uniformly convex} if $\delta_{\B} (\varepsilon) \geq C \varepsilon^p, \forall 0 < \varepsilon \leq 2$ and that $\B$ is $p$-uniformly convex if it is $(p,C)$-uniformly convex for some $C$. 

For a Banach space $\B$, we say that $\B$ satisfies the $p$-Clarkson inequality if for any two unit vectors $\xi_1, \xi_2 \in \B$, we have 
\[
\left\Vert \frac{\xi_1 + \xi_2}{2} \right\Vert^p + \left\Vert \frac{\xi_1 - \xi_2}{2} \right\Vert^p \leq 1.
\]

\begin{proposition}
\label{prop: p-C-uni conv for Hilbert and Lp}
For $2 \leq p < \infty$, if a Banach space $\B$ satisfies the $p$-Clarkson inequality, then it is $(p, \frac{1}{p 2^p})$-uniformly convex.

In particular, every Hilbert space is $(2, \frac{1}{8})$-uniformly convex and for every $2 \leq p < \infty$, every $L^p$-space is $(p, \frac{1}{p 2^p})$-uniformly convex.
\end{proposition} 

\begin{proof}
Assume that $\B$ satisfies the $p$-Clarkson inequality and let $\xi_1, \xi_2 \in \B, 0 < \varepsilon \leq 2$ with $\Vert \xi_1 \Vert = \Vert \xi_2 \Vert =1$ and $\Vert \xi_1 - \xi_2 \Vert \geq \varepsilon$. Then 
\begin{align*}
\left\Vert \frac{\xi_1 + \xi_2}{2} \right\Vert &\leq \left(1 - \left\Vert \frac{\xi_1 - \xi_2}{2} \right\Vert^p \right)^{\frac{1}{p}} \\
&\leq \left(1 - \frac{\varepsilon^p}{2^p} \right)^{\frac{1}{p}} \\
&\leq 1 - \frac{\varepsilon^p}{p 2^p}.
\end{align*}
It follows that $\delta_{\B} (\varepsilon) \geq \frac{\varepsilon^p}{p 2^p}$ for every $0<\varepsilon\leq2$.

We note that Hilbert spaces satisfy the $2$-Clarkson inequality by the parallelogram identity and that for every $2 \leq p < \infty$, every $L^p$ space satisfies the $p$-Clarkson inequality.
\end{proof}

\begin{theorem}\cite[Corollary 1]{Xu}
\label{thm: J-p inequality}
Let $2 \leq p < \infty$ and let $\B$ be a $(p,C)$-uniformly convex and strictly smooth real Banach space.  Then for every $\xi_1,  \xi_2 \in \B$, 
\[
\langle \xi_1 - \xi_2,  J_{\B,  p} \xi_1 - J_{\B,  p} \xi_2 \rangle_{\B} \geq \frac{2C}{p} \Vert \xi_1 - \xi_2 \Vert^p.
\]
\end{theorem}

\subsection{Vector-valued finite measure spaces}
\label{subsection: Vector valued finite measure spaces}

Let $V$ be a finite set and $m : V \rightarrow \mathbb{R}_{>0}$ be a function.  Given a Banach space $\B$ and $1 \leq p < \infty$,  the space $\ell^p (V, m ; \B)$ is the space of functions $\phi : V \rightarrow \B$ with the norm
\[
\Vert \phi \Vert^{p} = \sum_{v \in V} m(v) \Vert \phi (v) \Vert_{\B}^p.
\]
We also define $\ell^{\infty} (V,  m ; \B)$ with the norm 
\[
\Vert \phi \Vert_{\ell^{\infty} (V,  m ; \B)} = \max_{v \in V} \Vert \phi (v) \Vert_{\B}.
\]

For $1 < p < \infty$,  $q$ is called the \textit{H\"{o}lder conjugate of $p$} if $\frac{1}{p} + \frac{1}{q} =1$.  Let $\langle .,. \rangle_{\B} : \B \times \B^* \rightarrow \mathbb{R}$ denote the usual pairing between $\B$ and $\B^*$.  Define $\langle .,. \rangle_{ \ell^p (V, m ; \B)} :  \ell^p (V, m ; \B) \times  \ell^q (V, m ; \B^*)$ by 
\[
\langle \phi, \psi \rangle = \sum_{v \in V} m (v) \langle \phi (v) , \psi (v) \rangle_{\B}.
\]

Let $1 < p < \infty$ and let $q$ be the H\"{o}lder conjugate of $p$. Then, for any Banach space $\B$, $(\ell^p (V,  m ; \B))^{*}$ is isometrically isomorphic to $\ell^q (V, m ; \B^*)$.

\begin{proposition}
\label{prop: J ell-p}
Let $1 < p < \infty$, let $q$ be the H\"{o}lder conjugate of $p$, and let $\B$ be a real Banach space.  For $\phi \in \ell^p (V, m ; \B)$, let $\phi ' \in \ell^q (V, m ; \B^*)$ be such that for every $v \in V$,  $\phi ' (v) \in J_{\B, p} \phi (v)$.  Then $\phi ' \in J_{\ell^p (V, m ; \B),  p} \phi$. Moreover,  if $\B$ is strictly smooth, then $\ell^q (V, m ; \B^*)$ is strictly convex. In this case, for every $\phi \in \ell^p (V, m ; \B)$, 
\[
\forall v \in V,  (J_{\ell^p (V, m ; \B), p} \phi) (v) = J_{\B, p} (\phi (v)) .
\]
\end{proposition}

\begin{proof}
Let $\phi \in \ell^p (V, m ; \B)$ and let $\phi ' \in \ell^q (V, m ; \B^*)$ be such that for every $v \in V$,  $\phi ' (v) \in J_{\B, p} \phi (v)$. We note that $\frac{1}{q} = \frac{p-1}{p}$ and that $q (p-1) =p$. 
Then
\begin{align*}
\Vert \phi ' \Vert_{\ell^q (V, m ; \B^*)} &= \left( \sum_{v \in V} m(v) \Vert \phi ' (v) \Vert^q_{\B^*} \right)^{\frac{1}{q}} \\
&=^{\phi ' (v) \in J_{\B, p} \phi (v)} \left( \sum_{v \in V} m(v) \Vert \phi (v) \Vert^{q (p-1)}_{\B} \right)^{\frac{1}{q}} \\
&= \left(\left( \sum_{v \in V} m(v) \Vert \phi (v) \Vert^{p}_{\B} \right)^{\frac{1}{p}} \right)^{p-1} \\
&= \Vert \phi \Vert_{\ell^p (V, m ; \B)}^{p-1}.
\end{align*}
Also,
\begin{align*}
\langle \phi, \phi ' \rangle_{\ell^p (V, m ; \B)} &= \sum_{v \in V} m (v) \langle \phi (v), \phi ' (v) \rangle_{\B} \\
&= \sum_{v \in V} m (v) \Vert \phi (v) \Vert_{\B}^p \\
&= \Vert \phi \Vert_{\ell^p (V, m ; \B)}^p.
\end{align*}
For the second assertion, the strict convexity of $\ell^q (V, m ; \B^*)$ readily follows from the strict convexity of $\B^*$ and the strict convexity of $\ell^q (V,m)$.
\end{proof}

When $\B = \mathcal{H}$ is a real Hilbert space and $p=2$,  the space $\ell^2 (V, m ; \mathcal{H}) (= \ell^2 (V, m ; \mathcal{H}^*))$ is a real Hilbert space with the inner product
\[
\langle \phi, \psi \rangle = \sum_{v \in V} m (v) \langle \phi (v) , \psi (v) \rangle_{\mathcal{H}},
\]
and the map $J_{\ell^2 (V, m ; \mathcal{H}), 2}$ is the identity map.

When $\B = \mathbb{R}$, we write $\ell^p (V, m ; \mathbb{R}) = \ell^p (V,m)$.  Also,  when $m$ is clear from the context, we omit it and write $\ell^p (V ; \B)$ or $\ell^p (V)$ in the case $\B = \mathbb{R}$.

Every linear operator $T : \ell^p (V, m) \rightarrow \ell^p (V,  m)$ can be extended to $T \otimes I_{\B} : \ell^p (V, m ; \B) \rightarrow \ell^p (V, m ; \B)$.  Explicitly,  if $(T(v,u))_{v,u \in V}$ is the matrix of $T$,  then 
\[
(T \otimes I_{\B}) \phi (v) = \sum_{u \in V}  T (v,u) \phi (u),  \forall \phi \in \ell^p (V, m ; \B).
\]
We will say that $T$ is \text{a Markov operator preserving $m$}  if the following conditions hold:
\begin{itemize}
\item For every $v,u \in V$, $T (v,u) \geq 0$. 
\item For every $v \in V$,  $\sum_{u \in V} T (v,u) =1$.
\item For every $u \in V$, $\sum_{v \in V} m (v) T(v,u) = m (u)$. 
\end{itemize} 

\begin{proposition}
\label{prop: Markov operator norm bound}
Let $V, m$ be as above. For every Markov operator $T$ preserving $m$, every Banach space $\B$, and every $1 \leq p \leq \infty$,
\[
\Vert T \otimes I_{\B} \Vert_{B(\ell^p (V, m ; \B))} \leq 1.
\]
\end{proposition}

\begin{proof}
Assume first that $1 \leq p < \infty$. Then for every $\phi \in \ell^p (V, m ; \B)$, 
\begin{align*}
\Vert (T \otimes I_{\B}) \phi \Vert_{\ell^p (V, m ; \B)}^p &= \sum_{v \in V} m (v) \Vert \sum_{u \in V} T(v,u) \phi (u) \Vert^p \\
&\leq \sum_{v \in V} m (v) \sum_{u \in V} T(v,u) \Vert \phi (u) \Vert^p \\
&= \sum_{u \in V} \Vert \phi (u) \Vert^p \sum_{v \in V} m (v) T(v,u) \\
&= \sum_{u \in V} m (u) \Vert \phi (u) \Vert^p \\
&= \Vert \phi \Vert_{\ell^p (V, m ; \B)}^p.
\end{align*}

For $p = \infty$ and every $\phi \in \ell^\infty (V, m ; \B)$, we have
\begin{align*}
\Vert (T \otimes I_{\B}) \phi \Vert_{\ell^\infty (V, m ; \B)} &= \max_{v \in V} \Vert \sum_{u \in V} T(v,u) \phi (u) \Vert \\
&\leq \max_{v \in V} \sum_{u \in V} T(v,u) \Vert \phi (u) \Vert \\
&\leq \max_{v \in V} \sum_{u \in V} T(v,u) \left( \max_{u \in V} \Vert \phi (u) \Vert \right) \\
&= \max_{u \in V} \Vert \phi (u) \Vert \\
&= \Vert \phi \Vert_{\ell^\infty (V, m ; \B)}.
\end{align*}

\end{proof}

\begin{proposition}
\label{prop: ell-p real and ell-p complex norm equality}
Let $V,m$ be as above. For every $1 \leq p < \infty$ and every linear operator $T : \ell^p (V, m) \rightarrow \ell^p (V, m)$, we have 
\[
\Vert T \otimes I_{\mathbb{C}} \Vert_{B (\ell^p (V,m ; \mathbb{C}))} = \Vert T \Vert_{B (\ell^p (V,m))}.
\]
\end{proposition}

\begin{proof}
The inequality \[
\Vert T \otimes I_{\mathbb{C}} \Vert_{B (\ell^p (V,m ; \mathbb{C}))} \geq \Vert T \Vert_{B (\ell^p (V,m))}
\]
is immediate, so it remains to prove the reverse inequality.

Set 
\[
L_p = \frac{1}{2 \pi} \int_{0}^{2 \pi} \vert \cos (t) \vert^p dt
\]
and note that, for every $z \in \mathbb{C}, z=r e^{i \theta}$,
\[
\frac{1}{2 \pi} \int_{0}^{2 \pi} \vert \Re (e^{it} z)  \vert^p dt = \frac{1}{2 \pi} \int_{0}^{2 \pi} r^p \vert \cos (t+ \theta ) \vert^p dt = L_p r^p  = L_p \vert z \vert^p.
\] 
Thus, for every $\phi \in \ell^p (V,m ; \mathbb{C} )$,
\[
L_p \Vert \phi \Vert_{\ell^p (V,m ; \mathbb{C} )}^p = \frac{1}{2 \pi} \int_0^{2 \pi} \Vert \Re (e^{it} \phi) \Vert_{\ell^p (V,m )}^p dt.
\]

Consequently, for every $\phi \in \ell^p (V,m ; \mathbb{C} )$,
\begin{align*}
& \Vert (T \otimes I_{\mathbb{C}}) \phi \Vert^p_{\ell^p (V,m ; \mathbb{C})} \\
&\qquad = \frac{1}{L_p} \frac{1}{2 \pi} \int_0^{2 \pi} \Vert \Re (e^{it} ((T \otimes I_{\mathbb{C}}) \phi)) \Vert_{\ell^p (V,m )}^p dt \\
&\qquad = \frac{1}{L_p} \frac{1}{2 \pi} \int_0^{2 \pi} \Vert \Re ((T \otimes I_{\mathbb{C}}) (e^{it} \phi)) \Vert_{\ell^p (V,m )}^p dt \\
&\qquad =^{T_{v,u} \in \mathbb{R}, \forall v,u} \frac{1}{L_p} \frac{1}{2 \pi} \int_0^{2 \pi} \Vert T (\Re ( e^{it} \phi)) \Vert_{\ell^p (V,m )}^p dt \\
&\qquad \leq \Vert T \Vert_{B (\ell^p (V,m))}^p \left(\frac{1}{L_p} \frac{1}{2 \pi} \int_0^{2 \pi} \Vert \Re ( e^{it} \phi) \Vert_{\ell^p (V,m )}^p dt \right) \\
&\qquad = \Vert T \Vert_{B (\ell^p (V,m))}^p \Vert \phi \Vert_{\ell^p (V,m ; \mathbb{C} )}^p
\end{align*} 
as needed.
\end{proof}

\begin{proposition}
\label{prop: norm of T tensor Hilbert or Lp}
Let $V,m$ be as above. For every linear operator $T : \ell^2 (V, m) \rightarrow \ell^2 (V, m)$ and every Hilbert space $\mathcal{H}$ (real or complex), the following inequality holds: 
\[
\Vert T \otimes I_{\mathcal{H}} \Vert_{B (\ell^2 (V,m ; \mathcal{H}))} \leq \Vert T \Vert_{B (\ell^2 (V,m))}.
\]
Also, for every $1 \leq p < \infty$ and every space $L^p (\Omega, \mu)$ (real or complex), we have
\[
\Vert T \otimes I_{L^p (\Omega, \mu)} \Vert_{B (\ell^p (V,m ; L^p (\Omega, \mu)))} \leq \Vert T \Vert_{B (\ell^p (V,m))}.
\]
\end{proposition}

\begin{proof}
Let $\mathbb{F} = \mathbb{R} \text{ or } \mathbb{C}$. Let $\mathcal{H}$ be a Hilbert space over $\mathbb{F}$. We have \[
\ell^2 (V,m ; \mathcal{H}) = \ell^2 (V,m ; \mathbb{F}) \otimes_{\mathbb{F}} \mathcal{H}
\] and thus 
\[
\Vert T \otimes I_{\mathcal{H}} \Vert_{B (\ell^2 (V,m ; \mathcal{H}))} \leq \Vert T \Vert_{B (\ell^2 (V,m ; \mathbb{F}))} \Vert I \Vert_{B (\mathcal{H})} \leq^{\Cref{prop: ell-p real and ell-p complex norm equality}} \Vert T \Vert_{B (\ell^2 (V,m))} .
\]

Similarly, let $L^p (\Omega, \mu)$ be an $L^p$ space over $\mathbb{F}$. For $\phi \in \ell^p (V, m ;L^p (\Omega, \mu))$, $v \in V$, and $\omega \in \Omega$, we write $\phi (v, \omega) = (\phi (v)) (\omega)$. We also define, for every $\omega \in \Omega$, $\Phi_{\omega} \in \ell^p (V, m ; \mathbb{F})$ as $\Phi_{\omega} (v) = \phi (v, \omega)$. With this notation, for every $\phi \in \ell^p (V, m ;L^p (\Omega, \mu))$, we have
\begin{align*}
& \Vert (T \otimes I_{L^p (\Omega, \mu)}) \phi \Vert_{\ell^p (V,m ; L^p (\Omega, \mu))}^p \\
&\qquad = \sum_{v \in V} m(v) \int_{\Omega} \vert (T \phi) (v,\omega) \vert^p d \mu \\
&\qquad = \int_{\Omega} \left( \sum_{v \in V} m (v) \vert (T \Phi_{\omega}) (v) \vert^p \right) d \mu \\
&\qquad \leq \int_{\Omega} \Vert T \Vert_{B (\ell^p (V,m; \mathbb{F}))}^p \Vert \Phi_{\omega} \Vert^p_{\ell^p (V, m ; \mathbb{F})} d \mu \\
&\qquad = \Vert T \Vert_{B (\ell^p (V,m; \mathbb{F}))}^p \sum_{v \in V} m (v) \int_{\Omega} \vert \phi (v, \omega) \vert^p d \mu \\
&\qquad = \Vert T \Vert_{B (\ell^p (V,m; \mathbb{F}))}^p \Vert \phi \Vert^p_{\ell^p (V,m ; L^p (\Omega, \mu))} \\
&\qquad \leq^{\Cref{prop: ell-p real and ell-p complex norm equality}} \Vert T \Vert_{B (\ell^p (V,m))}^p \Vert \phi \Vert^p_{\ell^p (V,m ; L^p (\Omega, \mu))}
\end{align*}
as needed.
\end{proof}

The following projections will play a significant role in the sequel. Let $\lbrace U_1,...,U_k \rbrace$ be a partition of $V$,  i.e., $U_1,...,U_k$ are nonempty sets such that $V = U_1 \dot\cup ...  \dot\cup U_k$.  We define $P_{\lbrace U_1,...,U_k \rbrace , m} : \lbrace \phi : V \rightarrow \mathbb{R} \rbrace \rightarrow \lbrace \phi : V \rightarrow \mathbb{R} \rbrace$ to be the orthogonal projection onto $span \lbrace \mathbbm{1}_{U_1} ,..., \mathbbm{1}_{U_k} \rbrace$ with respect to the inner product of $\ell^2 (V,  m)$.  Explicitly, for a nonempty set $U \subseteq V$, we write $ m (U) = \sum_{v \in U} m(v)$, and with this notation, for every Banach space $\B$ and every function $\phi : V \rightarrow \B$, we have 
\[
((P_{\lbrace U_1,...,U_k \rbrace , m} \otimes I_{\B}) \phi ) (v) = \frac{1}{m (U_i)} \sum_{u \in U_i} m(u) \phi (u),  \forall 1 \leq i \leq k,  \forall v \in U_i.
\]

\begin{proposition}
\label{prop: proj in different m's comparison}
Let $V$ be a finite set and $m_1,  m_2 : V \rightarrow \mathbb{R}_{>0}$ be two functions.  Also let $\lbrace U_1,...,U_k \rbrace$ be a partition of $V$.  Assume that there is $\epsilon \geq 0$ such that for every $1 \leq i \leq k$ and every $v \in U_i$, 
\[
\left\vert \frac{m_2 (v) m_1 (U_i)}{m_1 (v) m_2 (U_i)} - 1 \right\vert \leq \epsilon .
\]
Then, for every Banach space $\B$ and every $1 \leq p < \infty$, 
\[
\Vert P_{\lbrace U_1,...,U_k \rbrace , m_1} \otimes I_{\B} - P_{\lbrace U_1,...,U_k \rbrace,  m_2} \otimes I_{\B} \Vert_{B (\ell^p (V , m_1 ; \B))} \leq \epsilon.
\]
\end{proposition}

\begin{proof}
Fix $\B$ and $1 \leq p < \infty$. For every $\phi \in \ell^p (V , m_1 ; \B)$, we have
\begin{align*}
& \Vert (P_{\lbrace U_1,...,U_k \rbrace, m_1} \otimes I_{\B}) \phi - (P_{\lbrace U_1,...,U_k \rbrace, m_2} \otimes I_{\B}) \phi \Vert_{ \ell^p (V , m_1 ; \B)}^{p} \\
&\qquad = \sum_{i=1}^{k} m_1 (U_i) \left\Vert \frac{1}{m_1 (U_i)} \sum_{u \in U_i} m_1 (u) \phi (u) - \frac{1}{m_2 (U_i)} \sum_{u \in U_i} m_2 (u) \phi (u) \right\Vert_{\B}^p \\
&\qquad = \sum_{i=1}^{k} m_1 (U_i) \left\Vert \sum_{u \in U_i} \frac{m_1 (u)}{m_1 (U_i)} \left( 1- \frac{m_2 (u) m_1 (U_i)}{m_1 (u) m_2 (U_i)} \right) \phi (u) \right\Vert_{\B}^p \\
&\qquad \leq^{\text{convexity of } \Vert . \Vert^p_{\B}} \sum_{i=1}^{k} \sum_{u \in U_i} m_1 (u) \left\vert 1 - \frac{m_2 (u) m_1 (U_i)}{m_1 (u) m_2 (U_i)} \right\vert^p \Vert \phi (u) \Vert_{\B}^p \\
&\qquad \leq \sum_{u \in V} \epsilon^p m_1 (u) \Vert \phi (u) \Vert_{\B}^p \\
&\qquad = \epsilon^p \Vert \phi \Vert_{ \ell^p (V , m_1 ; \B)}^{p}.
\end{align*}
Thus, for every $\phi \in \ell^p (V , m_1 ; \B)$, 
\[
\Vert (P_{U_1,...,U_k ; m_1} \otimes I_{\B}) \phi - (P_{U_1,...,U_k ; m_2} \otimes I_{\B}) \phi \Vert_{ \ell^p (V , m_1 ; \B)} \leq \epsilon \Vert \phi \Vert_{ \ell^p (V , m_1 ; \B)}
\]
as needed.
\end{proof}

\begin{corollary}
\label{corollary: proj in different m's comparison}
Let $V$ be a finite set and $m_1,  m_2 : V \rightarrow \mathbb{R}_{>0}$ be two functions.  Assume that there are constants $0 \leq \epsilon_1,  \epsilon_2 <1$ such that for every $v \in V$, 
\[
(1- \epsilon_1) m_1 (v) \leq m_2 (v) \leq (1+\epsilon_2) m_1 (v).
\]
Then, for every Banach space $\B$, every $1 \leq p < \infty$, and every partition $\lbrace U_1,...,U_k \rbrace$ of $V$, 
\[
\Vert P_{\lbrace U_1,...,U_k \rbrace , m_1} \otimes I_{\B} - P_{\lbrace U_1,...,U_k \rbrace,  m_2} \otimes I_{\B} \Vert_{B (\ell^p (V , m_1 ; \B))} \leq \frac{\epsilon_1 + \epsilon_2}{1- \epsilon_1}
\]
and
\[
\Vert P_{\lbrace U_1,...,U_k \rbrace , m_1} \otimes I_{\B} - P_{\lbrace U_1,...,U_k \rbrace,  m_2} \otimes I_{\B} \Vert_{B (\ell^p (V , m_2 ; \B))} \leq \frac{\epsilon_1 + \epsilon_2}{1- \epsilon_1} .
\]
\end{corollary}

\begin{proof}
For the first inequality,  we note that for every $v \in V$, 
\[
1- \epsilon_1 \leq \frac{m_2 (v)}{m_1 (v)} \leq 1+\epsilon_2.
\]
Also,  for every nonempty set $U \subseteq V$, 
\[
\frac{1}{1 + \epsilon_2} \leq \frac{m_1 (U)}{m_2 (U)} \leq \frac{1}{1-\epsilon_1}.
\]
Thus,  for every $v$ and every nonempty set $U \subseteq V$,
\[
\frac{1- \epsilon_1}{1+ \epsilon_2} \leq \frac{m_2 (v) m_1 (U)}{m_1 (v) m_2 (U)} \leq \frac{1+\epsilon_2}{1 -\epsilon_1}.
\]
Therefore, for every $v$ and every nonempty set $U \subseteq V$,
\[
\left\vert \frac{m_2 (v) m_1 (U)}{m_1 (v) m_2 (U)} - 1 \right\vert \leq \frac{\epsilon_1 + \epsilon_2}{1- \epsilon_1}
\]
and the first inequality follows from \Cref{prop: proj in different m's comparison}.

Similarly,  for the second inequality,  we note that for every $v \in V$,
\[
\frac{1}{1 + \epsilon_2} \leq \frac{m_1 (v)}{m_2 (v)} \leq \frac{1}{1-\epsilon_1}
\]
and that for every nonempty set $U \subseteq V$, 
\[
1- \epsilon_1 \leq \frac{m_2 (U)}{m_1 (U)} \leq 1+\epsilon_2.
\]
Thus,  for every $v$ and every nonempty set $U \subseteq V$,
\[
\frac{1-\epsilon_1}{1+\epsilon_2} \leq   \frac{m_1 (v) m_2 (U)}{m_2 (v) m_1 (U)} \leq \frac{1 + \epsilon_2}{1-\epsilon_1}
\]
and the proof is completed as in the previous case.
\end{proof}

\begin{proposition}
\label{prop: coupling and norm under different m's}
Let $V$ be a finite set and $m_1,  m_2 : V \rightarrow \mathbb{R}_{>0}$ be two functions.  Also let $1 < p < \infty$, let $q$ be the H\"{o}lder conjugate of $p$, and let $\B$ be a real Banach space.  Assume that there is $\epsilon \geq 0$ such that for every $v \in V$, $\vert \frac{m_2 (v)}{m_1 (v)} - 1 \vert \leq \epsilon$.  Then the following inequalities hold: 
\begin{enumerate}
\item For every $\phi : V  \rightarrow \B$ and every $\phi ' : V \rightarrow \B^*$, 
\[
\left\vert \langle \phi , \phi ' \rangle_{\ell^p (V, m_1 ; \B)} - \langle \phi , \phi ' \rangle_{\ell^p (V, m_2 ; \B)}  \right\vert \leq \epsilon \Vert \phi \Vert_{\ell^p (V, m_1 ; \B)}  \Vert \phi ' \Vert_{\ell^q (V, m_1 ; \B^*)}.
\]
\item For every $\phi : V  \rightarrow \B$, 
\[
\left\vert \Vert \phi \Vert_{ \ell^p (V, m_1 ; \B)}^p - \Vert \phi \Vert_{ \ell^p (V, m_2 ; \B)}^p \right\vert \leq \epsilon \Vert \phi \Vert_{\ell^p (V, m_1 ; \B)}^p.
\]
\end{enumerate}
\end{proposition}

\begin{proof}
Fix $\phi$ and $\phi '$ as above.  Then
\begin{align*}
& \left\vert \langle \phi , \phi ' \rangle_{\ell^p (V, m_1 ; \B)} - \langle \phi , \phi ' \rangle_{\ell^p (V, m_2 ; \B)} \right\vert \\
&\qquad = \left\vert \sum_{v \in V} m_1 (v) \langle \phi (v), \phi ' (v) \rangle_{\B} - \sum_{v \in V} m_2 (v) \langle \phi (v), \phi ' (v) \rangle_{\B} \right\vert \\
&\qquad \leq \sum_{v \in V} m_1 (v) \left\vert 1- \frac{m_2 (v)}{m_1 (v)} \right\vert \left\vert \langle \phi (v), \phi ' (v) \rangle_{\B} \right\vert \\
&\qquad \leq \epsilon \sum_{v \in V} m_1 (v) \Vert \phi (v) \Vert_{\B} \Vert \phi ' (v) \Vert_{\B^*} \\
&\qquad \leq^{\text{H\"{o}lder inequality}} \epsilon \Vert \phi \Vert_{\ell^p (V, m_1 ; \B)} \Vert \phi ' \Vert_{\ell^q (V, m_1 ; \B^*)}.
\end{align*}

The second assertion follows from the first one by choosing \[
\phi ' \in J_{\ell^p (V, m_1 ; \B), p} \phi.
\]
\end{proof}

\begin{corollary}
\label{corollary: coupling and norm under different m's}
Let $V$ be a finite set and $m_1,  m_2 : V \rightarrow \mathbb{R}_{>0}$ be two functions.  Also let $1 < p < \infty$, let $q$ be the H\"{o}lder conjugate of $p$, and let $\B$ be a real Banach space.  Assume that there are constants $0 \leq \epsilon_1,  \epsilon_2 <1$ such that for every $v \in V$, 
\[
(1- \epsilon_1) m_1 (v) \leq m_2 (v) \leq (1+\epsilon_2) m_1 (v).
\]
Then
\begin{enumerate}
\item For every $\phi : V \rightarrow \B$ and $\phi ' : V \rightarrow \B^*$, 
\begin{align*}
& \left\vert \langle \phi, \phi ' \rangle_{\ell^p (V, m_1 ; \B)} - \langle \phi, \phi ' \rangle_{\ell^p (V, m_{2} ; \B)} \right\vert \\
&\qquad \leq \max \lbrace \epsilon_1, \epsilon_2 \rbrace \Vert \phi \Vert_{\ell^p (V, m_1 ; \B)} \Vert \phi ' \Vert_{\ell^q (V, m_1 ; \B^*)},
\end{align*}
\begin{align*}
& \left\vert \langle \phi, \phi ' \rangle_{\ell^p (V, m_1 ; \B)} - \langle \phi, \phi ' \rangle_{\ell^p (V, m_{2} ; \B)} \right\vert \\
&\qquad \leq \max \lbrace \frac{\epsilon_1}{1- \epsilon_1}, \frac{\epsilon_2}{1 + \epsilon_2} \rbrace \Vert \phi \Vert_{\ell^p (V, m_{2} ; \B)} \Vert \phi ' \Vert_{\ell^q (V, m_{2} ; \B^*)}.
\end{align*}
\item For every $\phi : V \rightarrow \B$, 
\[
\left\vert \Vert \phi \Vert_{\ell^p (V, m_1 ; \B)}^p - \Vert \phi \Vert_{\ell^p (V, m_{2} ; \B)}^p \right\vert \leq \max \lbrace \epsilon_1,  \epsilon_2 \rbrace \Vert \phi \Vert_{\ell^p (V, m_1; \B)}^p,
\]
\[
\left\vert \Vert \phi \Vert_{\ell^p (V, m_1 ; \B)}^p - \Vert \phi \Vert_{\ell^p (V, m_{2} ; \B)}^p \right\vert \leq  \max \lbrace \frac{\epsilon_1}{1- \epsilon_1},  \frac{\epsilon_2}{1 + \epsilon_2} \rbrace \Vert \phi \Vert_{\ell^p (V, m_{2}; \B)}^p.
\]
\end{enumerate}
\end{corollary}

\begin{proof}
We note that for every $v \in V$,  
\[
\left\vert \frac{m_2 (v)}{m_1 (v)} -1 \right\vert \leq \max \lbrace \epsilon_1,  \epsilon_2 \rbrace
\]
and
\[
\left\vert \frac{m_1 (v)}{m_2 (v)} -1 \right\vert \leq \max \lbrace \frac{\epsilon_1}{1- \epsilon_1},  \frac{\epsilon_2}{1 + \epsilon_2} \rbrace.
\]
Thus all the inequalities follow from \Cref{prop: coupling and norm under different m's}.
\end{proof}

\subsection{Interpolation and complexification}

Two complex Banach spaces $\B_0, \B_1$ form a \textit{compatible pair} $(\B_0,\B_1)$ if they are continuously and linearly embedded in the same topological vector space. The idea of complex interpolation is that, given a compatible pair $(\B_0,\B_1)$ and a constant $0 \leq \theta \leq 1$, one can construct a new Banach space $\B_\theta = [\B_0, \B_1]_\theta$ as a ``convex combination'' of $\B_0$ and $\B_1$. We will not review this method here, and the interested reader can find more information on interpolation in \cite{InterpolationSpaces}.

\begin{theorem}\cite[Theorem 4.1.2, Theorem 5.1.2]{InterpolationSpaces}
\label{thm: interpolation}
Let $\B_0, \B_1$ be two complex Banach spaces that form a compatible pair. Also let $V$ be a finite set and let $m : V \rightarrow \mathbb{R}_{>0}$ be a function.  
For every $0 < \theta < 1, 1 \leq p_0, p_1 < \infty$ and every linear operator $T: \ell^2 (V,m) \rightarrow \ell^2 (V,m)$,
\[
\Vert T \otimes I_{\B_{\theta}} \Vert_{B (\ell^{p} (V, m ; \B_{\theta}))} \leq \Vert T \otimes I_{\B_{0}} \Vert_{B (\ell^{p_0} (V, m ; \B_{0}))}^{1- \theta} \Vert T \otimes I_{\B_{1}} \Vert_{B (\ell^{p_1} (V, m ; \B_{1}))}^{ \theta},
\]
with $\frac{1}{p} = \frac{1-\theta}{p_0} + \frac{\theta}{p_1}$.

Also, for every $0 < \theta < 1, 1 \leq p_1 < \infty$ and every linear operator $T: \ell^2 (V,m) \rightarrow \ell^2 (V,m)$,
\[
\Vert T \otimes I_{\B_{\theta}} \Vert_{B (\ell^{p} (V, m ; \B_{\theta}))} \leq \Vert T \otimes I_{\B_{0}} \Vert_{B (\ell^{\infty} (V, m ; \B_{0}))}^{1- \theta} \Vert T \otimes I_{\B_{1}} \Vert_{B (\ell^{p_1} (V, m ; \B_{1}))}^{ \theta},
\]
with $\frac{1}{p} = \frac{\theta}{p_1}$.
\end{theorem}

To apply complex interpolation to the real Banach spaces considered below, we use complexification. Given a real Banach space $\B$, we define its complexification $\B_{\mathbb{C}}$ as follows. As a vector space, $\B_{\mathbb{C}}$ consists of pairs $\xi_1 + i\xi_2$, where $\xi_1, \xi_2 \in \B$ with scalar multiplication 
\[
(\alpha + i \beta) (\xi_1 + i \xi_2) = (\alpha \xi_1 - \beta \xi_2) + i (\alpha \xi_2 + \beta \xi_1), \forall \alpha, \beta \in \mathbb{R}
\] 
and componentwise addition. There are several ways to define the norm on $\B_{\mathbb{C}}$ (which lead to equivalent but not identical norms). We define the norm on $\B_{\mathbb{C}}$ as 
\[
\Vert \xi_1 + i\xi_2 \Vert_{\B_{\mathbb{C}}}^2 = \frac{2}{\pi} \int_{0}^{\pi} \Vert \xi_1 \cos (t) - \xi_2 \sin (t) \Vert_{\B}^2 d t.
\] 
\begin{proposition}
\label{prop: l-2 complexification operator norm equality}
Let $V$ be a finite set, let $m: V \rightarrow \mathbb{R}_{>0}$ be a function, and let $T: \ell^2 (V,m) \rightarrow \ell^2 (V,m)$ be a linear operator. For every real Banach space $\B$, 
\[
\Vert T \otimes I_{\B_{\mathbb{C}}} \Vert_{B(\ell^2 (V,m ; \B_{\mathbb{C}}))} = \Vert T \otimes I_{\B} \Vert_{B(\ell^2 (V,m ; \B))}.
\]
\end{proposition}

\begin{proof}
We note that $\ell^2 (V,m ; \B)$ is a subspace of $\ell^2 (V,m ; \B_{\mathbb{C}})$ and thus 
\[
\Vert T \otimes I_{\B_{\mathbb{C}}} \Vert_{B(\ell^2 (V,m ; \B_{\mathbb{C}}))} \geq \Vert T \otimes I_{\B} \Vert_{B(\ell^2 (V,m ; \B))}.
\]
For the reverse inequality, let $\phi \in \ell^2 (V,m ; \B_{\mathbb{C}})$. Then there are $\phi_1, \phi_2 \in \ell^2 (V,m ; \B)$ such that $\phi = \phi_1 + i \phi_2$. Thus
\begin{align*}
& \Vert (T \otimes I_{\B_{\mathbb{C}}})\phi \Vert_{\ell^2 (V,m ; \B_{\mathbb{C}})}^2 \\
&= \sum_{v \in V} m (v) \left( \frac{2}{\pi} \int_{0}^{\pi} \Vert ((T \otimes I_{\B}) \phi_1) (v) \cos (t) - ((T \otimes I_{\B}) \phi_2) (v)\sin (t) \Vert_{\B}^2 d t \right) \\
&= \frac{2}{\pi} \int_{0}^{\pi} \Vert (T \otimes I_{\B}) (\phi_1 \cos (t) - \phi_2 \sin (t)) \Vert_{\ell^2 (V,m ; \B)}^2 dt \\
&\leq \Vert T \otimes I_{\B} \Vert_{B(\ell^2 (V,m ; \B))}^2 \left( \frac{2}{\pi} \int_{0}^{\pi} \Vert \phi_1 \cos (t) - \phi_2 \sin (t) \Vert_{\ell^2 (V,m ; \B)}^2 dt \right) \\
&= \Vert T \otimes I_{\B} \Vert_{B(\ell^2 (V,m ; \B))}^2 \Vert \phi \Vert_{\ell^2 (V,m ; \B_{\mathbb{C}})}^2,
\end{align*}
as needed.
\end{proof}

\Needspace{10\baselineskip}
Combining complex interpolation and complexification leads to the following:
\begin{proposition}
\label{prop: interpolation and complex}
Let $V$ be a finite set, let $m: V \rightarrow \mathbb{R}_{>0}$ be a function, let $T : \ell^2 (V,m) \rightarrow \ell^2 (V,m)$ be a linear operator, and let $\B$ be a real Banach space. Then, for every $0 < \theta \leq 1$, 
\[
\Vert T \otimes I_{\B} \Vert_{B(\ell^{p} (V,m ; \B))} \leq \Vert T \otimes I_{\B_{\mathbb{C}}} \Vert_{B (\ell^{\infty} (V, m ; \B_{\mathbb{C}}))}^{1- \theta} \Vert T \otimes I_{\B} \Vert_{B (\ell^{2} (V, m ; \B))}^{ \theta},
\]
with $\frac{1}{p} = \frac{\theta}{2}$.
\end{proposition}

\begin{proof}
The case $\theta =1$ is immediate, so we may assume $0 < \theta < 1$.

Let $p$ be as above. We note that $\ell^{p} (V,m ; \B)$ is a subspace of $\ell^{p} (V,m ; \B_{\mathbb{C}})$ and thus 
\[
\Vert T \otimes I_{\B} \Vert_{B(\ell^{p} (V,m ; \B))} \leq \Vert T \otimes I_{\B_{\mathbb{C}}} \Vert_{B(\ell^{p} (V,m ; \B_{\mathbb{C}}))}.
\]
Therefore, it suffices to prove that 
\[
\Vert T \otimes I_{\B_{\mathbb{C}}} \Vert_{B(\ell^{p} (V,m ; \B_{\mathbb{C}}))} \leq \Vert T \otimes I_{\B_{\mathbb{C}}} \Vert_{B (\ell^{\infty} (V, m ; \B_{\mathbb{C}}))}^{1- \theta} \Vert T \otimes I_{\B} \Vert_{B (\ell^{2} (V, m ; \B))}^{ \theta}.
\]
This inequality follows from \Cref{thm: interpolation} (with $\B_0 = \B_1 = \B_{\mathbb{C}}, p_1 =2$) and \Cref{prop: l-2 complexification operator norm equality}.
\end{proof}

\begin{corollary}
\label{corollary: interpolation and complex}
Let $V$ be a finite set, let $m: V \rightarrow \mathbb{R}_{>0}$ be a function, let $2 \leq p < \infty$, let $T : \ell^2 (V,m) \rightarrow \ell^2 (V,m)$ be a linear operator, and let $\B$ be a real Banach space. Assume that there is a constant $L \geq 1$ such that for every Banach space $\B '$ (real or complex), $\Vert T \otimes I_{\B '} \Vert_{B (\ell^{\infty} (V, m ; \B '))} \leq L$. Then 
\[
\Vert T \otimes I_{\B} \Vert_{B(\ell^{p} (V,m ; \B))} \leq L \Vert T \otimes I_{\B} \Vert_{B (\ell^{2} (V, m ; \B))}^{\frac{2}{p}}.
\]
\end{corollary}

\begin{proof}
Let $\theta = \frac{2}{p}$. Then, by \Cref{prop: interpolation and complex},
\begin{align*}
\Vert T \otimes I_{\B_{\mathbb{C}}} \Vert_{B(\ell^{p} (V,m ; \B_{\mathbb{C}}))} &\leq \Vert T \otimes I_{\B_{\mathbb{C}}} \Vert_{B (\ell^{\infty} (V, m ; \B_{\mathbb{C}}))}^{1- \frac{2}{p}} \Vert T \otimes I_{\B} \Vert_{B (\ell^{2} (V, m ; \B))}^{\frac{2}{p}} \\
&\leq L^{1- \frac{2}{p}} \Vert T \otimes I_{\B} \Vert_{B (\ell^{2} (V, m ; \B))}^{\frac{2}{p}} \\
&\leq L \Vert T \otimes I_{\B} \Vert_{B (\ell^{2} (V, m ; \B))}^{\frac{2}{p}}.
\end{align*}
\end{proof}

\subsection{Uniformly curved Banach spaces}
\label{subsection: Uniformly curved Banach spaces}

\begin{definition}[Fully contractive operator]
An operator $T: \ell^2 (V,m) \rightarrow \ell^2 (V,m)$ is called fully contractive if, for every (real or complex) Banach space $\B$, $\Vert T \otimes I_{\B} \Vert_{B( \ell^2 (V,m ; \B))} \leq 1$. 
\end{definition}

\begin{definition}[Uniformly curved space]
\label{uc def}
Let $\B$ be a Banach space. The space $\B$ is called uniformly curved if for every $0 < \varepsilon \leq 1$ there is $0 < \delta \leq 1$ such that for every space $\ell^2 (V,m)$ and every fully contractive linear operator $T:\ell^2 (V,m) \rightarrow \ell^2 (V,m)$, if $\Vert T \Vert_{B(\ell^2 (V,m))} \leq \delta$, then $\Vert T \otimes I_{\B} \Vert_{B(\ell^2 (V,m ; \B))} \leq \varepsilon$. 
\end{definition}

\Needspace{4\baselineskip}
The following theorem is due to Pisier \cite{Pisier2}:
\begin{theorem}
\label{thm: uni. cur. are s-reflex.}
Every uniformly curved Banach space is super-reflexive.
\end{theorem}

Given a monotone increasing function $\ucurvfun: [0,1] \rightarrow  [0,1]$ such that $\ucurvfun (0) =0$ and
\[
\lim_{t \rightarrow 0^{+}} \ucurvfun (t) = 0,
\]
we define 
$\mathcal{E}^{\text{u-curved}}_\ucurvfun$ to be the class of all uniformly curved real Banach spaces $\B$ such that for every $0<\delta\leq1$, every space $\ell^2 (V,m)$, and every fully contractive linear operator $T:\ell^2 (V,m) \rightarrow \ell^2 (V,m)$, if $\Vert T \Vert_{B(\ell^2 (V,m))} \leq \delta$, then $\Vert T \otimes I_{\B} \Vert_{B(\ell^2 (V,m ; \B))} \leq \ucurvfun (\delta)$. 

The following results were proved in \cite{OppZukBan}:

\begin{proposition} \cite[Corollary 2.10]{OppZukBan}
\label{prop: uni-curv is closed under ultraproducts}
For any monotone increasing function $\ucurvfun: [0,1] \rightarrow  [0,1]$ such that $\ucurvfun (0) =0$ and $\lim_{t \rightarrow 0^{+}} \ucurvfun (t) = 0$, the class $\mathcal{E}^{\text{u-curved}}_\ucurvfun$ defined above is closed under ultraproducts, subspaces and quotients of Banach spaces.
\end{proposition}

\begin{proposition} \cite[Proposition 2.11]{OppZukBan}
\label{prop: L contractive}
Let $T :\ell^2 (V,m) \rightarrow \ell^2 (V,m)$ be a linear operator and let $L\geq 1$, $0< \delta \leq 1$ be constants such that
\begin{enumerate}
\item $\Vert T \Vert_{B(\ell^2 (V,m))} \leq \delta$.
\item For every Banach space $\B$, $\Vert T \otimes \id_{\B} \Vert_{B(\ell^2 (V,m ; \B))} \leq L$.
\end{enumerate}
Then, for every monotone increasing function $\ucurvfun: [0,1] \rightarrow  [0,1]$ such that $\ucurvfun (0) =0$ and $\lim_{t \rightarrow 0^{+}} \ucurvfun (t) = 0$, and every $\B \in \mathcal{E}^{\text{u-curved}}_\ucurvfun$, $\Vert T \otimes \id_{\B} \Vert_{B(\ell^2 (V,m ; \B))} \leq L \ucurvfun (\delta)$.
\end{proposition}

\subsection{\texorpdfstring{$\theta$}{theta}-Hilbertian spaces}
\label{subsect Strictly theta-Hilbertian spaces}

We consider complex $\theta$-Hilbertian Banach spaces,
introduced in \cite[Section~6]{Pisier2}, where $0<\theta<1$.
We will not recall the full definition here. Informally, one starts
with finite-dimensional spaces obtained by complex interpolation of
a family of Banach spaces parametrized by the unit circle, requiring
the spaces in the family to be Hilbert spaces on a set of normalized
measure at least $\theta$, and then passes to ultraproducts.
Throughout, we use the norms specified by this construction.
For the endpoint $\theta=1$, we take the complex $1$-Hilbertian
spaces to be the complex Hilbert spaces.

We call a real Banach space $\B$ \textit{$\theta$-Hilbertian} if there
are a complex $\theta$-Hilbertian Banach space $\widehat{\B}$ and a
real-linear isometric embedding
\[
    j:\B\longrightarrow\widehat{\B}.
\]
Thus $\B$ is identified with a closed real linear subspace of
$\widehat{\B}$, endowed with the inherited norm.

The following proposition gives the quantitative uniform curvature
bound that we need.

\begin{proposition}
\label{prop: theta-Hilbertian uniformly curved}
Let $0<\theta_0\leq 1$. If $\B$ is a real $\theta$-Hilbertian Banach
space for some $\theta_0\leq\theta\leq 1$, then $\B$ is uniformly
curved and
\[
    \B\in
    \mathcal{E}^{\text{u-curved}}_{\ucurvfun(t)=t^{\theta_0}}.
\]
\end{proposition}

\begin{proof}
Choose a complex $\theta$-Hilbertian space $\widehat{\B}$ and a
real-linear isometric embedding $j:\B\longrightarrow\widehat{\B}$.
If $\theta=1$, then $\widehat{\B}$ is a Hilbert space and $\B$ is
a real Hilbert space. The assertion follows from the Hilbert-valued
operator norm bound and the inequality $t\leq t^{\theta_0}$ for
$0<t\leq 1$. We therefore assume that $0<\theta<1$.

Let $V$ be a finite set, let $m:V\longrightarrow\mathbb{R}_{>0}$,
and let $T:\ell^2(V,m)\longrightarrow\ell^2(V,m)$ be a fully
contractive real linear operator. Fix $0<t\leq 1$ such that
\[
    \Vert T\Vert_{B(\ell^2(V,m))}\leq t.
\]
Write $H=\ell^2(V,m;\mathbb{C})$ and let $T_{\mathbb{C}}\in B(H)$
be the complex linear extension of $T$, given by the same matrix.

For an operator $S\in B(H)$, its \textit{regular norm} satisfies
\[
    \Vert S\Vert_{\mathrm{reg}}
    =
    \sup_{\B'}
    \Vert S\otimes I_{\B'}\Vert_{B(\ell^2(V,m;\B'))},
\]
where the supremum is over all complex Banach spaces $\B'$;
see \cite[Section~1.2, equation~(1.1)]{Pisier2}.
Consequently, full contractivity of $T$ implies that
$\Vert T_{\mathbb{C}}\Vert_{\mathrm{reg}}\leq 1$.
Moreover,
\[
    \Vert T_{\mathbb{C}}\Vert_{B(H)}
    =
    \Vert T\Vert_{B(\ell^2(V,m))}
    \leq t.
\]
Indeed, the equality follows by writing a complex scalar-valued
function as $f+ig$ and using
$\Vert f+ig\Vert_H^2=\Vert f\Vert_{\ell^2(V,m)}^2+
\Vert g\Vert_{\ell^2(V,m)}^2$.

Pisier's estimate \cite[Section~7, equation~(7.2) and the paragraph
following it]{Pisier2}, together with the standard norm inequality
for complex interpolation, gives
\[
    \Vert S\otimes I_{\widehat{\B}}\Vert_
        {B(\ell^2(V,m;\widehat{\B}))}
    \leq
    \Vert S\Vert_{\mathrm{reg}}^{1-\theta}
    \Vert S\Vert_{B(H)}^\theta.
\]
To apply the cited estimate to the weighted space, conjugate $S$
by the isometry
\[
    D_m:\ell^2(V,m;\B')\longrightarrow\ell^2(V;\B'),
    \qquad
    (D_m\phi)(v)=m(v)^{1/2}\phi(v).
\]
This is an isometry for every Banach space $\B'$, so the scalar
operator norm, the regular norm, and the amplified operator norm
are all preserved by this conjugation.

Applying the preceding estimate to $T_{\mathbb{C}}$ yields
\[
    \Vert T_{\mathbb{C}}\otimes I_{\widehat{\B}}\Vert_
        {B(\ell^2(V,m;\widehat{\B}))}
    \leq t^\theta.
\]
Finally, the coordinatewise embedding induced by $j$ is isometric
and intertwines $T\otimes I_{\B}$ with the restriction of
$T_{\mathbb{C}}\otimes I_{\widehat{\B}}$, because the matrix
coefficients of $T$ are real. Hence
\[
    \Vert T\otimes I_{\B}\Vert_{B(\ell^2(V,m;\B))}
    \leq
    \Vert T_{\mathbb{C}}\otimes I_{\widehat{\B}}\Vert_
        {B(\ell^2(V,m;\widehat{\B}))}
    \leq t^\theta
    \leq t^{\theta_0}.
\]
Since $t^{\theta_0}\longrightarrow 0$ as $t\longrightarrow 0^+$,
this proves the assertion.
\end{proof}

The interpolation construction also gives a quantitative bound on
the modulus of convexity and implies uniform smoothness.

\begin{proposition}
\label{prop: theta-Hilbertian convexity and smoothness}
Let $0<\theta_0\leq 1$ and set $p=2/\theta_0$.
Every real $\theta$-Hilbertian Banach space $\B$ with
$\theta_0\leq\theta\leq 1$ is uniformly smooth and $(p,1/(p\,4^{p-1}))$-uniformly convex.
\end{proposition}

\begin{proof}
The case $\theta_0=1$ is the Hilbert space case, so assume
$0<\theta_0<1$.
Choose a complex $\theta$-Hilbertian space $\widehat{\B}$ into
which $\B$ embeds real-linearly and isometrically.
Since $\theta\geq\theta_0$, the space $\widehat{\B}$ is also
$\theta_0$-Hilbertian.

By \cite[Proposition~2.1 and the paragraph following its proof]{LaatSalle},
inequality~(2) of that paper holds on $\widehat{\B}$ with exponent
$p=2/\theta_0$ and coefficient $2^{2-p}$.
Applying it to a random variable taking the values $\xi$ and $\eta$
with equal probabilities gives
\[
    \left\Vert\frac{\xi+\eta}{2}\right\Vert^p
    +\frac{\Vert\xi-\eta\Vert^p}{4^{p-1}}
    \leq
    \frac{\Vert\xi\Vert^p+\Vert\eta\Vert^p}{2}.
\]
This inequality restricts to $\B$. Hence
\[
    \delta_{\B}(\varepsilon)
    \geq
    1-\left(1-\frac{\varepsilon^p}{4^{p-1}}\right)^{1/p}
    \geq
    \frac{\varepsilon^p}{p\,4^{p-1}},
\]
where the last inequality follows from
$1-(1-s)^{1/p}\geq s/p$ for $0\leq s\leq 1$.

The class of complex $\theta_0$-Hilbertian spaces is closed under
duality, by interpolation duality
\cite[Section~5.2]{Pisier2} and duality for reflexive ultraproducts
\cite[proof of Proposition~7.1]{Heinrich1980}.
Applying the same convexity result to $\widehat{\B}^{*}$ shows
that $\widehat{\B}$ is uniformly smooth.
Uniform smoothness passes to real linear isometric subspaces,
so $\B$ is uniformly smooth as well.
\end{proof}

Together with \Cref{prop: theta-Hilbertian uniformly curved},
this gives common parameters
\[
    p=\frac{2}{\theta_0},
    \qquad
    C=\frac{1}{p\,4^{p-1}},
    \qquad
    \ucurvfun(t)=t^{\theta_0}
\]
for the convexity and curvature assumptions, as well as uniform
smoothness, throughout the class of real $\theta$-Hilbertian
spaces with $\theta\geq\theta_0$.

\begin{remark}
By \Cref{thm: uni. cur. are s-reflex.} and the preceding proposition,
every real $\theta$-Hilbertian Banach space is super-reflexive.
\end{remark}

\subsection{Random walks on multigraphs}
\label{section: Random walks on multigraphs}
We recall that a multigraph $(V,E)$ is a generalization of a graph in which there can be multiple edges between two vertices.  Throughout, all multigraphs are finite and loop-free. We require a multigraph to have no isolated vertices whenever we use its random walk operator (see below) or its degree-weighted $\ell^p$-space; graph partitions (see below) themselves are allowed to contain pieces with isolated vertices.  For a multigraph $(V,E)$ and $v, u \in V$,  we denote by $m (\lbrace u,v\rbrace)$ the number of edges between $u$ and $v$.  We will call $m$ the \textit{degree function of $(V,E)$}.   If $m (\lbrace u,v\rbrace) \geq 1$,  we say that $u$ and $v$ are \textit{neighbors} and write $v \sim u$.  We further define the \textit{degree of $v \in V$} to be $m (v) = \sum_{u \in V} m (\lbrace u,v\rbrace)$.  For a nonempty set of vertices $U \subseteq V$ and a vertex $v \in V$,  we also define $m (U) = \sum_{u \in U} m (u)$ and $m (v,U) = \sum_{u \in U,  u \sim v} m (\lbrace v,u \rbrace)$.  Although $E$ is not a set but a multiset,  we will abuse notation below and write $\lbrace u,v \rbrace \in E$ when $m (\lbrace u,v\rbrace) \geq 1$.   In sums indexed by \(\{u,v\}\in E\), each distinct unordered adjacent pair is included once; its multiplicity is supplied by the factor \(m(\{u,v\})\).

With this definition,  for $1 \leq  p \leq \infty$ and a Banach space $\B$, we can define $\ell^p (V, m )$ and  $\ell^p (V, m ; \B)$ as above.  

The random walk operator $M : \ell^p (V ) \rightarrow \ell^p (V )$ is defined by 
\[
M \phi (v) = \frac{1}{m (v)} \sum_{u \sim v} m (\lbrace u,v\rbrace) \phi (u) .
\]

\begin{proposition}
\label{prop: rw into coupling sum}
Let $1 < p < \infty$ be a constant and let $\B$ be a Banach space such that $\B$ is strictly smooth.  Then for every $\phi \in \ell^p (V,  m ; \B)$,
\begin{align*}
& \langle (M \otimes I_{\B}) \phi, J_{\ell^p (V, m ; \B), p} \phi \rangle_{\ell^p (V,m ; \B)} \\
&\qquad = \sum_{\lbrace u,v \rbrace \in E} m (\lbrace u,v \rbrace ) \left( \langle \phi (v), J_{\B, p} (\phi (u)) \rangle_{\B} + \langle \phi (u), J_{\B, p} (\phi (v)) \rangle_{\B} \right).
\end{align*}
\end{proposition}

\begin{proof}
Fix $\phi \in \ell^p (V,  m ; \B)$.  By \Cref{prop: J ell-p}, for every $v \in V$, \[
(J_{\ell^p (V, m ; \B), p} \phi) (v) = J_{\B, p} (\phi (v)).
\]  Therefore
\begin{align*}
& \langle (M \otimes I_{\B}) \phi, J_{\ell^p (V, m ; \B), p} \phi \rangle_{\ell^p (V,m ; \B)} \\
&\qquad = \sum_{v \in V} m (v) \langle \frac{1}{m (v)} \sum_{u \in V, \lbrace u,v \rbrace \in E} m (\lbrace u,v\rbrace) \phi (u), J_{\B, p} (\phi (v)) \rangle \\
&\qquad = \sum_{v \in V} \sum_{u \in V, \lbrace u,v \rbrace \in E} m (\lbrace u,v\rbrace) \langle \phi (u), J_{\B, p} (\phi (v)) \rangle \\
&\qquad = \sum_{\lbrace u,v \rbrace \in E} m (\lbrace u,v \rbrace ) \left( \langle \phi (v), J_{\B, p} (\phi (u)) \rangle_{\B} + \langle \phi (u), J_{\B, p} (\phi (v)) \rangle_{\B} \right)
\end{align*}
as needed.
\end{proof}

Let $(V,E),  (V' , E')$ be two multigraphs with degree functions $m, m'$.  We say that $(V',E ')$ is a sub-multigraph of $(V,E)$ if $V ' \subseteq V,  E ' \subseteq E$ and for every $\lbrace u,v \rbrace \in E$, $m ' (\lbrace u,v \rbrace) \leq m  (\lbrace u,v \rbrace)$ (we define $m ' (\lbrace u,v \rbrace) =0$ for $\lbrace u,v \rbrace \in E \setminus E'$).  For a multigraph $(V, E)$,  a set of multigraphs $\lbrace (V_i,  E_i ) \rbrace_{i=1}^k$ is called a \textit{partition of $(V,E)$} if for every $1 \leq i \leq k$,  $(V_i,E_i)$ is a sub-multigraph of $(V,E)$ and for every $\lbrace u,v \rbrace \in E$, $m  (\lbrace u,v \rbrace) = \sum_{i=1}^k m_i (\lbrace u,v \rbrace)$, where $m_i$ denotes the degree function of $ (V_i,  E_i )$.

\begin{lemma}
\label{lemma: sub-multigraph norm}
Let $1 < p < \infty$ be a constant and let $\B$ be a Banach space.  Also let $(V,E)$ be a multigraph and $\lbrace (V_i,  E_i ) \rbrace_{i=1}^k$ a partition of $(V,E)$ such that every $(V_i,E_i)$ has no isolated vertices.  Then, for every $\phi \in \ell^p (V, m ; \B)$, 
\[
\Vert \phi \Vert^p_{\ell^p (V, m ; \B)} = \sum_{i=1}^k \Vert \left.  \phi \right\vert_{V_i} \Vert^p_{\ell^p (V_i, m_i ; \B)}.
\]
\end{lemma}

\begin{proof}
We note that for every $v \in V$,  $m (v) = \sum_{i=1}^k m_i (v)$ (for $v \notin V_i$, we define $m_i (v) =0$).  Thus
\begin{align*}
\Vert \phi \Vert^p_{\ell^p (V, m ; \B)} &= \sum_{v \in V} m (v) \Vert \phi (v) \Vert_{\B}^p \\
&= \sum_{i=1}^k \sum_{v \in V_i} m_i (v) \Vert \phi (v) \Vert_{\B}^p \\
&= \sum_{i=1}^k \Vert \left. \phi \right\vert_{V_i} \Vert^p_{\ell^p (V_i, m_i ; \B)}
\end{align*}
as needed.
\end{proof}

\begin{proposition}
\label{prop: rw bound on sub-multigraph}
Let $1 < p < \infty$ be a constant and let $\B$ be a Banach space.  Also let $(V,E)$ be a multigraph and $\lbrace (V_i,  E_i ) \rbrace_{i=1}^k$ a partition of $(V,E)$ such that every $(V_i,E_i)$ does not have isolated vertices.  We denote by $M_i$ the random walk operator on $(V_i,  E_i)$.  Let $\phi,  \phi ' \in \ell^p (V, m ; \B)$ and assume that there is $\epsilon \geq 0$ such that for every $1 \leq i \leq k$, 
\[
\Vert (M_i \otimes I_{\B}) \left. \phi ' \right\vert_{V_i} \Vert_{\ell^p (V_i , m_i ; \B)} \leq \epsilon \Vert \left. \phi \right\vert_{V_i} \Vert_{\ell^p (V_i , m_i ; \B)}.
\]
Then 
\[
\Vert (M \otimes I_{\B}) \phi ' \Vert_{\ell^p (V , m ; \B)} \leq \epsilon \Vert \phi  \Vert_{\ell^p (V , m ; \B)}.
\] 
\end{proposition}

\begin{proof}
Let $\phi,  \phi '$ and $\epsilon$ be as above.  

We note that for every $v \in V$, 
\begin{align*}
(M \otimes I_{\B}) \phi ' (v) &= \frac{1}{m (v)} \sum_{\lbrace u,v \rbrace \in E} m (\lbrace u,v \rbrace) \phi ' (u) \\
&= \frac{1}{m (v)} \sum_{i =1}^k \sum_{\lbrace u,v \rbrace \in E_i} m_i (\lbrace u,v \rbrace) \phi ' (u) \\
&= \sum_{1 \leq i \leq k, v \in V_i} \frac{m_i (v)}{m (v)} \frac{1}{m_i (v)} \sum_{\lbrace u,v \rbrace \in E_i} m_i (\lbrace u,v \rbrace) \phi ' (u) \\
&= \sum_{1 \leq i \leq k, v \in V_i} \frac{m_i (v)}{m (v)} ((M_i \otimes I_{\B}) \phi ' (v)).
\end{align*}

Thus
\begin{align*}
& \Vert (M \otimes I_{\B}) \phi ' \Vert_{\ell^p (V , m ; \B)}^p \\
&\qquad = \sum_{v \in V} m (v) \Vert (M \otimes I_{\B}) \phi ' (v) \Vert_{\B}^p \\
&\qquad = \sum_{v \in V} m (v) \left\Vert \sum_{1 \leq i \leq k, v \in V_i} \frac{m_i (v)}{m (v)} ((M_i \otimes I_{\B}) \phi ' (v)) \right\Vert_{\B}^p \\
&\qquad \leq^{\text{convexity of } \Vert . \Vert^p} \sum_{v \in V} \sum_{1 \leq i \leq k, v \in V_i} m_i (v) \Vert ((M_i \otimes I_{\B}) \phi ' (v)) \Vert_{\B}^p \\
&\qquad = \sum_{i =1}^k \Vert (M_i \otimes I_{\B}) \left. \phi ' \right\vert_{V_i} \Vert_{\ell^p (V_i , m_i ; \B)}^p \\
&\qquad \leq \sum_{i =1}^k \epsilon^p \Vert \left. \phi \right\vert_{V_i} \Vert_{\ell^p (V_i , m_i ; \B)}^p \\
&\qquad =^{\Cref{lemma: sub-multigraph norm}} \epsilon^p \Vert \phi \Vert_{\ell^p (V , m ; \B)}^p
\end{align*}
which proves the required inequality.
\end{proof}

\section{Norm bounds for random walks on bipartite (multi)graphs}
\label{section: Norm bounds for random walks on bipartite (multi)graphs}
A multigraph $(V,E)$ is called \textit{bipartite} if there is a partition $V = V_0 \sqcup V_1$ such that for every $\lbrace u,v \rbrace \in E$,  $\vert \lbrace u,v \rbrace \cap V_0 \vert = \vert \lbrace u,v \rbrace \cap V_1 \vert =1$.  The sets $V_0, V_1$ are called the \textit{sides}. Equivalently, each edge has one vertex on each side.  

We note that for a bipartite multigraph $(V,E)$ with sides $V_0, V_1$,  the function $\mathbbm{1}_{V_0} - \mathbbm{1}_{V_1}$ is always an eigenfunction of $M$ with eigenvalue $-1$.  As for any multigraph,  the constant function $\mathbbm{1}_V = \mathbbm{1}_{V_0} + \mathbbm{1}_{V_1}$ is also an eigenfunction of $M$.  We refer to these two eigenfunctions as the ``trivial'' eigenfunctions of $M$.  Since $M$ is self-adjoint with respect to the inner product of $\ell^2 (V,m)$ (and thus has orthogonal eigenspaces), it follows that $M$ preserves the space $\left( span \lbrace \mathbbm{1}_{V_0} + \mathbbm{1}_{V_1},  \mathbbm{1}_{V_0} - \mathbbm{1}_{V_1} \rbrace \right)^\perp = \left( span \lbrace \mathbbm{1}_{V_0} ,   \mathbbm{1}_{V_1} \rbrace \right)^\perp$.  Let $P_{(V,E),  \sides} = P_{\lbrace V_0,  V_1 \rbrace, m }$ be the operator defined in \Cref{subsection: Vector valued finite measure spaces} above, where $m$ is the degree function of $(V,E)$.   For $1 \leq p \leq \infty$ and a Banach space $\B$,  we define $\lambda_{\bipartite,  p,  \B} (V,E)$ by 
\[
\lambda_{\bipartite,  p,  \B} (V,E) = \Vert M (I-P_{(V,E),  \sides}) \otimes I_{\B} \Vert_{B (\ell^p (V,  m ; \B))}.
\]

When $\B = \mathbb{R}$, we write $\lambda (V,E) = \lambda_{\bipartite,  2,  \mathbb{R}} (V,E)$. We note that if $(V,E)$ has more than two vertices, then $\lambda (V,E)$ is the second-largest eigenvalue of $M$ (counting eigenvalues with multiplicity when $(V,E)$ is disconnected). This follows from the fact that the spectrum of $M$ is symmetric for a bipartite multigraph; thus, if $(V,E)$ is connected, then all the eigenvalues of $M$ in $ \left( span \lbrace \mathbbm{1}_{V_0} ,   \mathbbm{1}_{V_1} \rbrace \right)^\perp$ are in the interval $[- \lambda (V,E), \lambda (V,E)]$. 

\Needspace{3\baselineskip}
\begin{proposition}
\label{prop: lambda Hilbert equal lambda}
For every (real or complex) Hilbert space $\mathcal{H}$, we have $\lambda_{\bipartite,  2,  \mathcal{H}} (V,E) \leq \lambda (V,E)$. 
\end{proposition}

\begin{proof}
The inequality readily follows from \Cref{prop: norm of T tensor Hilbert or Lp}.
\end{proof}

\begin{proposition}
\label{prop: bipartite lambda is always leq 2}
For every $1 \leq p \leq \infty$ and every Banach space $\B$, we have $\lambda_{\bipartite,  p,  \B} (V,E) \leq 2$.
\end{proposition}

\begin{proof}
Fix $\B$ and $p$ as above. We note that both $M$ and $P_{(V,E),  \sides}$ are Markov operators preserving $m$. Thus, by \Cref{prop: Markov operator norm bound}, we have 
\begin{align*}
& \lambda_{\bipartite, p, \B} (V,E) \\
& \leq \Vert M \otimes I_{\B} \Vert_{B (\ell^p (V, m ; \B))} + \Vert M \otimes I_{\B} \Vert_{B (\ell^p (V, m ; \B))} \Vert P_{(V,E), \sides} \otimes I_{\B} \Vert_{B (\ell^p (V, m ; \B))} \\
& \leq 2.
\end{align*}
\end{proof}

\begin{proposition}
\label{prop: lambda-p bound by lambda-2}
For every $2 \leq p < \infty$ and every real Banach space $\B$, we have 
\[
\lambda_{\bipartite,  p,  \B} (V,E)  \leq 2 \left(\lambda_{\bipartite,  2,  \B} (V,E) \right)^{\frac{2}{p}}.
\]
\end{proposition}

\begin{proof}
Combine \Cref{corollary: interpolation and complex} and \Cref{prop: bipartite lambda is always leq 2}.
\end{proof}

\begin{proposition}
\label{prop: lambda Lp bound}
For every $2 \leq p < \infty$ and every (real or complex) space $L^p (\Omega, \mu)$, we have 
\[
\lambda_{\bipartite,  p, L^p (\Omega, \mu) } (V,E) \leq 2 (\lambda (V,E))^{\frac{2}{p}}.
\]
\end{proposition}

\begin{proof}
Fix $2 \leq p < \infty$. By \Cref{prop: norm of T tensor Hilbert or Lp}, it suffices to prove that
\[
\Vert M (I-P_{(V,E),  \sides}) \Vert_{B (\ell^p (V,m))} \leq 2 (\lambda (V,E))^{\frac{2}{p}}
\]
and this follows from \Cref{prop: lambda-p bound by lambda-2} with $\B = \mathbb{R}$.
\end{proof}

\begin{proposition}
\label{prop: lambda ineq of uni-curved}
Let $2 \leq p < \infty$ and let $\ucurvfun: [0,1] \rightarrow  [0,1]$ be a monotone increasing function such that $\ucurvfun (0) =0$ and 
$\lim_{t \rightarrow 0^{+}} \ucurvfun (t) = 0$. Then, for every $\B \in \mathcal{E}^{\text{u-curved}}_\ucurvfun$, we have 
\[
\lambda_{\bipartite,  p, \B} (V,E) \leq 4  \left( \ucurvfun (\lambda (V,E)) \right)^{\frac{2}{p}}.
\]
\end{proposition}

\begin{proof}
Fix $p, \ucurvfun$ as above, and let $\B \in \mathcal{E}^{\text{u-curved}}_\ucurvfun$. Then
\begin{align*}
& \lambda_{\bipartite, p, \B } (V,E) \\
& \leq^{\Cref{prop: lambda-p bound by lambda-2}} 2 \left( \lambda_{\bipartite, 2, \B } (V,E) \right)^{\frac{2}{p}} \\
& \leq^{\Cref{prop: L contractive}, \Cref{prop: bipartite lambda is always leq 2}} 2 \left(2 \ucurvfun ( \lambda (V,E)) \right)^{\frac{2}{p}} \\
& \leq 4 \left( \ucurvfun (\lambda (V,E)) \right)^{\frac{2}{p}}
\end{align*}
\end{proof}

\begin{proposition}
\label{prop: submultigraph lambda bound}
Let $(V,E)$ be a bipartite multigraph with sides $V_0, V_1$.  Assume that $(V' ,E')$, with $V' = V$, is a sub-multigraph of $(V,E)$ and that there is $0 \leq \epsilon <1$ such that for every $v \in V$,  $(1- \epsilon) m (v) \leq m' (v)$, where $m'$ is the degree function of $(V', E')$.  Then 
\[
\lambda (V,E) \leq  \frac{\epsilon}{1- \epsilon}  + \frac{2 \epsilon}{(1-\epsilon)^2} +  \lambda (V',E') .
\]
\end{proposition}

\begin{proof}
As noted above,  the random walk operator $M$ preserves the space $ \left( span \lbrace \mathbbm{1}_{V_0} ,   \mathbbm{1}_{V_1} \rbrace \right)^\perp$, where the orthogonal complement is taken with respect to the inner product of $\ell^2 (V, m)$.   Since $M$ is self-adjoint,  it suffices to show that for every $\phi \in \ell^2 (V,m)$, we have 
\begin{align*}
& \langle M (I- P_{(V,E), \sides}) \phi , (I- P_{(V,E), \sides}) \phi \rangle_{\ell^2 (V,m)} \\
& \leq \left(\frac{\epsilon}{1- \epsilon} + \frac{2 \epsilon}{(1-\epsilon)^2} + \lambda (V',E') \right) \Vert \phi \Vert_{\ell^2 (V,m)}^2.
\end{align*}

Fix $\phi \in \ell^2 (V,m)$ and set $\psi = (I- P_{(V,E), \sides}) \phi$ and $\psi ' = (I- P_{(V',E '), \sides}) \phi$.   By \Cref{corollary: proj in different m's comparison} with $\epsilon_1 = \epsilon$ and $\epsilon_2 = 0$, we have
\begin{align*}
\Vert \psi - \psi ' \Vert_{\ell^2 (V,m)} &= \Vert P_{(V,E), \sides} \phi - P_{(V' ,E '), \sides} \phi \Vert_{\ell^2 (V,m)} \\
& \leq \Vert P_{(V,E), \sides} - P_{(V' ,E '), \sides} \Vert_{B (\ell^2 (V,m))} \Vert \phi \Vert_{\ell^2 (V,m)} \\
& \leq \frac{\epsilon}{1- \epsilon} \Vert \phi \Vert_{\ell^2 (V,m)}.
\end{align*}
Therefore, we also have
\[
\Vert \psi ' \Vert_{\ell^2 (V,m)} \leq \Vert \psi  \Vert_{\ell^2 (V,m)} + \frac{\epsilon}{1- \epsilon}  \Vert \phi \Vert_{\ell^2 (V,m)} \leq \frac{1}{1- \epsilon}  \Vert \phi \Vert_{\ell^2 (V,m)}.
\]

Thus
\begin{align*}
& \langle M \psi , \psi \rangle_{\ell^2 (V,m)} - \langle M \psi ' , \psi ' \rangle_{\ell^2 (V,m)} \\
& = \langle M (\psi - \psi ') , \psi \rangle_{\ell^2 (V,m)} + \langle M \psi ' , \psi - \psi ' \rangle_{\ell^2 (V,m)} \\
& \leq \Vert M (\psi - \psi ') \Vert_{\ell^2 (V,m)} \Vert \psi \Vert_{\ell^2 (V,m)} + \Vert M \psi ' \Vert_{\ell^2 (V,m)} \Vert \psi - \psi ' \Vert_{\ell^2 (V,m)} \\
& \leq^{\Vert M \Vert_{B (\ell^2 (V,m))} \leq 1} \frac{\epsilon}{1- \epsilon} \Vert \phi \Vert_{\ell^2 (V,m)} \Vert \psi \Vert_{\ell^2 (V,m)} \\*
& \qquad + \Vert \psi ' \Vert_{\ell^2 (V,m)} \frac{\epsilon}{1- \epsilon} \Vert \phi \Vert_{\ell^2 (V,m)} \\
& \leq \frac{\epsilon}{1- \epsilon} \Vert \phi \Vert_{\ell^2 (V,m)}^2 + \frac{\epsilon}{(1-\epsilon)^2} \Vert \phi \Vert_{\ell^2 (V,m)}^2
\end{align*}
and it remains to prove that 
\[
\langle M \psi ' ,  \psi ' \rangle_{\ell^2 (V,m)}  \leq \left( \frac{ \epsilon}{(1-\epsilon)^2} +  \lambda (V',E') \right) \Vert \phi \Vert_{\ell^2 (V,m)}^2.
\]

Let $M'$ denote the random walk operator on $(V',E')$.  Then
\begin{align*}
& \langle M \psi ' , \psi ' \rangle_{\ell^2 (V,m)} - \langle M ' \psi ' , \psi ' \rangle_{\ell^2 (V ',m ')} \\
& =^{\Cref{prop: rw into coupling sum}} \sum_{\lbrace v,u \rbrace \in E} m (\lbrace v,u \rbrace ) (2 \psi ' (v) \psi ' (u)) \\*
& \qquad - \sum_{\lbrace v,u \rbrace \in E} m ' (\lbrace v,u \rbrace ) (2 \psi ' (v) \psi ' (u)) \\
& = \sum_{\lbrace v,u \rbrace \in E} (m (\lbrace v,u \rbrace ) - m' (\lbrace v,u \rbrace ) ) (2 \psi ' (v) \psi ' (u)) \\
& \leq \sum_{\lbrace v,u \rbrace \in E} (m (\lbrace v,u \rbrace ) - m' (\lbrace v,u \rbrace ) ) ((\psi ' (v))^2+ (\psi ' (u))^2) \\
& = \sum_{v \in V} (m (v) - m ' (v)) (\psi ' (v))^2 \leq \epsilon \sum_{v \in V} m (v) (\psi ' (v))^2 \\
& = \epsilon \Vert \psi ' \Vert_{\ell^2 (V,m)}^2 \leq \frac{ \epsilon}{(1- \epsilon)^2} \Vert \phi \Vert^2_{\ell^2 (V,m)}.
\end{align*}

Thus, it remains to show that
\[
\langle M ' \psi ' ,  \psi ' \rangle_{\ell^2 (V ',m ')}  \leq \lambda (V',E')  \Vert \phi \Vert_{\ell^2 (V,m)}^2.
\]
Indeed, since $\psi ' =  (I- P_{(V',E '), \sides}) \phi$, we have
\[
\langle M ' \psi ' ,  \psi ' \rangle_{\ell^2 (V ',m ')} \leq \lambda (V',E') \Vert \phi \Vert_{\ell^2 (V' , m')}^2 \leq \lambda (V',E') \Vert \phi \Vert_{\ell^2 (V , m)}^2
\]
as required.
\end{proof}

For $D_0, D_1 \in \mathbb{N}$ and $0 \leq \varepsilon <1$,  a bipartite multigraph $(V,E)$ with sides $V_0, V_1$ is called \textit{$(D_0, D_1, \varepsilon)$-almost biregular} if for every $i =0,1$ and every $v \in V_i$, we have $\vert m (v) - D_i \vert \leq \varepsilon D_i$.  We also define $m_{D_0, D_1} : V \rightarrow \mathbb{N}$ by 
\[
m_{D_0,  D_1} (v) = 
\begin{cases}
D_0 & v \in V_0 \\
D_1 & v \in V_1
\end{cases}.
\]
With this notation,  $(V,E)$ is \textit{$(D_0, D_1, \varepsilon)$-almost biregular} if for every $v \in V$,  $\vert m (v) - m_{D_0,D_1} (v) \vert \leq \varepsilon m_{D_0,D_1} (v)$.  

\begin{proposition}
\label{prop: norm and coupling comparison with almost bireg}
Let $(V,E)$ be a bipartite multigraph with sides $V_0, V_1$ that is $(D_0, D_1, \varepsilon)$-almost biregular for constants $D_0,  D_1 \in \mathbb{N}$ and $0 \leq \varepsilon <1$.  Also let $1 < p < \infty$ be a constant,  $q$ be the H\"{o}lder conjugate of $p$, and $\B$ be a real Banach space.  
\begin{enumerate}
\item For every $\phi : V \rightarrow \B$ and $\phi ' : V \rightarrow \B^*$, we have 
\begin{align*}
& \left\vert \langle \phi, \phi ' \rangle_{\ell^p (V, m ; \B)} - \langle \phi, \phi ' \rangle_{\ell^p (V, m_{D_0,D_1} ; \B)} \right\vert \\
& \leq \varepsilon \Vert \phi \Vert_{\ell^p (V, m_{D_0, D_1} ; \B)} \Vert \phi ' \Vert_{\ell^q (V, m_{D_0, D_1} ; \B^*)},
\end{align*}
\begin{align*}
& \left\vert \langle \phi, \phi ' \rangle_{\ell^p (V, m ; \B)} - \langle \phi, \phi ' \rangle_{\ell^p (V, m_{D_0,D_1} ; \B)} \right\vert \\
& \leq \frac{\varepsilon}{1- \varepsilon} \Vert \phi \Vert_{\ell^p (V, m ; \B)} \Vert \phi ' \Vert_{\ell^q (V, m ; \B^*)}.
\end{align*}
\item For every $\phi : V \rightarrow \B$, we have 
\[
\left\vert \Vert \phi \Vert_{\ell^p (V, m ; \B)}^p - \Vert \phi \Vert_{\ell^p (V, m_{D_0, D_1} ; \B)}^p \right\vert \leq \varepsilon \Vert \phi \Vert_{\ell^p (V, m_{D_0, D_1}; \B)}^p,
\]
\[
\left\vert \Vert \phi \Vert_{\ell^p (V, m ; \B)}^p - \Vert \phi \Vert_{\ell^p (V, m_{D_0, D_1} ; \B)}^p \right\vert \leq \frac{\varepsilon}{1- \varepsilon} \Vert \phi \Vert_{\ell^p (V, m; \B)}^p.
\]
\end{enumerate}
\end{proposition}

\begin{proof}
These inequalities follow from \Cref{corollary: coupling and norm under different m's} with $m_1 = m_{D_0, D_1},\allowbreak  m_2 = m$ and $\epsilon_1 = \epsilon_2 = \varepsilon$.
\end{proof}

\begin{proposition}
\label{prop: proj comparison with almost bireg}
Let $(V,E)$ be a bipartite multigraph with sides $V_0, V_1$ that is $(D_0, D_1, \varepsilon)$-almost biregular for constants $D_0,  D_1 \in \mathbb{N}$ and $0 \leq \varepsilon <1$.  Also let $1 < p < \infty$ be a constant and let $\B$ be a Banach space.  Then, for any partition $\lbrace U_1,...,U_k \rbrace$ of $V$, we have 
\[
\Vert P_{\lbrace U_1,...,U_k \rbrace,  m} \otimes I_{\B} - P_{\lbrace U_1,...,U_k \rbrace,  m_{D_0,  D_1}} \otimes I_{\B} \Vert_{B (\ell^p (V, m ; \B))} \leq  \frac{2 \varepsilon}{1- \varepsilon}
\]
and 
\[
\Vert P_{\lbrace U_1,...,U_k \rbrace,  m} \otimes I_{\B} - P_{\lbrace U_1,...,U_k \rbrace,  m_{D_0,  D_1}} \otimes I_{\B} \Vert_{B (\ell^p (V, m_{D_0,D_1} ; \B))} \leq  \frac{2 \varepsilon}{1- \varepsilon}.
\]
In particular, 
\[
\Vert P_{(V,E),  \sides} \otimes I_{\B} - P_{\lbrace V_0,  V_1 \rbrace,  m_{D_0,  D_1}} \otimes I_{\B} \Vert_{B (\ell^p (V, m ; \B))} \leq  \frac{2 \varepsilon}{1- \varepsilon}.
\]
\end{proposition}

\begin{proof}
These inequalities follow from \Cref{corollary: proj in different m's comparison} with $m_1 = m_{D_0, D_1},  m_2 = m$ and $\epsilon_1 = \epsilon_2 = \varepsilon$.
\end{proof}

\begin{lemma}
\label{lemma: rw of almost bireg} 
Let $(V,E)$ be a bipartite multigraph with sides $V_0, V_1$ that is $(D_0, D_1, \varepsilon)$-almost biregular for constants $D_0,  D_1 \in \mathbb{N}$ and $0 \leq \varepsilon <1$.  Also let $1 < p < \infty$ be a constant and let $\B$ be a Banach space.   Then
\[
\Vert M (I - P_{\lbrace V_0,  V_1 \rbrace,  m_{D_0,  D_1}}) \otimes I_{\B} \Vert_{B (\ell^p (V, m ; \B))} \leq \frac{2 \varepsilon}{1- \varepsilon} + \lambda_{\bipartite,  p,  \B} (V,E).
\]
\end{lemma}

\begin{proof}
By \Cref{prop: proj comparison with almost bireg}, we have 
\begin{align*}
& \Vert M (I - P_{\lbrace V_0, V_1 \rbrace, m_{D_0, D_1}}) \otimes I_{\B} \Vert_{B (\ell^p (V, m ; \B))} \\
& \leq \Vert M ( P_{(V,E), \sides} - P_{\lbrace V_0, V_1 \rbrace, m_{D_0, D_1}}) \otimes I_{\B} \Vert_{B (\ell^p (V, m ; \B))} \\*
& \qquad + \Vert M (I - P_{(V,E), \sides} ) \otimes I_{\B} \Vert_{B (\ell^p (V, m ; \B))} \\
& \leq \Vert M \otimes I_{\B} \Vert_{B (\ell^p (V, m ; \B))} \Vert (P_{(V,E), \sides} - P_{\lbrace V_0, V_1 \rbrace, m_{D_0, D_1}}) \otimes I_{\B} \Vert_{B (\ell^p (V, m ; \B))} \\*
& \qquad +\lambda_{\bipartite, p, \B} (V,E) \\
& \leq^{\Cref{prop: proj comparison with almost bireg} } \frac{2 \varepsilon}{1- \varepsilon} + \lambda_{\bipartite, p, \B} (V,E)
\end{align*}
as required.
\end{proof}

\section{Cocycle equalities and inequalities}
\label{section: Cocycle equalities and inequalities}
In this section, we prove equalities and inequalities regarding cocycles that will be needed in the sequel. Throughout this section and the next section, all cocycles are taken with respect to linear isometric representations.

Let $\Gamma = \langle S \vert R \rangle$ be a group such that $S = \lbrace s_1^{\pm},...,s_n^{\pm} \rbrace$ and all the relators in $R$ are cyclically reduced in the alphabet $S$.  The set $S$ denotes the formal symmetric alphabet of \(2n\) letters, and reduced words are counted as formal words, even when they represent the same element of \(\Gamma\). Throughout, we assume that $n \geq 2$ and that there is $\RelLen \in \mathbb{N},  \RelLen \geq 4$ such that all the relators in $R$ are of length $\RelLen$ (with respect to the alphabet $S$).


\begin{proposition}
\label{prop: bound norm c (w)}
Let $\Gamma$ be a discrete group generated by $S = \lbrace s_1^{\pm},...,s_n^{\pm} \rbrace$. Let $\B$ be a Banach space, let $\pi$ be an isometric representation of $\Gamma$ on $\B$, and let $c : \Gamma \rightarrow \B$ be a cocycle with respect to $\pi$.  Let $k \in \mathbb{N}$ and let $W_k$ denote the set of all reduced words in the alphabet $S$ of length $k$.  Then, for every $1 < p < \infty$,
\[
\sum_{w \in W_k} \Vert c (w) \Vert^p \leq (2n-1)^{k-1} k^p  \sum_{s \in S} \Vert c (s) \Vert^p.
\]
\end{proposition}

\begin{proof}
Fix $k \in \mathbb{N}$.  Every $w \in W_k$ is of the form $w = s^{(1)} ... s^{(k)}$, where $s^{(1)},...,s^{(k)} \in S$.  We define the function $A : W_k \times S \rightarrow \mathbb{N} \cup \lbrace 0 \rbrace$ by
\[
A (s^{(1)} ... s^{(k)},  s) = \vert \lbrace 1 \leq i \leq k : s = s^{(i)} \rbrace \vert.
\]
We note that 
\begin{align*}
\sum_{s \in S} \sum_{w \in W_k} A (w,s) &= \sum_{w \in W_k} \sum_{s \in S} A (w,s) = \sum_{w \in W_k} k = (2n) (2n-1)^{k-1} k.
\end{align*}
By symmetry, for every $s, s' \in S$, we have 
\[
\sum_{w \in W_k} A (w,s) = \sum_{w \in W_k} A (w,s ')
\]
and the preceding computation therefore gives, for every $s \in S$, 
\[
\sum_{w \in W_k} A (w,s)  = (2n-1)^{k-1} k.
\]

For every $w = s^{(1)} ... s^{(k)}$, we have 
\begin{align*}
\Vert c (w) \Vert &= \Vert c (s^{(1)}) + \pi (s^{(1)}) c (s^{(2)} ... s^{(k)} ) \Vert \\
& \leq \Vert c (s^{(1)}) \Vert + \Vert c (s^{(2)} ... s^{(k)} ) \Vert \leq ... \\
& \leq \sum_{i=1}^k \Vert c (s^{(i)}) \Vert.
\end{align*}
Thus
\[
\Vert c (w) \Vert^p \leq \left( \sum_{i=1}^k \Vert c (s^{(i)}) \Vert \right)^p \leq k^{p-1} \sum_{i=1}^k \Vert c (s^{(i)}) \Vert^p =   k^{p-1} \sum_{s \in S}  A (w,s) \Vert c (s) \Vert^p.
\]

Combining this with the previous computation yields 
\begin{align*}
\sum_{w \in W_k} \Vert c (w) \Vert^p &\leq \sum_{w \in W_k} k^{p-1} \sum_{s \in S} A (w,s) \Vert c (s) \Vert^p \\
& = k^{p-1} \sum_{s \in S} \Vert c (s) \Vert^p \sum_{w \in W_k} A (w,s) \\
& = k^{p-1} \sum_{s \in S} \Vert c (s) \Vert^p \left( (2n-1)^{k-1} k \right) \\
& = (2n-1)^{k-1} k^p \sum_{s \in S} \Vert c (s) \Vert^p .
\end{align*}
\end{proof}

\begin{proposition}
\label{prop:c g-inv}
Let $\Gamma$ be a discrete group, let $\B$ be a Banach space, let $\pi$ be an isometric representation of $\Gamma$ on $\B$, and let $c : \Gamma \rightarrow \B$ be a $1$-cocycle with respect to $\pi$.  Then, for every $g$,  $\pi (g^{-1}) c (g) = - c ( g^{-1})$.  
\end{proposition}

\begin{proof}
Let $g \in \Gamma$. Then
\begin{align*}
\underline{0} &= c (e) = c (g^{-1} g) =c (g^{-1})+ \pi (g^{-1}) c (g)
\end{align*}
and the desired identity follows.
\end{proof}

\begin{proposition}
\label{prop: J-p pi (g)}
Let $\Gamma$ be a discrete group, let $\B$ be a Banach space, let $\pi$ be an isometric representation of $\Gamma$ on $\B$, and let $c : \Gamma \rightarrow \B$ be a cocycle with respect to $\pi$.  Assume that $\B$ is strictly smooth. Then, for every $1 < p < \infty$,  every $g \in \Gamma$, and every $\xi \in \B$, we have $ J_{p} (\pi (g) \xi) = \overline{\pi} (g) J_{p} (\xi) $.
\end{proposition}

\begin{proof}
By the definition of the function $J_p : \B \rightarrow \B^*$,  if we set $\eta = J_{p} (\pi (g) \xi) $,  then $\eta$ is the unique vector in $\B^*$ such that $\Vert \eta \Vert_{\B^*} = \Vert \pi (g) \xi \Vert^{p-1}_{\B}$ and $\langle \pi (g) \xi,  \eta \rangle_{\B} = \Vert \pi (g) \xi \Vert^{p}_{\B}$.  We will show that $\overline{\pi} (g) J_{p} (\xi)$ satisfies these two conditions and hence that $ J_{p} (\pi (g) \xi) = \overline{\pi} (g) J_{p} (\xi) $.

First,
\begin{align*}
\Vert \overline{\pi} (g) J_{p} (\xi) \Vert_{\B^{*}} &=^{\overline{\pi} \text{ is an isometric representation}} \Vert J_{p} (\xi) \Vert_{\B^{*}} \\
& = \Vert \xi \Vert_{\B}^{p-1} =^{\pi \text{ is an isometric representation}} \Vert \pi (g) \xi \Vert^{p-1}_{\B}
\end{align*}
as required.

Second,  
\begin{align*}
\langle \pi (g) \xi, \overline{\pi} (g) J_{p} (\xi) \rangle_{\B} &= \langle \xi, J_{p} (\xi) \rangle_{\B} \\
& = \Vert \xi \Vert_{\B}^{p} = \Vert \pi (g) \xi \Vert^{p}_{\B}
\end{align*}
as required.
\end{proof}

\begin{proposition}
\label{prop: relation coupling}
Let $\Gamma$ be a finitely generated group with a generating set $S$ and let $w$ be a relation of $\Gamma$.  Write $w = w_1 w_2 w_3 w_4$, where $w_1,w_2,w_3,w_4$ are words in the alphabet $S$.  Then, for every $1 < p < \infty$,  every Banach space $\B$ such that $\B$ is strictly smooth,  every isometric representation $\pi$ of $\Gamma$ on $\B$, and every cocycle $c$ with respect to $\pi$, we have
\[
\sum_{i=1}^4 \left(  \langle  c (w_{i}^{-1}),  J_p (c (w_{i+1})) \rangle_{\B} + \langle  c (w_{i+1}),  J_p (c (w_{i}^{-1})) \rangle_{\B}  \right) \geq 0,
\]
where indices are taken modulo $4$.
\end{proposition}

\begin{proof}
In this proof,  the indices of the $w_i$ are taken modulo $4$; i.e., $w_4$ is also denoted $w_0$.  

Fix $\pi$ and $c$ and define $\rho$ to be the affine isometric action $\rho (g) .\xi = \pi (g) \xi + c (g)$.  Let $\xi_0 = \underline{0}$ and $\xi_i = \rho (w_1 ... w_i).\underline{0}$ for $i =1,...,3$.  Since $w_1 w_2 w_3 w_4$ is a relation,  we also have $\xi_4 = \rho (w_1 w_2 w_3 w_4). \underline{0} = \rho (e). \underline{0} = \underline{0} = \xi_0$.  Accordingly, we also take the indices of the $\xi_i$ modulo $4$. 

We also set $\eta_i = \xi_{i} - \xi_{i-1}$ (again, indices are taken modulo $4$) and note that 
\[
\sum_{i=1}^4 \eta_i = \sum_{i=1}^4 (\xi_{i} - \xi_{i-1}) = \underline{0}.
\]

Thus 
\begin{align*}
0 &= \langle \sum_{i=1}^4 \eta_i , \sum_{i=1}^4 J_p \eta_i \rangle_{\B} \\
& = \sum_{i=1}^4 \Vert \eta_i \Vert^p_{\B} + \sum_{i=1}^4 \langle \eta_i, J_p \eta_{i+2} \rangle_{\B} \\*
& \qquad + \sum_{i=1}^4 \left(\langle \eta_i, J_p \eta_{i+1} \rangle_{\B} + \langle \eta_{i+1}, J_p \eta_{i} \rangle_{\B} \right).
\end{align*}

Let $q$ denote the H\"{o}lder conjugate of $p$.  It follows that 
\begin{align*}
& - \sum_{i=1}^4 \left(\langle \eta_i, J_p \eta_{i+1} \rangle_{\B} + \langle \eta_{i+1}, J_p \eta_{i} \rangle_{\B} \right) \\
& = \sum_{i=1}^4 \Vert \eta_i \Vert^p_{\B} + \sum_{i=1}^4 \langle \eta_i, J_p \eta_{i+2} \rangle_{\B} \\
& \geq \sum_{i=1}^4 \Vert \eta_i \Vert^p_{\B} - \sum_{i=1}^4 \Vert \eta_i \Vert_{\B} \Vert J_p \eta_{i+2} \Vert_{\B^*} \\
& \geq^{\text{Young's inequality}} \sum_{i=1}^4 \Vert \eta_i \Vert^p_{\B} \\*
& \qquad - \sum_{i=1}^4 \left( \frac{1}{p} \Vert \eta_i \Vert_{\B}^p + \frac{1}{q} \Vert J_p \eta_{i+2} \Vert_{\B^*}^q \right) \\
& = 0.
\end{align*}

To complete the proof, we show that for every $1 \leq i \leq 4$,  
\[
- \langle \eta_i,  J_p \eta_{i+1} \rangle_{\B} =  \langle  c (w_{i}^{-1}),  J_p (c (w_{i+1})) \rangle_{\B}
\]
and 
\[
-  \langle \eta_{i+1},  J_p \eta_{i} \rangle_{\B} = \langle  c (w_{i+1}),  J_p (c (w_{i}^{-1})) \rangle_{\B}.
\]

We note that for $2 \leq i \leq 4$,  
\begin{align*}
\eta_i &= \xi_{i} - \xi_{i-1} = c (w_1 ... w_{i}) - c (w_1 ... w_{i-1}) \\
& = (c(w_1 ... w_{i-1}) + \pi (w_1 ... w_{i-1}) c (w_{i})) - c (w_1 ... w_{i-1}) \\
& = \pi (w_1 ... w_{i-1}) c (w_{i}).
\end{align*}
Also,  for $i=1$,  
\[
\eta_1 = \xi_1 - \xi_0 =  c (w_1) - \underline{0} =c (w_1).
\]
For $i=1$, we define $\pi (w_1 ...w_{i-1})$ to be $I_{\B}$. With this notation, for all $1 \leq i \leq 4$, we have
\[
\eta_i =  \pi (w_1 ... w_{i-1}) c (w_{i}).
\]

Fix $1 \leq i \leq 4$.  Then
\begin{align*}
& - \langle \eta_i, J_p \eta_{i+1} \rangle_{\B} \\
& = - \langle \pi (w_1 ... w_{i-1}) c (w_{i}), J_p \left( \pi (w_1 ... w_{i}) c (w_{i+1}) \right) \rangle_{\B} \\
& =^{\Cref{prop: J-p pi (g)}} - \langle \pi (w_1 ... w_{i-1}) c (w_{i}), \overline {\pi} (w_1 ... w_{i}) J_p \left( c (w_{i+1}) \right) \rangle_{\B} \\
& = - \langle \pi (w_i^{-1}) c (w_{i}), J_p \left( c (w_{i+1}) \right) \rangle_{\B} \\
& =^{\Cref{prop:c g-inv}} - \langle - c (w_{i}^{-1}), J_p \left( c (w_{i+1}) \right) \rangle_{\B} \\
& = \langle c (w_{i}^{-1}), J_p (c (w_{i+1})) \rangle_{\B}
\end{align*}
as required.

Similarly,
\begin{align*}
& - \langle \eta_{i+1}, J_p \eta_{i} \rangle_{\B} \\
& = - \langle \pi (w_1 ... w_{i}) c (w_{i+1}), J_p (\pi (w_1 ... w_{i-1}) c (w_{i}) )\rangle_{\B} \\
& = - \langle c (w_{i+1}), \overline{\pi} (w_{i}^{-1} ... w_1^{-1}) J_p (\pi (w_1 ... w_{i-1}) c (w_{i}) )\rangle_{\B} \\
& =^{\Cref{prop: J-p pi (g)}} - \langle c (w_{i+1}), J_p (\pi (w_{i}^{-1}) c (w_{i}) )\rangle_{\B} \\
& =^{\Cref{prop:c g-inv}} - \langle c (w_{i+1}), J_p (- c (w_{i}^{-1}) )\rangle_{\B} \\
& = \langle c (w_{i+1}), J_p ( c (w_{i}^{-1}) )\rangle_{\B}
\end{align*}
as required.
\end{proof}

A cocycle $c$ is called \textit{harmonic with respect to $S$} (or \textit{harmonic} if $S$ is clear from the context) if $\sum_{s \in S} c (s) = 0$.  

Fix $k \in \mathbb{N}$.  We denote by $W_k$ the set of all reduced words of length $k$ in the alphabet $S$.  For $s \in S$,  we further denote by $W_{s,k}$ the set of all reduced words of length $k$ in the alphabet $S$ that begin with $s$.  

\begin{proposition}
\label{prop: sum c (w)}
Let $\Gamma$ be a discrete group generated by $S = \lbrace s_1^{\pm},...,s_n^{\pm} \rbrace$.  Then, for every isometric representation $\pi$ of $\Gamma$,  every $k \in \mathbb{N}$, every harmonic cocycle $c$ with respect to $\pi$ and every $s \in S$, we have 
\[
\sum_{w \in W_{s,k}} c (w) =(1 +(2n-1) + ... + (2n-1)^{k-1}) c(s)   = \frac{(2n-1)^{k}-1}{2n-2} c (s)
\]
\end{proposition}

\begin{proof}
We prove the equality by induction on $k$.  For $k=1$, for every $s \in S$, we have $W_{s,1} = \lbrace s \rbrace$ and thus $\sum_{w \in W_{s,1}} c (w) = c(s)$, as required.  Assume that the equality holds for $k$.  Every $w \in W_{s,k+1}$ is of the form $w = s w'$, where $w ' \in \bigcup_{t \in S,  t \neq s^{-1}} W_{t,k}$.  Thus
\begin{align*}
& \sum_{w \in W_{s,k+1}} c (w) \\
& = \sum_{t \in S, t \neq s^{-1}} \sum_{w' \in W_{t,k}} c (s w') \\
& = \sum_{t \in S, t \neq s^{-1}} \sum_{w' \in W_{t,k}} (c(s) + \pi (s) c (w')) \\
& = (2n-1)^k c (s) + \pi (s) \sum_{t \in S, t \neq s^{-1}} \sum_{w' \in W_{t,k}} c (w') \\
& =^{\text{induction assumption}} (2n-1)^k c (s) \\*
& \qquad + \pi (s) \sum_{t \in S, t \neq s^{-1}} (1 +(2n-1) + \cdots + (2n-1)^{k-1}) c (t) \\
& = (2n-1)^k c (s) \\*
& \qquad +(1 +(2n-1) + ... +(2n-1)^{k-1}) \pi (s) \left( \sum_{t \in S, t \neq s^{-1}} c (t) \right) \\
& =^{c \text{ is harmonic}} (2n-1)^k c (s) \\*
& \qquad +(1 +(2n-1) + ... +(2n-1)^{k-1}) ( - \pi (s) c (s^{-1})) \\
& =^{\Cref{prop:c g-inv}} (2n-1)^k c (s) \\*
& \qquad +(1 +(2n-1) + ... + (2n-1)^{k-1}) c (s) \\
& = (1 +(2n-1) + \cdots + (2n-1)^{k}) c (s)
\end{align*}
\end{proof}

\begin{proposition}
\label{prop: couplings sum ineq}
Let $\Gamma$ be a discrete group generated by $S = \lbrace s_1^{\pm},...,s_n^{\pm} \rbrace$.  Let $\B$ be a uniformly convex and strictly smooth Banach space, let $\pi$ be an isometric representation of $\Gamma$ on $\B$, and let $c: \Gamma \rightarrow \B$ be a harmonic cocycle with respect to $\pi$.  Assume further that there are $2 \leq p < \infty$ and $0< C \leq 1$ such that $\B$ is $(p,C)$-uniformly convex.

Then, for every $k \in \mathbb{N}$,
\[
\sum_{s \in S} \sum_{w \in W_{s,k}} \langle c(s), J_p (c (w)) \rangle_{\B}  \geq \frac{2C}{p} \frac{(2n-1)^{k} -1}{2n-2} \sum_{s \in S} \Vert c (s) \Vert^p_{\B}.
\]
\end{proposition}

\begin{proof}
The proof is by induction on $k$.  For $k=1$,  $W_{s,1} = \lbrace s \rbrace$, and thus 
\[
\sum_{s \in S} \sum_{w \in W_{s,1}} \langle c(s), J_p (c (w)) \rangle_{\B} = \sum_{s \in S}  \Vert c (s) \Vert^p_{\B} \geq \frac{2C}{p} \sum_{s \in S}  \Vert c (s) \Vert^p_{\B}
\]
as required.

Assume that the inequality holds for $k$. 
\begin{align*}
& \sum_{s \in S} \sum_{w \in W_{s,k+1}} \langle c(s), J_p (c (w)) \rangle_{\B} \\
& = \sum_{s \in S} \sum_{w ' \in W_{k} \setminus W_{s^{-1}, k}} \langle c(s), J_p (c (s w')) \rangle_{\B} \\
& = \sum_{s \in S} \sum_{w ' \in W_{k} \setminus W_{s^{-1}, k}} \langle c(s), J_p (c (s w')) - J_p (\pi (s) c (w')) \rangle_{\B} \\*
& \qquad + \sum_{s \in S} \sum_{w ' \in W_{k} \setminus W_{s^{-1}, k}} \langle c(s), J_p (\pi (s) c (w')) \rangle_{\B} \\
& \geq^{c (s) = c (s w') - \pi (s) c (w'), \Cref{thm: J-p inequality} } \sum_{s \in S} \sum_{w ' \in W_{k} \setminus W_{s^{-1}, k}} \frac{2C}{p} \Vert c(s) \Vert_{\B}^p \\*
& \qquad + \sum_{s \in S} \sum_{w ' \in W_{k} \setminus W_{s^{-1}, k}} \langle c(s), J_p (\pi (s) c (w')) \rangle_{\B} \\
& = \frac{2C}{p} (2n-1)^k \sum_{s \in S} \Vert c(s) \Vert_{\B}^p \\*
& \qquad + \sum_{s \in S} \sum_{w ' \in W_{k} \setminus W_{s^{-1}, k}} \langle c(s), J_p (\pi (s) c (w')) \rangle_{\B}.
\end{align*}
Thus, it suffices to prove that 
\[
\sum_{s \in S} \sum_{w ' \in W_{k} \setminus W_{s^{-1}, k}} \langle c(s),  J_p (\pi (s) c (w'))  \rangle_{\B} =  \sum_{s \in S} \sum_{w  \in W_{s,k}} \langle c(s),  J_p (c (w))  \rangle_{\B}.
\]

Indeed,
\begin{align*}
& \sum_{s \in S} \sum_{w ' \in W_{k} \setminus W_{s^{-1}, k}} \langle c(s), J_p (\pi (s) c (w')) \rangle_{\B} \\
& =^{\Cref{prop: J-p pi (g)}} \sum_{s \in S} \sum_{w ' \in W_{k} \setminus W_{s^{-1}, k}} \langle c(s), \overline{\pi} (s) J_p (c (w')) \rangle_{\B} \\
& = \sum_{s \in S} \sum_{w ' \in W_{k} \setminus W_{s^{-1}, k}} \langle \pi (s^{-1}) c(s), J_p (c (w')) \rangle_{\B} \\
& =^{\Cref{prop:c g-inv}} \sum_{s \in S} \sum_{w ' \in W_{k} \setminus W_{s^{-1}, k}} \langle - c(s^{-1}), J_p (c (w')) \rangle_{\B} \\
& = \sum_{s \in S} \sum_{s' \in S, s' \neq s^{-1}} \sum_{w' \in W_{s' ,k}} \langle - c(s^{-1}), J_p (c (w')) \rangle_{\B} \\
& = \sum_{s' \in S} \sum_{w' \in W_{s' ,k}} \langle \sum_{s \in S, s \neq s'} - c(s), J_p (c (w')) \rangle_{\B} \\
& =^{c \text{ is harmonic}} \sum_{s' \in S} \sum_{w' \in W_{s' ,k}} \langle c (s'), J_p (c (w')) \rangle_{\B} \\
& = \sum_{s \in S} \sum_{w \in W_{s,k}} \langle c(s), J_p (c (w)) \rangle_{\B}
\end{align*}
as required.
\end{proof}

We recall that Shalom \cite{Shalom} proved that a finitely generated group $\Gamma$ has property (T) if and only if, for every orthogonal representation $\pi$ of $\Gamma$ and every cocycle $c$ of $\pi$,  if $c$ is harmonic, then $c \equiv 0$.  Below,  we prove a generalization of this result that is probably well known to experts, but for which we have not found a reference.

\begin{theorem}
\label{thm: exist non-zero harmonic cocycle}
Let $\Gamma$ be a finitely generated group with a generating set $S = \lbrace s_1^{\pm},...,s_n^{\pm} \rbrace$, and let $\mathcal{E}$ be a class of uniformly convex and uniformly smooth Banach spaces that is stable under ultraproducts.  If $\Gamma$ does not have property $(F_{\mathcal{E}})$,  then there are $\B  \in \mathcal{E}$,  an isometric linear representation $\pi$ of $\Gamma$ on $ \B $, and a cocycle $c: \Gamma \rightarrow \B $ with respect to $\pi$ such that $c$ is harmonic and $c \not\equiv 0$.
\end{theorem}

\begin{proof}
By \Cref{thm: S-uniform},  we have $\mathcal{C}_{\mathcal{E}}^{S-\text{uniform}} \neq \emptyset$,  i.e., there are $\B_0  \in \mathcal{E}$ and an isometric action $\rho_0$ on $\B_0$ that is $S$-uniform.  Let $\pi_0$ and $c_0$ be the linear isometric representation and the cocycle associated with $\rho_0$.  

We define $T \in B (\B_0)$ by 
\[
T \xi = \frac{1}{\vert S \vert} \sum_{s \in S} \pi_0 (s) \xi .
\]
We note that $\Vert T \Vert_{B (\B_0)} \leq 1$ and that $\B_0$ is reflexive.  Thus, the mean ergodic theorem for power-bounded operators on reflexive Banach spaces gives $\B_0 = \Ker (I-T) \oplus \overline{\im (I-T)}$.  Consequently, for every $\xi \in \B_0$, there are unique $\xi_0 \in \Ker (I-T)$ and $\xi_1 \in \overline{\im (I-T)}$ such that $\xi = \xi_0 + \xi_1$. 

Let $\xi_0 \in \Ker (I-T)$. Then 
\[
\Vert \xi_0 \Vert = \Vert T \xi_0 \Vert = \left\Vert  \frac{1}{\vert S \vert} \sum_{s \in S} \pi_0 (s) \xi_0  \right\Vert  \leq \frac{1}{\vert S \vert} \sum_{s \in S} \Vert \pi_0 (s) \xi_0 \Vert = \Vert \xi_0 \Vert .
\]
Thus, by the strict convexity of $\B_0$, for every $s,  s' \in S$, we have $\pi_0 (s) \xi_0 = \pi_0 (s ') \xi_0$. Therefore, for every $s \in S$,  $\xi_0 = T \xi_0 = \pi_0 (s) \xi_0$.  Since $S$ generates $\Gamma$, it follows that $\Ker (I-T) = \B_0^{\pi_0 (\Gamma)}$.  By \Cref{prop: J-p pi (g)}, for every $g \in \Gamma$,  $\overline{\pi_0} (g) J_2 (\xi_0) =  J_2 (\xi_0)$.

Also, for every $\eta \in \B_0$,  we have 
\begin{align*}
& \langle (I-T) \eta, J_2 (\xi_0) \rangle \\
& = \langle \eta, J_2 (\xi_0) \rangle - \frac{1}{\vert S \vert} \sum_{s \in S} \langle \pi_0 (s) \eta, J_2 (\xi_0) \rangle \\
& = \langle \eta, J_2 (\xi_0) \rangle - \frac{1}{\vert S \vert} \sum_{s \in S} \langle \eta, \overline{\pi_0} (s^{-1}) J_2 (\xi_0) \rangle \\
& = \langle \eta, J_2 (\xi_0) \rangle - \frac{1}{\vert S \vert} \sum_{s \in S} \langle \eta, J_2 (\xi_0) \rangle \\
& =0.
\end{align*}

Thus, for every $\xi_0 \in \Ker (I-T)$ and every $\xi_1 \in \overline{\im (I-T)}$, we have $\langle \xi_1,  J_2 (\xi_0)\rangle =0$.

Let $\xi = \frac{1}{\vert S \vert} \sum_{s \in S} c_0 (s)$ and decompose $\xi = \xi_0 + \xi_1$ as above.  We show that $\xi_0 =0$.  Indeed, 
\begin{align*}
& \Vert \xi_0 \Vert^2 \\
& = \langle \xi_0, J_2 (\xi_0) \rangle = \langle \xi_0 + \xi_1 , J_2 (\xi_0) \rangle \\
& = \langle \frac{1}{\vert S \vert} \sum_{s \in S} c_0 (s), J_2 (\xi_0) \rangle \\
& = \frac{1}{2n} \sum_{i=1}^n \langle c_0 (s_i), J_2 (\xi_0) \rangle + \frac{1}{2n} \sum_{i=1}^n \langle c_0 (s_i^{-1}), J_2 (\xi_0) \rangle \\
& =^{\Cref{prop:c g-inv}} \frac{1}{2n} \sum_{i=1}^n \langle c_0 (s_i), J_2 (\xi_0) \rangle + \frac{1}{2n} \sum_{i=1}^n \langle - \pi_0 (s_i^{-1}) c_0 (s_i), J_2 (\xi_0) \rangle \\
& = \frac{1}{2n} \sum_{i=1}^n \langle c_0 (s_i), J_2 (\xi_0) \rangle - \frac{1}{2n} \sum_{i=1}^n \langle c_0 (s_i), \overline{\pi_0} (s_i) J_2 (\xi_0) \rangle \\
& = \frac{1}{2n} \sum_{i=1}^n \langle c_0 (s_i), J_2 (\xi_0) \rangle - \frac{1}{2n} \sum_{i=1}^n \langle c_0 (s_i), J_2 (\xi_0) \rangle \\
& =0
\end{align*}

Thus, $\frac{1}{\vert S \vert} \sum_{s \in S} c_0 (s) \in \overline{\im (I-T)}$, and there exist $\eta_n \in \B_0$ such that $\lim_n \eta_n - T \eta_n = \frac{1}{\vert S \vert} \sum_{s \in S} c_0 (s)$.  

Set $\B_n = \B_0$ and define $\rho_n$ using the linear isometric representation $\pi_0$ and the cocycle 
\begin{align*}
c_n (g) &= \frac{1}{\disp_{\rho_0, S} (\eta_n)} \left( c_0 (g) - ( \eta_n - \pi_0 (g) \eta_n) \right) \\
& = \frac{1}{\disp_{\rho_0, S} (\eta_n)} \left( \rho_0 (g) \eta_n - \eta_n \right).
\end{align*} 
Thus, for every $n$,  
\begin{align*}
\disp_{\rho_n, S} (0_n) &= \max_{s \in S} \Vert c_n (s) \Vert \\
& = \frac{1}{\disp_{\rho_0, S} (\eta_n)} \max_{s \in S} \Vert \rho_0 (s) \eta_n - \eta_n \Vert \\
& = \frac{1}{\disp_{\rho_0, S} (\eta_n)} \disp_{\rho_0, S} (\eta_n) =1.
\end{align*}
Thus, for a nonprincipal ultrafilter $\mathcal{U}$, there is an affine isometric action $\rho = \rho_{\mathcal{U}}$ of $\Gamma$ on the ultraproduct $\lbrace \B_n \rbrace_{\mathcal{U}}$.  Let $\B$ denote the ultraproduct and let $c$ denote the cocycle of $\rho$.  We have shown that for every $n$,  
$\max_{s \in S} \Vert c_n (s) \Vert =1$.  Thus $\max_{s \in S} \Vert c (s) \Vert =1$, and in particular,  $c$ is non-zero.  We recall that $\rho_0$ is $S$-uniform. In particular, $\disp_{\rho_0,  S} (\eta_n) \geq 1$.  Thus
\begin{align*}
& \left\Vert \frac{1}{\vert S \vert} \sum_{s \in S} c_n (s) \right\Vert \\
& = \frac{1}{\disp_{\rho_0, S} (\eta_n)} \left\Vert \frac{1}{\vert S \vert} \sum_{s \in S} \left( c_0 (s) - (\eta_n - \pi_0 (s) \eta_n) \right) \right\Vert \\
& \leq^{\disp_{\rho_0, S} (\eta_n) \geq 1} \left\Vert \frac{1}{\vert S \vert} \sum_{s \in S} \left( c_0 (s) - (\eta_n - \pi_0 (s) \eta_n) \right) \right\Vert \\
& = \left\Vert \left( \frac{1}{\vert S \vert} \sum_{s \in S} c_0 (s) \right) - (I-T) \eta_n \right\Vert \rightarrow^{n \rightarrow \infty} 0.
\end{align*}
Therefore, $ \frac{1}{\vert S \vert} \sum_{s \in S} c (s) =0$, i.e., $c$ is harmonic, as required.
\end{proof}

\section{Fixed point property criteria}
\label{section: Fixed point property criteria}

\subsection{Graphs associated to the group presentation}
\label{subsection: Graphs associated to the group presentation}


We define multigraphs associated to this presentation that play a role analogous to that of the link in \.{Z}uk's criterion for fixed point properties. The group $\Gamma$ and the generating set $S$ are as in the previous section, i.e., $\Gamma = \langle S \vert R \rangle$, where $S = \lbrace s_1^{\pm},..., s_n^{\pm} \rbrace, n \geq 2$ is a formal generating set and all the relations in $R$ are of length $\RelLen \geq 4$. 

For $a \in \mathbb{N}$,  we denote by $W_{a}$ the set of all reduced words of length $a$ in the alphabet $S$.  We fix $a_1,  a_2, a_3, a_4 \in \lbrace \lfloor \frac{\RelLen}{4} \rfloor, \lceil  \frac{\RelLen}{4} \rceil \rbrace$ such that $a_1 + ... + a_4= \RelLen$.  All the indices below are taken modulo $4$,  e.g.,  $a_0$ is defined as $a_4$.  We also set $k_0^j = a_j$ and $k_1^j = a_{j+1}$.  

For words $w_0,  w_1$ in the alphabet $S$ and $r \in R$,  we say that $w_0$ \textit{cyclically precedes} $w_1$ in $r$ if one of the following holds: 
\begin{enumerate}[label=($CP${{\arabic*}}),  start=1]
\item  $w_0 \in W_{k_0^1}, w_1 \in W_{k_1^1}$ and there is $w \in W_{\RelLen - a_1 -a_2}$ such that $w_0 w_1 w  = r$.
\item  $w_0 \in W_{k_0^2}, w_1 \in W_{k_1^2}$ and there are $w \in W_{a_1},  w' \in W_{a_4}$ such that $w w_0 w_1 w' =r$.
\item  $w_0 \in W_{k_0^3}, w_1 \in W_{k_1^3}$  and there is $w \in W_{\RelLen - a_3 -a_4}$ such that $w w_0 w_1  = r$.
\item $w_0 \in W_{k_0^4}, w_1 \in W_{k_1^4}$  and there is $w \in W_{\RelLen - a_1 -a_4}$  such that $w_1 w w_0  =r$.
\end{enumerate}
We treat each of these cases separately below.  For $j=1,...,4$ and $w_0 \in W_{k_0^j}, w_1 \in W_{k_1^j}$,  we define $\cpnum_j (w_0, w_1 ; R)$ as the number of relations $r \in R$ such that $(CP j)$ holds. For example, 
\begin{align*}
& \cpnum_1 (w_0, w_1 ; R) \\
& = \vert \lbrace r \in R : \text{ there is } w \in W_{\RelLen - a_1 -a_2} \text{ such that } w_0 w_1 w \\
& = r \rbrace \vert.
\end{align*}

For $j =1,...,4$,  we define the following bipartite multigraph $(V^{j},  E^{j}) = (V^{j} (R),  E^{j} (R)) $: $V^j = V^j_0 \cup   V^j_1$, 
\[
V^j_i = \lbrace (i,w) : w \in W_{a_{j+i}} \rbrace = \lbrace (i,w) : w \in W_{k_{i}^j} \rbrace,
\]
where $j+i$ above is taken modulo $4$.

To define the edge multiset,  we define the number of edges connecting $(0,w_0) \in  V^j_0$ and $(1,w_1) \in  V^j_1$ as $m^j ((0, w_0), (1,w_1)) = \cpnum_{j} (w_0^{-1},  w_1, R)$.  We denote by $M_j$ the random walk operator on $(V^j,E^j)$.

Given a Banach space $\B$,  a linear isometric representation $\pi$ of $\Gamma$ on $\B$, and a cocycle $c$ with respect to $\pi$,   we define the \textit{cocycle function} $\phi_c^j : V^j \rightarrow \B$ for every $1 \leq j \leq 4$ by $\phi^j_c ((i,w)) =c (w)$ for every $i=0,1$ and $(i,w) \in V^j_i$.

\begin{proposition}[General \.{Z}uk-type criterion]
\label{prop: general Zuk}
Let $\Gamma = \langle S \vert R \rangle$ be a group such that $S = \lbrace s_1^{\pm},...,s_n^{\pm} \rbrace, n \geq 2$ and all the relators in $R$ are cyclically reduced in the alphabet $S$.  Assume that there is $\RelLen \in \mathbb{N},  \RelLen \geq 4$ such that all the relations in $R$ are of length $\RelLen$ (with respect to the alphabet $S$).  Let $(V^{j},  E^{j}) = (V^{j} (R),  E^{j} (R)) $ be the multigraphs defined above, and let $\mathcal{E}$ be a class of uniformly convex and uniformly smooth real Banach spaces that is closed under ultraproducts.  

Assume that the graphs $(V^{j},  E^{j}), j=1,...,4$ do not have isolated vertices and that there is $1 < p < \infty$ such that, for every $\B \in \mathcal{E}$ and every non-zero harmonic cocycle $c: \Gamma \rightarrow \B$, the following holds for every $1 \leq j \leq 4$: 
\[
\langle (M_j \otimes I_{\B}) \phi_c^j,  J_{\ell^p (V^j,  m^j ; \B),p} \phi_c^j \rangle_{\ell^p (V^j,  m^j ; \B)} < 0.
\]
Then $\Gamma$ has property $(F_{\mathcal{E}})$. 
\end{proposition}

\begin{proof}
Assume, for a contradiction, that $\Gamma$ does not have property $(F_{\mathcal{E}})$.  By \Cref{thm: exist non-zero harmonic cocycle},  there are $\B \in \mathcal{E}$ and a non-zero harmonic cocycle $c: \Gamma \rightarrow \B$.  

For every relation $w \in R$,  we write $w = w_1 w_2 w_3 w_4$, where $w_i \in W_{a_i}$.  Then
\begin{dmath*}
0 \leq^{\Cref{prop: relation coupling} } 
\sum_{w_1 w_2 w_3 w_4 \in R} \sum_{j=1}^4 \left(  \langle  c (w_{j}^{-1}),  J_{\B,p} (c (w_{j+1})) \rangle_{\B} + \langle  c (w_{j+1}),  J_{\B,p} (c (w_{j}^{-1})) \rangle_{\B}  \right) = 
\sum_{j=1}^4 \sum_{w_j \in W_{a_j},  w_{j+1} \in W_{a_{j+1}}} \cpnum_{j} (w_j, w_{j+1} ; R) \left(  \langle  c (w_{j}^{-1}),  J_{\B,p} (c (w_{j+1})) \rangle_{\B} + \langle  c (w_{j+1}),  J_{\B,p} (c (w_{j}^{-1})) \rangle_{\B}  \right) = 
\sum_{j=1}^4 \sum_{w_j \in W_{a_j},  w_{j+1} \in W_{a_{j+1}}}  m^j (\lbrace (0, w_j^{-1}),  (1, w_{j+1}) \rbrace) \left(  \langle  c (w_{j}^{-1}),  J_{\B,p} (c (w_{j+1})) \rangle_{\B} + \langle  c (w_{j+1}),  J_{\B,p} (c (w_{j}^{-1})) \rangle_{\B}  \right) = 
\sum_{j=1}^4 \sum_{v \in V_0^j,  u \in V^j_1,  \lbrace v,u \rbrace \in E^j} m^j (\lbrace v,u  \rbrace) \left(  \langle  \phi^j_c (v),  J_{\B,p} (\phi^j_c (u)) \rangle_{\B} + \langle \phi^j_c (u) ,  J_{\B,p} (\phi^j_c (v)) \rangle_{\B}  \right) =^{\Cref{prop: rw into coupling sum}} 
\sum_{j=1}^4  \langle (M_j \otimes I_{\B}) \phi^j_c,  J_{\ell^p (V^j, m^j ; \B),  p} \phi^j_c \rangle_{\ell^p (V^j,m^j ; \B)} <^{\text{By assumption}} 0.
\end{dmath*}
This is a contradiction, so $\Gamma$ has property $(F_{\mathcal{E}})$.
\end{proof}

\begin{remark}
To clarify the analogy with \.{Z}uk's criterion,  it is instructive to repeat this analysis when the relations are decomposed into $3$ pieces,  i.e., when every relation is written as $w_1 w_2 w_3$.  We omit the details of this analysis and only note that, in this case, the analogue of \Cref{prop: relation coupling} yields 
\begin{align*}
& \sum_{j=1}^3 \left( \langle c (w_{j}^{-1}), J_p (c (w_{j+1})) \rangle_{\B} + \langle c (w_{j+1}), J_p (c (w_{j}^{-1})) \rangle_{\B} \right) \\
& \geq \frac{1}{2} \sum_{j=1}^3 \left( \Vert c (w_{j}^{-1}) \Vert^p + \Vert c (w_{j+1}) \Vert^p \right)
\end{align*}
and thus the criterion for every $j =1,...,3$ becomes 
\[
\langle (M_j \otimes I_{\B}) \phi_c^j,  J_{\ell^p (V^j,  m^j ; \B),p} \phi_c^j \rangle_{\ell^p (V^j,  m^j ; \B)} < \frac{1}{2} \Vert \phi^j_c \Vert^p_{\ell^p (V^j,  m^j ; \B)}.
\]
We note that the factor $\frac{1}{2}$ on the right-hand side of the inequality is exactly the factor $\frac{1}{2}$ appearing in \.{Z}uk's original criterion.

This comparison illustrates the difficulty in satisfying this criterion: every bipartite multigraph that has more than $2$ vertices has an eigenvalue that is $\geq 0$.  Thus, for a graph satisfying the criterion of \Cref{prop: general Zuk}, one cannot hope to show that, for every $\phi \in \ell^p (V^j, m^j ; \B)$ orthogonal to the constant maps, the required inequality holds, and then apply it to $\phi_c^j$.  In \.{Z}uk's classical criterion, by contrast, property (T) is deduced by establishing a spectral bound $< \frac{1}{2}$ for random walks on graphs associated to the group $\Gamma$ for every map orthogonal to the constant maps (the projection onto the constant maps usually corresponds to some downward Laplacian and can be handled using Hodge theory).   
\end{remark}

To obtain sufficient conditions for applying \Cref{prop: general Zuk},  we will decompose the multigraphs $(V^j, E^j)$ into sub-multigraphs.  For $s \in S$,  as above, we define $W_{s,a}$ to be the set of reduced words in $W_{a}$ whose first letter is $s$.  We also set $V_{i,s}^j = \lbrace (i,w) : w \in W_{s, a_{j+i}} \rbrace$ for $i=0,1,  j=1,...,4$.  We note that for every $1 \leq j \leq 4$ and for every $s \in S$,  if $v \in V_{0,s}^j$ and $u \in  V_{1,s}^j$, then $\lbrace v,u \rbrace \notin E^j$.  Indeed,  there are $w_0 ',  w_1 '$ such that $v = (0, s w_0'),  u = (1, s w_1')$.  Since the relations in $R$ are cyclically reduced,  $(s w_0')^{-1}$ cannot cyclically precede $(s w_1' )$ and hence $v,u$ cannot be connected by an edge in $(V^j, E^j)$.  

We also define $V_{s,s'}^j = V_{0,s}^j \cup V_{1,s'}^j$ and define $(V_{s,s'}^j,  E_{s,s'}^j)$ to be the sub-multigraph of $(V^j,E^j)$ spanned by the vertex set $V_{s,s'}^j$.  We denote by $m_{s,s'}^j$ the degree function of this graph.  We note that for every $1 \leq j \leq 4$,  $\lbrace (V_{s,s'}^j,  E_{s,s'}^j) \rbrace_{s,s' \in S,  s \neq s'}$ is a partition of $(V^j, E^j)$.

\subsection{Random walks on harmonic cocycle functions}

In this section, we study random walks on graphs generalizing the graphs $(V^j,  E^j)$ above and derive a sufficient condition for cocycle functions to be ``concentrated in the negative eigenspaces'' of such random walks.  

Let $\Gamma$ be a group generated by the set $S = \lbrace s_1^{\pm},...,s_n^{\pm} \rbrace$ (as above, the set $S$ denotes the formal symmetric alphabet of \(2n\) letters, and reduced words are counted as formal words, even when they represent the same element of \(\Gamma\)).   For $k \in \mathbb{N}$,  as above, we denote by $W_k$ the set of reduced words of length $k$ in the alphabet $S$ and by $W_{s,k}$ the set of reduced words of length $k$ whose first letter is $s$. 

Let $k_0, k_1 \in \mathbb{N}$ be constants.  We define $(k_0, k_1)$-word bipartite multigraphs as follows. A \textit{$(k_0, k_1)$-word bipartite multigraph} is a bipartite multigraph $(V,E)$ with sides $V_0 = \lbrace (0, w) : w \in W_{k_0}  \rbrace,  V_1 =  \lbrace (1, w) : w \in W_{k_1}  \rbrace$ such that for every $s \in S$ and every $w_0 \in W_{s, k_0},  w_1 \in W_{s,k_1}$, we have $\lbrace (0,w_0),  (1, w_1) \rbrace \notin E$.  For $1 \leq j \leq 4$,  the multigraph $(V^j, E^j)$ defined above is an example of an $(a_j, a_{j+1})$-word bipartite multigraph.  However,  the definition of a $(k_0, k_1)$-word bipartite multigraph is general and does not depend on the presentation of $\Gamma$.  

Given a $(k_0, k_1)$-word bipartite multigraph $(V,E)$,   we set $V_{i,s} = \lbrace (i,w) : w \in W_{s, k_i} \rbrace$ for $i=0,1$ and $s \in S$.  Also,  for $s,s' \in S,  s \neq s'$,  we set $V_{s,s'} = V_{0,s} \cup V_{1,s'}$ and define $(V_{s,s'},  E_{s,s'})$ to be the sub-multigraph spanned by the vertex set $
V_{s,s'}$.  We denote by $m_{s,s'}$ the degree function of $(V_{s,s'},  E_{s,s'})$. We note that $\lbrace (V_{s,s'},  E_{s,s'}) \rbrace_{s,s' \in S, s \neq s'}$ is a partition of $(V,E)$.  

Let $\B$ be a real Banach space, and let $c: \Gamma \rightarrow \B$ be a cocycle.  Given a  $(k_0, k_1)$-word bipartite multigraph $(V,E)$,  we define the \textit{cocycle map} $\phi_c : V \rightarrow \B$ by $\phi_c ((i,w)) = c (w)$ for every $i=0,1,  w \in W_{k_i}$.   We also define the \textit{cocycle first-letter map $\psi_c^{\alpha_0, \alpha_1} : V \rightarrow \B$} as follows: for every $i=0,1$,  every $s \in S$ and every $v \in V_{i,s}$,  $\psi_c^{\alpha_0,  \alpha_1} (v) = \alpha_i c (s)$.

Before stating and proving the main result of this section,  we will need some lemmas about $(k_0, k_1)$-word bipartite multigraphs with the norm, projections, and coupling taken with respect to $m_{D_0 ',  D_1 '} : V \rightarrow \mathbb{N}$.  We recall that $m_{D_0 ',  D_1 '}$ was defined above as $m_{D_0 ',  D_1 '} (v) = D_i ',  \forall v \in V_i$.  

\begin{lemma}
\label{lemma: bounding phi-c norm by psi-c norm}
Let $\Gamma$ be as above, let $1 < p < \infty,  k_0, k_1 \in \mathbb{N}$ be constants, and let $\B$ be a real Banach space.  Let $(V,E)$ be a  $(k_0, k_1)$-word bipartite multigraph.  Then for every $D_0 ',  D_1 ' \in \mathbb{N}$ and every cocycle $c : \Gamma \rightarrow \B$, we have 
\[
\Vert \phi_c \Vert^p_{\ell^p (V, m_{D_0 ',  D_1 '} ; \B)} \leq \Vert \psi_c^{k_0, k_1} \Vert^p_{\ell^p (V, m_{D_0 ',  D_1 '} ; \B)}.
\]
\end{lemma} 

\begin{proof}
By the definitions of $m_{D_0 ',  D_1 '}$ and of $\phi_c$,  we have 
\begin{align*}
& \Vert \phi_c \Vert^p_{\ell^p (V, m_{D_0 ', D_1 '} ; \B)} \\
& = \sum_{i=0}^1 \sum_{w \in W_{k_i}} m_{D_0 ' , D_1 '} ((i,w)) \Vert c (w) \Vert^p_{\B} \\
& = \sum_{i=0}^1 \sum_{w \in W_{k_i}} D_i ' \Vert c (w) \Vert^p_{\B} \\
& \leq^{\Cref{prop: bound norm c (w)}} \sum_{i=0}^1 D_i ' (2n-1)^{k_i-1} k_i^p \sum_{s \in S} \Vert c (s) \Vert^p_{\B} \\
& = \sum_{i=0}^1 \sum_{s \in S} D_i ' \vert V_{i,s} \vert \Vert k_i c (s) \Vert^p_{\B} \\
& = \Vert \psi_c^{k_0, k_1} \Vert^p_{\ell^p (V, m_{D_0 ', D_1 '} ; \B)}.
\end{align*}
\end{proof}

We note that $\lbrace V_{i,s} : i =0,1,  s \in S \rbrace$ is a partition of $V$ and refer the reader to the definition of $P_{\lbrace V_{i,s} : i =0,1,  s \in S \rbrace,  m_{D_0 ' ,  D_1 '}}$ given in \Cref{subsection: Vector valued finite measure spaces} above.

\begin{lemma}
\label{lemma: proj on word graph}
Let $\Gamma$ be as above, let $k_0, k_1 \in \mathbb{N}$ be constants, and let $\B$ be a real Banach space.  Let $(V,E)$ be a  $(k_0, k_1)$-word bipartite multigraph.  Then for every $D_0 ',  D_1 ' \in \mathbb{N}$ and every $\phi : V \rightarrow \B$, the following holds for every $i =0,1$,  every $s \in S$ and every $v \in V_{i,s}$:
\[
(P_{\lbrace V_{i,s} : i =0,1,  s \in S \rbrace,  m_{D_0 ' ,  D_1 '}} \otimes I_{\B}) \phi  (v) = \frac{1}{(2n-1)^{k_i-1}} \sum_{w \in W_{s, k_i}} \phi ( (i,w)).
\]
\end{lemma}

\begin{proof}
Fix $i= 0,1,  s \in S$ and $v \in V_{i,s}$.  Then 
\begin{align*}
& (P_{\lbrace V_{i,s} : i =0,1, s \in S \rbrace, m_{D_0 ' , D_1 '}} \otimes I_{\B}) \phi (v) \\
& = \frac{1}{m_{D_0 ', D_1 '} (V_{i,s})} \sum_{u \in V_{i,s}} m_{D_0 ' , D_1 '} (u) \phi (u) \\
& = \frac{1}{D_i ' \vert V_{i,s} \vert} \sum_{w \in W_{s, k_i}} D_i ' \phi ((i,w)) \\
& = \frac{1}{(2n-1)^{k_i-1}} \sum_{w \in W_{s, k_i}} \phi ((i,w)).
\end{align*}
\end{proof}

\begin{lemma}
\label{lemma: proj of phi-c}
Let $\Gamma$ be as above, let $1 < p < \infty,  k_0, k_1 \in \mathbb{N}$ be constants, and let $\B$ be a real Banach space.  Let $(V,E)$ be a  $(k_0, k_1)$-word bipartite multigraph.  Then for every $D_0 ',  D_1 ' \in \mathbb{N}$ and every harmonic cocycle $c : \Gamma \rightarrow \B$, we have 
\[
(P_{\lbrace V_{i,s} : i =0,1,  s \in S \rbrace,  m_{D_0 ' ,  D_1 '}} \otimes I_{\B})  \phi_c  = \psi_c^{\alpha_0,  \alpha_1},
\]
where $\alpha_i = \frac{(2n-1)^{k_i}-1}{(2n-1)^{k_i -1} (2n-2)}$.
\end{lemma}

\begin{proof}
Fix $i=0,1$, $s \in S$ and $v \in V_{i,s}$.  Then
\begin{align*}
& (P_{\lbrace V_{i,s} : i =0,1, s \in S \rbrace, m_{D_0 ' , D_1 '}} \otimes I_{\B}) \phi_c (v) \\
& =^{\Cref{lemma: proj on word graph}} \frac{1}{(2n-1)^{k_i -1}} \sum_{w \in W_{s, k_i}} c (w) \\
& =^{\Cref{prop: sum c (w)}} \frac{(2n-1)^{k_i}-1}{(2n-1)^{k_i -1} (2n-2)} c(s).
\end{align*}
\end{proof}

\begin{lemma}
\label{lemma: psi c coupling with J-p phi c}
Let $\Gamma$ be as above, let $2 \leq p < \infty,  0 < C \leq 1 ,  k_0, k_1 \in \mathbb{N},  \alpha_0 >0 ,  \alpha_1 >0$ be constants, and let $\B$ be a real $(p,C)$-uniformly convex and strictly smooth Banach space.  Let $(V,E)$ be a  $(k_0, k_1)$-word bipartite multigraph.  

Then for every $D_0 ',  D_1 ' \in \mathbb{N}$ and every harmonic cocycle $c : \Gamma \rightarrow \B$, we have 
\[
\langle \psi_c^{\alpha_0, \alpha_1},  J_{\ell^p (V,  m_{D_0 ',  D_1 '} ; \B), p} \phi_c \rangle_{\ell^p (V,  m_{D_0 ',  D_1 '} ; \B)} \geq \frac{2C}{p} \Vert \psi_c^{\sqrt[p]{\alpha_0} , \sqrt[p]{\alpha_1}} \Vert^p_{\ell^p (V, m_{D_0 ',  D_1 '} ; \B)}.
\]
\end{lemma}

\begin{proof}
By \Cref{lemma: proj on word graph}, $(P_{\lbrace V_{i,s} : i =0,1,  s \in S \rbrace,  m_{D_0 ' ,  D_1 '}} \otimes I_{\B})  \psi_c^{\alpha_0, \alpha_1} = \psi_c^{\alpha_0,  \alpha_1}$.  Thus
\begin{align*}
& \langle \psi_c^{\alpha_0, \alpha_1}, J_{\ell^p (V, m_{D_0 ', D_1 '} ; \B), p} \phi_c \rangle_{\ell^p (V, m_{D_0 ', D_1 '} ; \B)} \\
& = \langle (P_{\lbrace V_{i,s} : i =0,1, s \in S \rbrace, m_{D_0 ' , D_1 '}} \otimes I_{\B}) \psi_c^{\alpha_0, \alpha_1}, J_{\ell^p (V, m_{D_0 ', D_1 '} ; \B), p} \phi_c \rangle_{\ell^p (V, m_{D_0 ', D_1 '} ; \B)} \\
& = \langle \psi_c^{\alpha_0, \alpha_1}, (P_{\lbrace V_{i,s} : i =0,1, s \in S \rbrace, m_{D_0 ' , D_1 '}} \otimes I_{\B^*}) J_{\ell^p (V, m_{D_0 ', D_1 '} ; \B), p} \phi_c \rangle_{\ell^p (V, m_{D_0 ', D_1 '} ; \B)} \\
& = \sum_{i=0}^1 D_i' \sum_{s \in S} \sum_{v \in V_{i,s}} \\*
& \qquad \langle \alpha_i c (s), (P_{\lbrace V_{i,s} : i =0,1, s \in S \rbrace, m_{D_0 ' , D_1 '}} \otimes I_{\B^*}) J_{\ell^p (V, m_{D_0 ', D_1 '} ; \B), p} \phi_c (v) \rangle_{\B} \\
& =^{\Cref{lemma: proj on word graph}} \sum_{i=0}^1 D_i' \sum_{s \in S} \sum_{v \in V_{i,s}} \\*
& \qquad \left\langle \alpha_i c (s), \frac{1}{(2n-1)^{k_i-1}} \sum_{w \in W_{s, k_i}} J_{\ell^p (V, m_{D_0 ', D_1 '} ; \B), p} \phi_c ( (i,w)) \right\rangle_{\B} \\
& =^{\Cref{prop: J ell-p}} \sum_{i=0}^1 D_i ' \alpha_i \sum_{s \in S} \sum_{w \in W_{s, k_i}} \langle c (s), J_{\B, p} (c (w)) \rangle_{\B} \\
& \geq^{\Cref{prop: couplings sum ineq}} \frac{2C}{p} \sum_{i=0}^1 D_i ' \alpha_i \frac{(2n-1)^{k_i} -1}{2n-2} \sum_{s \in S} \Vert c (s) \Vert^p_{\B} \\
& \geq \frac{2C}{p} \sum_{i=0}^1 D_i ' (2n-1)^{k_i-1} \alpha_i \sum_{s \in S} \Vert c (s) \Vert^p_{\B} \\
& = \frac{2C}{p} \sum_{i=0}^1 \sum_{s \in S} D_i ' \vert V_{i,s} \vert \ \Vert \sqrt[p]{\alpha_i} c (s) \Vert^p_{\B} \\
& =\frac{2C}{p} \Vert \psi_c^{\sqrt[p]{\alpha_0} , \sqrt[p]{\alpha_1}} \Vert^p_{\ell^p (V, m_{D_0 ', D_1 '} ; \B)}.
\end{align*} 
\end{proof}

We also need consequences of the assumption that all the multigraphs $(V_{s,s'}, E_{s,s'})$ are $(D_0,D_1, \varepsilon)$-almost biregular.

\begin{lemma}
\label{lemma: local almost reg}
Let $\Gamma$ be as above, and let $k_0, k_1 \in \mathbb{N}$ be constants.  Let $(V,E)$ be a  $(k_0, k_1)$-word bipartite multigraph.  Assume that there are $D_0,  D_1 \in \mathbb{N}$ and $0 \leq \varepsilon <1$ such that for every $s, s' \in S,  s \neq s'$,  the multigraph $(V_{s,s'}, E_{s,s'})$ is $(D_0,D_1, \varepsilon)$-almost biregular.  Then
\begin{itemize}
\item The multigraph $(V,E)$ is $((2n-1) D_0,  (2n-1) D_1,  \varepsilon)$-almost biregular.
\item  For every $i=0,1$,  every $s \in S$ and every $v \in V_{i,s}$, the following holds for every $s ' \in S,  s' \neq s$: 
\[
\left\vert \frac{m (v, V_{i+1, s'} )}{m (v)} - \frac{1}{2n-1} \right\vert \leq \frac{2 \varepsilon}{(2n-1) (1 -\varepsilon)},
\]
where the index $i+1$ above is taken modulo $2$.
\end{itemize}
\end{lemma}

\begin{proof}
For the first assertion, note that for every $v \in V$ either $v \in V_{0,s}$ for some $s \in S$ or $v \in V_{1,s'}$.  If $v \in V_{0,s}$, then 
\[
\vert m (v) - (2n-1) D_0 \vert \leq \sum_{s' \in S,  s' \neq s} \vert m_{s,s'} (v) - D_0 \vert \leq \sum_{s' \in S,  s' \neq s} \varepsilon D_0 =  \varepsilon (2n-1) D_0.
\]
A similar computation shows that if $v \in V_{1,s'}$,  then 
\[
\vert m (v) - (2n-1) D_1 \vert  \leq \varepsilon (2n-1) D_1.
\]

For the second assertion,  if $v \in V_{0,s}$, then for every $s' \in S,  s' \neq s$,  we note that $m (v, V_{i+1, s'} ) = m_{s,s'} (v)$ and hence, by our assumption, 
\[
(1- \varepsilon) D_0 \leq m (v, V_{i+1, s'} )  \leq (1+ \varepsilon) D_0.
\]
Also, by the first assertion, 
\[
(1- \varepsilon)(2n-1) D_0 \leq m (v) \leq (1+ \varepsilon)(2n-1) D_0.
\]
It follows that 
\[
\left\vert \frac{m (v, V_{i+1, s'} )}{m (v)} - \frac{1}{2n-1} \right\vert \leq \frac{2 \varepsilon}{(2n-1) (1 -\varepsilon)}.
\]
The proof for $v \in V_{1,s}$ is similar and is left to the reader.
\end{proof}

\begin{lemma}
\label{lemma: M psi c minus psi c}
Let $\Gamma$ be as above, let $2 \leq p < \infty,  k_0, k_1 \in \mathbb{N},    \alpha_0 >0,  \alpha_1 >0$ be constants, let $\B$ be a real Banach space, and let $c : \Gamma \rightarrow \B$ be a harmonic cocycle.  Let $(V,E)$ be a  $(k_0, k_1)$-word bipartite multigraph.  Assume that there are $D_0,  D_1 \in \mathbb{N}$ and $0 \leq \varepsilon <1$ such that for every $s, s' \in S,  s \neq s'$,  the multigraph $(V_{s,s'}, E_{s,s'})$ is $(D_0,D_1, \varepsilon)$-almost biregular.  

Then for every $D_0 ',  D_1 ' \in \mathbb{N}$, we have 
\begin{dmath*}
\left\Vert (M \otimes I_{\B})  \psi_c^{\alpha_0,  \alpha_1} + \frac{1}{2n-1} \psi_c^{\alpha_1, \alpha_0} \right\Vert_{\ell^p (V,  m_{D_0 ' ,  D_1 '} ; \B)} \leq  \frac{2 \varepsilon}{1 -\varepsilon} \Vert \psi_c^{\alpha_1 , \alpha_0} \Vert_{\ell^p (V, m_{D_0 ',  D_1 '} ; \B)}.
\end{dmath*}

In particular,  if $\alpha_{i+1} \leq k_i$ for $i=0,1$,  then 
\begin{dmath*}
\left\Vert (M \otimes I_{\B})  \psi_c^{\alpha_0,  \alpha_1} + \frac{1}{2n-1} \psi_c^{\alpha_1, \alpha_0} \right\Vert_{\ell^p (V,  m_{D_0 ' ,  D_1 '} ; \B)} \leq  \frac{2 \varepsilon}{1 -\varepsilon} \Vert \psi_c^{k_0 , k_1} \Vert_{\ell^p (V, m_{D_0 ',  D_1 '} ; \B)}.
\end{dmath*}
\end{lemma}

\begin{proof}
Fix $i =0,1$,  $s \in S$ and $v \in V_{i,s}$.  Below,  $i+1$ is always taken modulo $2$.  Then 
\begin{align*}
& (M \otimes I_{\B}) \psi_c^{\alpha_0, \alpha_1} (v) \\
& = \frac{1}{m (v)} \sum_{u \in V_{i+1}, u \sim v} m (\lbrace u,v \rbrace) \psi_c^{\alpha_0, \alpha_1} (u) \\
& = \frac{1}{m (v)} \sum_{s' \in S, s' \neq s} \sum_{u \in V_{i+1, s'}, u \sim v} m (\lbrace u,v \rbrace) \alpha_{i+1} c (s') \\
& = \frac{1}{m (v)} \sum_{s' \in S, s' \neq s} m (v, V_{i+1, s'}) \alpha_{i+1} c (s') \\
& = \sum_{s' \in S, s' \neq s} \left( \left( \frac{m (v, V_{i+1, s'})}{m (v)} - \frac{1}{2n-1} \right) \alpha_{i+1} c (s') \right) \\*
& \qquad + \frac{1}{2n-1} \alpha_{i+1} \sum_{s' \in S, s' \neq s} c (s') \\
& =^{c \text{ is harmonic}} \sum_{s' \in S, s' \neq s} \left( \left( \frac{m (v, V_{i+1, s'})}{m (v)} - \frac{1}{2n-1} \right) \alpha_{i+1} c (s') \right) \\*
& \qquad - \frac{1}{2n-1} \alpha_{i+1} c (s) \\
& = \sum_{s' \in S, s' \neq s} \left( \left( \frac{m (v, V_{i+1, s'})}{m (v)} - \frac{1}{2n-1} \right) \alpha_{i+1} c (s') \right) \\*
& \qquad - \frac{1}{2n-1} \psi_c^{\alpha_1, \alpha_0} (v).
\end{align*}
Thus, for $v \in V_{i,s}$,  
\begin{align*}
& \left\Vert (M \otimes I_{\B}) \psi_c^{\alpha_0, \alpha_1} (v) + \frac{1}{2n-1} \psi_c^{\alpha_1, \alpha_0} (v) \right\Vert^p_{\B} \\
& = \left\Vert \sum_{s' \in S, s' \neq s} \left( \left( \frac{m (v, V_{i+1, s'})}{m (v)} - \frac{1}{2n-1} \right) \alpha_{i+1} c (s') \right) \right\Vert^p_{\B} \\
& \leq (2n-1)^{p-1} \sum_{s' \in S, s' \neq s} \left\Vert \left( \frac{m (v, V_{i+1, s'})}{m (v)} - \frac{1}{2n-1} \right) \alpha_{i+1} c (s') \right\Vert^p_{\B} \\
& \leq^{\Cref{lemma: local almost reg}} \left( \frac{2 \varepsilon}{1 -\varepsilon} \right)^p \alpha_{i+1}^p \frac{1}{2n-1} \sum_{s' \in S, s' \neq s} \left\Vert c (s') \right\Vert^p_{\B} .
\end{align*}
It follows that
\begin{align*}
& \left\Vert (M \otimes I_{\B}) \psi_c^{\alpha_0, \alpha_1} + \frac{1}{2n-1} \psi_c^{\alpha_1, \alpha_0} \right\Vert_{\ell^p (V, m_{D_0 ' , D_1 '} ; \B)}^p \\
& \leq \left( \frac{2 \varepsilon}{1 -\varepsilon} \right)^p \sum_{i=0}^1 \sum_{s \in S} \sum_{v \in V_{i,s}} D_i ' \alpha_{i+1}^p \frac{1}{2n-1} \sum_{s' \in S, s' \neq s} \left\Vert c (s') \right\Vert^p_{\B} \\
& = \left( \frac{2 \varepsilon}{1 -\varepsilon} \right)^p \sum_{i=0}^1 (2n-1)^{k_i-1} D_i ' \alpha_{i+1}^p \sum_{s' \in S} \left\Vert c (s') \right\Vert^p_{\B} \sum_{s \in S, s \neq s'} \frac{1}{2n-1} \\
& = \left( \frac{2 \varepsilon}{1 -\varepsilon} \right)^p \sum_{i=0}^1 (2n-1)^{k_i-1} D_i ' \alpha_{i+1}^p \sum_{s' \in S} \left\Vert c (s') \right\Vert^p_{\B} \\
& \leq \left( \frac{2 \varepsilon}{1 -\varepsilon} \right)^p \sum_{i=0}^1 \sum_{s' \in S} D_i' \vert V_{i,s'} \vert \left\Vert \alpha_{i+1} c (s') \right\Vert^p_{\B} \\
& = \left( \frac{2 \varepsilon}{1 -\varepsilon} \right)^p \Vert \psi_c^{\alpha_1 , \alpha_0} \Vert^p_{\ell^p (V, m_{D_0 ', D_1 '} ; \B)}
\end{align*}
and the required inequality follows.
\end{proof}

We can now state and prove the main result of this section:

\begin{theorem}
\label{thm: rw on cocycles main}
Let $\Gamma$ be as above, let $2 \leq p < \infty,  0 < C \leq 1,  k_0, k_1 \in \mathbb{N}, k_0 \geq 2,  k_1 \geq 2$ be constants, and let $\B$ be a real Banach space that is $(p,C)$-uniformly convex and strictly smooth.  Let $(V,E)$ be a  $(k_0, k_1)$-word bipartite multigraph and assume that there are constants $0 \leq \varepsilon < 1,  D_0, D_1 \in \mathbb{N},  0 \leq \lambda$ such that the following conditions hold:
\begin{itemize}
\item For every $s, s' \in S,  s \neq s'$,  the multigraph $(V_{s,s'}, E_{s,s'})$ is $(D_0,D_1, \varepsilon)$-almost biregular.
\item For every $s, s' \in S,  s \neq s'$,  $\lambda_{\bipartite,  p,  \B} (V_{s,s'}, E_{s,s'}) \leq \lambda$.
\end{itemize}
For $i=0, 1$,  we set $D_i ' = (2n-1) D_i $ and $\alpha_i = \frac{(2n-1)^{k_i}-1}{(2n-1)^{k_i -1} (2n-2)}$.
Then for every harmonic cocycle $c : \Gamma \rightarrow \B$, we have 
\begin{align*}
& \langle (M \otimes I_{\B}) \phi_c , J_{\ell^p (V, m ; \B), p} \phi_c \rangle_{\ell^p (V, m ; \B)} \\
& \leq \left( \frac{\varepsilon (5 + 3\varepsilon)}{1 -\varepsilon} + (1+\varepsilon) \lambda \right) \Vert \psi_c^{k_0 , k_1} \Vert_{\ell^p (V, m_{D_0 ', D_1 '} ; \B)}^p \\*
& \qquad - \frac{1}{2n-1} \frac{2C}{p} \Vert \psi_c^{\sqrt[p]{\alpha_1} , \sqrt[p]{\alpha_0}} \Vert^p_{\ell^p (V, m_{D_0 ', D_1 '} ; \B)}.
\end{align*}

In particular, if 
\[
\frac{p}{2C} \left( \frac{\varepsilon (5 + 3\varepsilon)}{1 -\varepsilon}  + (1+\varepsilon) \lambda \right) <  \frac{1}{2n-1}  \min_{i=0,1} \frac{1}{k_i^p},
\]
then for every non-zero harmonic cocycle $c : \Gamma \rightarrow \B$, we have 
\[
\langle (M \otimes I_{\B}) \phi_c ,  J_{\ell^p (V, m ; \B),  p} \phi_c \rangle_{\ell^p (V, m ; \B)} <0.
\]
\end{theorem}

\begin{proof}
Let $P = P_{\lbrace V_{i,s} : i=0,1, s \in S \rbrace,  m_{D_0 ',  D_1 '}}$ be the projection defined in \Cref{subsection: Vector valued finite measure spaces} above.  

With this notation,
\begin{align*}
& \langle (M \otimes I_{\B}) \phi_c , J_{\ell^p (V, m ; \B), p} \phi_c \rangle_{\ell^p (V, m ; \B)} \\
& = \langle (M \otimes I_{\B}) (P \otimes I_{\B}) \phi_c , J_{\ell^p (V, m ; \B), p} \phi_c \rangle_{\ell^p (V, m ; \B)} \\*
& \qquad + \langle (M \otimes I_{\B}) ((I-P) \otimes I_{\B}) \phi_c , J_{\ell^p (V, m ; \B), p} \phi_c \rangle_{\ell^p (V, m ; \B)}.
\end{align*}

We will show that 
\begin{equation}
\label{thm ineq 1}
\begin{split}
& \langle (M \otimes I_{\B}) (P \otimes I_{\B}) \phi_c ,  J_{\ell^p (V, m ; \B),  p} \phi_c \rangle_{\ell^p (V, m ; \B)}  \leq \\
& \left( \frac{\varepsilon (3 +\varepsilon)}{1 -\varepsilon} \right) \Vert \psi_c^{k_0 , k_1} \Vert_{\ell^p (V, m_{D_0 ',  D_1 '} ; \B)}^p \\*
& \qquad - \frac{1}{2n-1} \frac{2C}{p} \Vert \psi_c^{\sqrt[p]{\alpha_1} , \sqrt[p]{\alpha_0}} \Vert^p_{\ell^p (V, m_{D_0 ',  D_1 '} ; \B)}.
\end{split}
\end{equation}
and that 
\begin{equation}
\label{thm ineq 2}
\begin{aligned}
& \langle (M \otimes I_{\B}) ((I-P) \otimes I_{\B}) \phi_c , J_{\ell^p (V, m ; \B), p} \phi_c \rangle_{\ell^p (V, m ; \B)} \\
& \leq \left( \frac{2 \varepsilon}{1- \varepsilon} + \lambda \right) (1 + \varepsilon) \Vert \psi_c^{k_0, k_1} \Vert_{\ell^p (V, m_{D_0 ' , D_1 '} ; \B)}^p.
\end{aligned}
\end{equation}

\paragraph*{\textbf{Proving inequality  \eqref{thm ineq 1}}}
By \Cref{prop: J ell-p},  
\[
J_{\ell^p (V, m ; \B),  p} \phi_c = J_{\ell^p (V, m_{D_0 ' , D_1 '} ; \B),  p} \phi_c,
\]
so
\begin{align*}
& \langle (M \otimes I_{\B}) (P \otimes I_{\B}) \phi_c , J_{\ell^p (V, m ; \B), p} \phi_c \rangle_{\ell^p (V, m ; \B)} \\
& = \langle (M \otimes I_{\B}) (P \otimes I_{\B}) \phi_c , J_{\ell^p (V, m_{D_0 ', D_1 '} ; \B), p} \phi_c \rangle_{\ell^p (V, m ; \B)} .
\end{align*}
Since $(V,E)$ is $(D_0 ',  D_1 ',  \varepsilon)$-almost biregular (by \Cref{lemma: local almost reg}), it follows that
\begin{align*}
& \langle (M \otimes I_{\B}) (P \otimes I_{\B}) \phi_c , J_{\ell^p (V, m_{D_0 ', D_1 '} ; \B), p} \phi_c \rangle_{\ell^p (V, m ; \B)} \\*
& \qquad - {\langle (M \otimes I_{\B}) (P \otimes I_{\B}) \phi_c , J_{\ell^p (V, m_{D_0 ', D_1 '} ; \B), p} \phi_c \rangle_{\ell^p (V, m_{D_0 ' , D_1 '} ; \B)} } \\
& \leq^{\Cref{prop: norm and coupling comparison with almost bireg}} \frac{\varepsilon}{1- \varepsilon} { \begin{aligned}[t]
& \Vert (M \otimes I_{\B}) (P \otimes I_{\B}) \phi_c \Vert_{\ell^p (V, m ; \B)} \\*
& \qquad \Vert J_{\ell^p (V, m_{D_0 ', D_1 '} ; \B), p} \phi_c \Vert_{\ell^q (V, m ; \B^*)}
\end{aligned} } \\
& \leq^{M \text{ is a Markov operator preserving } m} \frac{\varepsilon}{1- \varepsilon} \Vert (P \otimes I_{\B}) \phi_c \Vert_{\ell^p (V, m ; \B)} \\*
& \qquad \Vert \phi_c \Vert_{\ell^p (V, m ; \B)}^{p-1} \\
& \leq^{\Cref{prop: norm and coupling comparison with almost bireg}} \frac{\varepsilon (1+\varepsilon)}{1- \varepsilon} \Vert (P \otimes I_{\B}) \phi_c \Vert_{\ell^p (V, m_{D_0 ',D_1 '} ; \B)} \\*
& \qquad \Vert \phi_c \Vert_{\ell^p (V, m_{D_0 ',D_1 '} ; \B)}^{p-1} \\
& \leq^{P \text{ is a Markov operator preserving } m_{D_0 ', D_1'}} \frac{\varepsilon (1 +\varepsilon)}{1- \varepsilon} \Vert \phi_c \Vert_{\ell^p (V, m_{D_0', D_1 '} ; \B)}^p \\
& \leq^{\Cref{lemma: bounding phi-c norm by psi-c norm}} \frac{\varepsilon (1 +\varepsilon)}{1- \varepsilon} \Vert \psi_c^{k_0,k_1} \Vert_{\ell^p (V, m_{D_0 ', D_1 '} ; \B)}^p.
\end{align*}
It therefore remains to prove that 
\begin{align*}
& \langle (M \otimes I_{\B}) (P \otimes I_{\B}) \phi_c , J_{\ell^p (V, m ; \B), p} \phi_c \rangle_{\ell^p (V, m_{D_0 ' , D_1 '} ; \B)} \\
& \leq \frac{2 \varepsilon}{1 -\varepsilon} \Vert \psi_c^{k_0 , k_1} \Vert_{\ell^p (V, m_{D_0 ', D_1 '} ; \B)}^p \\*
& \qquad - \frac{1}{2n-1} \frac{2C}{p} \Vert \psi_c^{\sqrt[p]{\alpha_1} , \sqrt[p]{\alpha_0}} \Vert^p_{\ell^p (V, m_{D_0 ', D_1 '} ; \B)}.
\end{align*}

By \Cref{lemma: proj of phi-c},  $(P \otimes I_{\B}) \phi_c = \psi_c^{\alpha_0,  \alpha_1}$.  Also, our assumption that $k_0, k_1 \geq 2$ implies that $k_i \geq \alpha_{i+1}$. Thus
\begin{align*}
& \langle (M \otimes I_{\B}) (P \otimes I_{\B}) \phi_c , J_{\ell^p (V, m ; \B), p} \phi_c \rangle_{\ell^p (V, m_{D_0 ' , D_1 '} ; \B)} \\
& = \langle (M \otimes I_{\B}) \psi_c^{\alpha_0, \alpha_1} , J_{\ell^p (V, m ; \B), p} \phi_c \rangle_{\ell^p (V, m_{D_0 ' , D_1 '} ; \B)} \\
& = \langle (M \otimes I_{\B}) \psi_c^{\alpha_0, \alpha_1} + \frac{1}{2n-1} \psi_c^{\alpha_1, \alpha_0} , J_{\ell^p (V, m ; \B), p} \phi_c \rangle_{\ell^p (V, m_{D_0 ' , D_1 '} ; \B)} \\*
& \qquad - \frac{1}{2n-1} \langle \psi_c^{\alpha_1, \alpha_0} , J_{\ell^p (V, m ; \B), p} \phi_c \rangle_{\ell^p (V, m_{D_0 ' , D_1 '} ; \B)} \\
& \leq^{\Cref{lemma: psi c coupling with J-p phi c}} { \left\Vert (M \otimes I_{\B}) \psi_c^{\alpha_0, \alpha_1} + \frac{1}{2n-1} \psi_c^{\alpha_1, \alpha_0} \right\Vert_{\ell^p (V, m_{D_0 ' , D_1 '} ; \B)} } \\*
& \qquad \left\Vert J_{\ell^p (V, m ; \B), p} \phi_c \right\Vert_{\ell^q (V, m_{D_0 ', D_1 '}; \B^*)} \\*
& \qquad - \frac{1}{2n-1} \frac{2C}{p} \Vert \psi_c^{\sqrt[p]{\alpha_1} , \sqrt[p]{\alpha_0}} \Vert^p_{\ell^p (V, m_{D_0 ', D_1 '} ; \B)} \\
& \leq^{\Cref{lemma: M psi c minus psi c}, k_i \geq \alpha_{i+1}} \frac{2 \varepsilon}{1 -\varepsilon} \Vert \psi_c^{k_0 , k_1} \Vert_{\ell^p (V, m_{D_0 ', D_1 '} ; \B)} \Vert \phi_c \Vert^{p-1}_{\ell^p (V, m_{D_0 ', D_1 '} ; \B)} \\*
& \qquad - \frac{1}{2n-1} \frac{2C}{p} \Vert \psi_c^{\sqrt[p]{\alpha_1} , \sqrt[p]{\alpha_0}} \Vert^p_{\ell^p (V, m_{D_0 ', D_1 '} ; \B)} \\
& \leq^{\Cref{lemma: bounding phi-c norm by psi-c norm}} \frac{2 \varepsilon}{1 -\varepsilon} \Vert \psi_c^{k_0 , k_1} \Vert_{\ell^p (V, m_{D_0 ', D_1 '} ; \B)}^p \\*
& \qquad - \frac{1}{2n-1} \frac{2C}{p} \Vert \psi_c^{\sqrt[p]{\alpha_1} , \sqrt[p]{\alpha_0}} \Vert^p_{\ell^p (V, m_{D_0 ', D_1 '} ; \B)}.
\end{align*}
This computation completes the proof of \eqref{thm ineq 1}.

\paragraph*{\textbf{Proving inequality  \eqref{thm ineq 2}}}

We first prove 
\begin{equation}
\label{thm ineq 3}
\Vert (M \otimes I_{\B}) ((I-P) \otimes I_{\B}) \phi_c \Vert_{\ell^p (V, m ; \B)} \leq \left( \frac{2 \varepsilon}{1- \varepsilon} + \lambda \right)   \Vert  \phi_c \Vert_{\ell^p (V, m ; \B)}.
\end{equation}
For $s,s ' \in S,  s \neq s'$,  let $M_{s,s'}$ denote the random walk operator on the sub-multigraph $(V_{s,s'}, E_{s, s'})$.  We note that the sub-multigraphs $\lbrace (V_{s,s'}, E_{s, s'}) : s,s ' \in S,  s \neq s' \rbrace$ form a partition of the multigraph $(V,E)$.  Thus,  by \Cref{prop: rw bound on sub-multigraph} with $\phi ' = ((I-P) \otimes I_{\B}) \phi_c$ and $\phi = \phi_c$,  to prove \eqref{thm ineq 3}, it suffices to show that, for every $s,s ' \in S,  s \neq s'$, 
\begin{align*}
& \Vert (M_{s,s'} \otimes I_{\B}) \left. \left( ((I-P) \otimes I_{\B}) \phi_c \right) \right\vert_{V_{s,s'}} \Vert_{\ell^p (V_{s,s'}, m_{s,s'} ; \B)} \\
& \leq \left( \frac{2 \varepsilon}{1- \varepsilon} + \lambda \right) \Vert \left. \phi_c \right\vert_{V_{s,s'}} \Vert_{\ell^p (V_{s,s'}, m_{s,s'} ; \B)}.
\end{align*}

The definition of $P$ gives 
\begin{align*}
& \left. \left( (P \otimes I_{\B}) \phi_c \right) \right\vert_{V_{s,s'}} \\
& = (P_{\lbrace V_{0,s}, V_{1,s'} \rbrace, m_{D_0 ', D_1 '}} \otimes I_{\B} ) \left( \left. \phi_c \right\vert_{V_{s,s'}} \right) \\
& = (P_{\lbrace V_{0,s}, V_{1,s'} \rbrace, m_{D_0 , D_1 }} \otimes I_{\B} ) \left( \left. \phi_c \right\vert_{V_{s,s'}} \right).
\end{align*}

Thus, for every $s,s' \in S, s \neq s'$,
\begin{align*}
& \Vert (M_{s,s'} \otimes I_{\B}) \left. \left( ((I-P) \otimes I_{\B}) \phi_c \right) \right\vert_{V_{s,s'}} \Vert_{\ell^p (V_{s,s'}, m_{s,s'} ; \B)} \\
& = \Vert (M_{s,s'} \otimes I_{\B}) ((I-P_{\lbrace V_{0,s}, V_{1,s'} \rbrace, m_{D_0 , D_1 }}) \otimes I_{\B}) \left( \left. \phi_c \right\vert_{V_{s,s'}} \right) \Vert_{\ell^p (V_{s,s'}, m_{s,s'} ; \B)} \\
& \leq^{\Cref{lemma: rw of almost bireg}} \left( \frac{2 \varepsilon}{1- \varepsilon} + \lambda \right) \Vert \left. \phi_c \right\vert_{V_{s,s'}} \Vert_{\ell^p (V_{s,s'}, m_{s,s'} ; \B)}.
\end{align*}

Therefore
\begin{align*}
& \langle (M \otimes I_{\B}) ((I-P) \otimes I_{\B}) \phi_c , J_{\ell^p (V, m ; \B), p} \phi_c \rangle_{\ell^p (V, m ; \B)} \\
& \leq^{\eqref{thm ineq 3}} \left( \frac{2 \varepsilon}{1- \varepsilon} + \lambda \right) \Vert \phi_c \Vert_{\ell^p (V, m ; \B)}^p \\
& \leq^{\Cref{prop: norm and coupling comparison with almost bireg}, \Cref{lemma: local almost reg}} \left( \frac{2 \varepsilon}{1- \varepsilon} + \lambda \right) (1 + \varepsilon) \Vert \phi_c \Vert_{\ell^p (V, m_{D_0 ' , D_1 '} ; \B)}^p \\
& \leq^{\Cref{lemma: bounding phi-c norm by psi-c norm}} \left( \frac{2 \varepsilon}{1- \varepsilon} + \lambda \right) (1 + \varepsilon) \Vert \psi_c^{k_0, k_1} \Vert_{\ell^p (V, m_{D_0 ' , D_1 '} ; \B)}^p.
\end{align*}
as needed. 

Combining \eqref{thm ineq 1} and \eqref{thm ineq 2} proves the inequality 
\begin{align*}
& \langle (M \otimes I_{\B}) \phi_c , J_{\ell^p (V, m ; \B), p} \phi_c \rangle_{\ell^p (V, m ; \B)} \\
& \leq \left( \frac{\varepsilon (5 + 3\varepsilon)}{1 -\varepsilon} + (1+\varepsilon) \lambda \right) \Vert \psi_c^{k_0 , k_1} \Vert_{\ell^p (V, m_{D_0 ', D_1 '} ; \B)}^p \\*
& \qquad - \frac{1}{2n-1} \frac{2C}{p} \Vert \psi_c^{\sqrt[p]{\alpha_1} , \sqrt[p]{\alpha_0}} \Vert^p_{\ell^p (V, m_{D_0 ', D_1 '} ; \B)}.
\end{align*}

Assume now that 
\[
\frac{p}{2C} \left( \frac{\varepsilon (5 + 3\varepsilon)}{1 -\varepsilon} + (1+\varepsilon) \lambda \right) <  \frac{1}{2n-1}  \min_{i=0,1} \frac{1}{k_i^p}.
\]
For a non-zero harmonic cocycle $c$, it follows that
\begin{align*}
& \left( \frac{\varepsilon (5 + 3\varepsilon)}{1 -\varepsilon} + (1+\varepsilon) \lambda \right) \Vert \psi_c^{k_0 , k_1} \Vert_{\ell^p (V, m_{D_0 ', D_1 '} ; \B)}^p \\*
& \qquad - \frac{1}{2n-1} \frac{2C}{p} \Vert \psi_c^{\sqrt[p]{\alpha_1} , \sqrt[p]{\alpha_0}} \Vert^p_{\ell^p (V, m_{D_0 ', D_1 '} ; \B)} \\
& \leq^{\alpha_0, \alpha_1 \geq 1} \sum_{i=0}^1 \left(\left( \frac{\varepsilon (5 + 3\varepsilon)}{1 -\varepsilon} + (1+\varepsilon) \lambda \right) k_i^p - \frac{1}{2n-1} \frac{2C}{p} \right) \sum_{s \in S} D_i ' \vert V_{i,s} \vert \Vert c(s) \Vert^p \\
& <0,
\end{align*}
and thus 
\[
\langle (M \otimes I_{\B}) \phi_c ,  J_{\ell^p (V, m ; \B),  p} \phi_c \rangle_{\ell^p (V, m ; \B)} < 0.
\]
\end{proof}

\subsection{Fixed point criterion and instantiations}
\label{subsection: Fixed point criterion and instantiations}

Let $\Gamma = \langle S \vert R \rangle$ be a group such that $S = \lbrace s_1^{\pm},...,s_n^{\pm} \rbrace$ and all the relators in $R$ are cyclically reduced in the alphabet $S$.  The set $S$ denotes the formal symmetric alphabet of \(2n\) letters, and reduced words are counted as formal words, even when they represent the same element of \(\Gamma\).  We also assume that there is $\RelLen \in \mathbb{N},  \RelLen \geq 8$ such that all the relations in $R$ are of length $\RelLen$ (with respect to the alphabet $S$). Also, let $(V^{j},  E^{j}) = (V^{j} (R),  E^{j} (R)), j=1,...,4 $ be the graphs associated to the presentation, as defined in \Cref{subsection: Graphs associated to the group presentation} above. 

\begin{theorem}
\label{thm: general fixed point criterion}
Let $\Gamma$ be as above, and let $2 \leq p < \infty,  0 < C \leq 1$ be constants. Also, let $\mathcal{E}$ be a class of $(p,C)$-uniformly convex and uniformly smooth real Banach spaces that is closed under ultraproducts. 

Assume that there are constants $0 \leq \varepsilon < 1,  D_0^j, D_1^j \in \mathbb{N}, j=1,...4,  0 \leq \lambda$ such that the following conditions hold:
\begin{itemize}
\item For every $s, s' \in S,  s \neq s'$ and every $j=1,...,4$,  the multigraph $(V_{s,s'}^j, E_{s,s'}^j)$ is $(D_0^j,D_1^j, \varepsilon)$-almost biregular.
\item For every $s, s' \in S,  s \neq s'$, every $j=1,...,4$ and every $\B \in \mathcal{E}$,  $\lambda_{\bipartite,  p,  \B} (V_{s,s'}^j, E_{s,s'}^j) \leq \lambda$.
\end{itemize}

If
\[
\frac{p}{2C} \left( \frac{\varepsilon (5 + 3\varepsilon)}{1 -\varepsilon} + (1+\varepsilon)\lambda \right) <  \frac{1}{2n-1}  \frac{1}{\lceil \frac{\RelLen}{4} \rceil^p},
\]
then $\Gamma$ has property $(F_{\mathcal{E}})$.
\end{theorem}

\begin{proof}
By \Cref{prop: general Zuk}, it suffices to show that, for every $j=1,...,4$ and every non-zero harmonic cocycle $c$, 
\[
\langle (M_j \otimes I_{\B}) \phi_c^j,  J_{\ell^p (V^j,  m^j ; \B),p} \phi_c^j \rangle_{\ell^p (V^j,  m^j ; \B)} < 0.
\]
We will show that this follows from \Cref{thm: rw on cocycles main}.

We note that for every $j=1,...,4$, the multigraph $(V^{j},  E^{j})$ is a $(k_0, k_1)$-word bipartite multigraph with $k_0, k_1 \in \lbrace \lfloor \frac{\RelLen}{4} \rfloor ,\lceil \frac{\RelLen}{4} \rceil \rbrace$. Since $\RelLen \geq 8$, we have $k_0, k_1 \geq 2$. Thus the conditions of \Cref{thm: rw on cocycles main} are satisfied for every $(V^{j},  E^{j}), j=1,...,4$ and, for every non-zero harmonic cocycle $c : \Gamma \rightarrow \B$, we obtain 
\[
\langle (M_j \otimes I_{\B}) \phi_c^j,  J_{\ell^p (V^j,  m^j ; \B),p} \phi_c^j \rangle_{\ell^p (V^j,  m^j ; \B)} < 0
\]
as needed.
\end{proof}

We now give a few instances of this criterion in which fixed point properties can be deduced by bounding $\lambda (V^j, E^j)$:
\begin{theorem}
\label{thm: Hilbert and Lp fixed point criterion}
Let $\Gamma$ be as above. Assume that there are constants $0 \leq \varepsilon < 1,  D_0^j, D_1^j \in \mathbb{N}, j=1,...4,  0 \leq \lambda$ such that the following conditions hold:
\begin{itemize}
\item For every $s, s' \in S,  s \neq s'$ and every $j=1,...,4$,  the multigraph $(V_{s,s'}^j, E_{s,s'}^j)$ is $(D_0^j,D_1^j, \varepsilon)$-almost biregular.
\item For every $s, s' \in S,  s \neq s'$ and every $j=1,...,4$,  $\lambda (V_{s,s'}^j, E_{s,s'}^j) \leq \lambda$.
\end{itemize}
The following conclusions hold:
\begin{enumerate}
\item If 
\[
8 \left( \frac{\varepsilon (5 + 3\varepsilon)}{1 -\varepsilon} + (1+\varepsilon) \lambda \right) <  \frac{1}{2n-1}  \frac{1}{\lceil \frac{\RelLen}{4} \rceil^2},
\]
then $\Gamma$ has property $(T)$.
\item For $2 \leq p < \infty$, if 
\[
2^{p-1} p^2  \left( \frac{\varepsilon (5 + 3\varepsilon)}{1 -\varepsilon} + (1+\varepsilon) 2 \lambda^{\frac{2}{p}} \right) <  \frac{1}{2n-1}  \frac{1}{\lceil \frac{\RelLen}{4} \rceil^p},
\]
then $\Gamma$ has property $(F_{L^{p'}})$ for every $p' \in [1,p]$.
\end{enumerate}
\end{theorem}

\begin{proof}
For the first assertion, let $\mathcal{E}_{\text{Hilbert}}$ be the class of all real Hilbert spaces. This class is closed under ultraproducts, and every $\mathcal{H} \in \mathcal{E}_{\text{Hilbert}}$ is $(2, \frac{1}{8})$-uniformly convex (by \Cref{prop: p-C-uni conv for Hilbert and Lp}) and uniformly smooth. Also, by \Cref{prop: lambda Hilbert equal lambda}, for every $\mathcal{H} \in \mathcal{E}_{\text{Hilbert}}$, every $j=1,...,4$ and every $s,s' \in S, s \neq s'$, 
\[
\lambda_{\bipartite, 2, \mathcal{H}} (V^j_{s,s'}, E^j_{s,s'}) \leq \lambda (V^j_{s,s'}, E^j_{s,s'}) \leq \lambda.
\] 
Thus, the conditions of \Cref{thm: general fixed point criterion} are satisfied, and $\Gamma$ has property $(F_{\mathcal{H}})$ for every real Hilbert space, or, equivalently, $\Gamma$ has property $(T)$.

For the second assertion, we note that by the first assertion and \cite[Theorem 1.3]{BFGM} and \cite[Corollary D]{BGM}, $\Gamma$ has property $(F_{L^{p'}})$ for every $p' \in [1,2]$. Let $p' \in [2, p]$ and let $\mathcal{E}_{L^{p'}}$ be the class of all real $L^{p'}$ spaces. This class is closed under ultraproducts, and every $\B \in \mathcal{E}_{L^{p'}}$ is $(p', \frac{1}{p' 2^{p'}})$-uniformly convex (by \Cref{prop: p-C-uni conv for Hilbert and Lp}) and uniformly smooth. Also, by \Cref{prop: lambda Lp bound}, for every $\B \in \mathcal{E}_{L^{p'}}$, every $j=1,...,4$ and every $s,s' \in S, s \neq s'$, 
\[
\lambda_{\bipartite, p', \B} (V^j_{s,s'}, E^j_{s,s'}) \leq 2 (\lambda (V^j_{s,s'}, E^j_{s,s'}))^{\frac{2}{p'}} \leq 2 \lambda^{\frac{2}{p'}}
\]
(the last inequality is justified by the assumptions of the theorem, which imply that $\lambda <1$).

Since $2 \leq p ' \leq p$, we have 
\[
2^{p'-1} (p')^2  \left( \frac{\varepsilon (5 + 3\varepsilon)}{1 -\varepsilon} + (1+\varepsilon) 2 \lambda^{\frac{2}{p'}} \right) <  \frac{1}{2n-1}  \frac{1}{\lceil \frac{\RelLen}{4} \rceil^{p'}}.
\]
Thus, the conditions of \Cref{thm: general fixed point criterion} are satisfied, and $\Gamma$ has property $(F_{L^{p'}})$.     
\end{proof}

We also give a more general instance of the criterion using the following class of Banach spaces. Given constants $2 \leq p < \infty, 0 < C \leq 1$ and a monotone increasing function $\ucurvfun: [0,1] \rightarrow  [0,1]$ such that $\ucurvfun (0) = 0$ and $\lim_{t \rightarrow 0^{+}} \ucurvfun (t) = 0$, we define $\mathcal{E} (p,C,\ucurvfun)$ to be the class of all real Banach spaces satisfying the following conditions: every $\B \in \mathcal{E} (p,C,\ucurvfun)$ is
\begin{itemize}
\item $(p,C)$-uniformly convex;
\item uniformly curved with $\B \in \mathcal{E}^{\text{u-curved}}_\ucurvfun$; and
\item uniformly smooth.
\end{itemize}

\begin{theorem}
\label{thm: uni-curv fixed point criterion}
Let $\Gamma$ be as above. Also, let $2 \leq p < \infty, 0 < C \leq 1$ be constants, and let $\ucurvfun: [0,1] \rightarrow  [0,1]$ be a monotone increasing function such that $\ucurvfun (0) =0$ and $\lim_{t \rightarrow 0^{+}} \ucurvfun (t) = 0$.
Assume that there are constants $0 \leq \varepsilon < 1,  D_0^j, D_1^j \in \mathbb{N}, j=1,...4,  0 < \lambda \leq 1$ such that the following conditions hold:
\begin{itemize}
\item For every $s, s' \in S,  s \neq s'$ and every $j=1,...,4$,  the multigraph $(V_{s,s'}^j, E_{s,s'}^j)$ is $(D_0^j,D_1^j, \varepsilon)$-almost biregular.
\item For every $s, s' \in S,  s \neq s'$ and every $j=1,...,4$,  $\lambda (V_{s,s'}^j, E_{s,s'}^j) \leq \lambda$.
\end{itemize}
If 
\[
\frac{p}{2C} \left( \frac{\varepsilon (5 + 3\varepsilon)}{1 -\varepsilon} + (1+\varepsilon) 4 \left( \ucurvfun (\lambda) \right)^{\frac{2}{p}} \right) <  \frac{1}{2n-1}  \frac{1}{\lceil \frac{\RelLen}{4} \rceil^p},
\]
then $\Gamma$ has property $(F_{\mathcal{E} (p,C,\ucurvfun )})$.
\end{theorem}

\begin{proof}
Given a function $\delta: (0,2] \rightarrow (0,1]$, we define $\mathcal{E} (p,C,\ucurvfun, \delta) \subseteq \mathcal{E} (p,C,\ucurvfun)$ to be the class of all $\B \in \mathcal{E} (p,C,\ucurvfun)$ such that $\delta_{\B^*} \geq \delta$. We note that by definition 
\[
\mathcal{E} (p,C,\ucurvfun) = \bigcup_{\delta :  (0,2] \rightarrow (0,1]} \mathcal{E} (p,C,\ucurvfun, \delta)
\]
and thus, under the hypotheses of the theorem, it suffices to show for every $\delta: (0,2] \rightarrow (0,1]$ that $\Gamma$ has property $(F_{\mathcal{E} (p,C,\ucurvfun, \delta)})$. 

Fix $\delta : (0,2] \rightarrow (0,1]$. We note that the class $\mathcal{E} (p,C,\ucurvfun, \delta)$ is closed under ultraproducts: 
\begin{itemize}
\item For any $\delta_0 : (0,2] \rightarrow (0,1]$ and any Banach spaces $\lbrace \B_n \rbrace_{n \in \mathbb{N}}$ such that for every $i$ and every $0 < \varepsilon \leq 2$, $\delta_{\B_i} (\varepsilon) \geq \delta_0 (\varepsilon)$, we have $\delta_{\lbrace \B_n \rbrace_{\mathcal{U}}} (\varepsilon) \geq \delta_0 (\varepsilon), \forall \varepsilon \in (0,2]$. In particular, if every $\B_n$ is $(p,C)$-uniformly convex, then so is $\lbrace \B_n \rbrace_{\mathcal{U}}$.
\item We have shown that $\lbrace \B_n \rbrace_{\mathcal{U}}$ is uniformly convex and hence reflexive. It follows that the isometric embedding $\lbrace \B_n^* \rbrace_{\mathcal{U}} \rightarrow \lbrace \B_n \rbrace_{\mathcal{U}}^*$ is surjective. By assumption, for every $n$,  $\delta_{\B_n^*} (\varepsilon) \geq \delta (\varepsilon), \forall \varepsilon \in (0,2]$, and thus the same holds for $ \lbrace \B_n \rbrace_{\mathcal{U}}^*$.
\item By \Cref{prop: uni-curv is closed under ultraproducts}, the class $\mathcal{E}^{\text{u-curved}}_\ucurvfun$ is closed under ultraproducts.
\end{itemize}
By \Cref{prop: lambda ineq of uni-curved}, for every $\B \in  \mathcal{E}^{\text{u-curved}}_\ucurvfun$, every $j=1,...,4$ and every $s,s' \in S, s \neq s'$, we have 
\[
\lambda_{\bipartite,  p, \B} (V^j_{s,s'},E^j_{s,s'}) \leq 4  \left( \ucurvfun (\lambda (V^j_{s,s'},E^j_{s,s'})) \right)^{\frac{2}{p}} \leq 4  \left( \ucurvfun (\lambda)  \right)^{\frac{2}{p}}.
\]
Thus the conditions of \Cref{thm: general fixed point criterion} are satisfied, and $\Gamma$ has property $(F_{\mathcal{E} (p,C,\ucurvfun, \delta)})$, as required.
\end{proof}

\section{Fixed point properties for random groups}
\label{section: Fixed point properties for random groups}

The binomial model provides an alternative to the density model mentioned in the introduction:
\begin{definition}[Gromov binomial model]
Let $n \in \mathbb{N}, n \geq 2$ be a constant, let $\RelLen \in \mathbb{N}$ be a parameter, and let $\varrho : \mathbb{N} \rightarrow (0,1)$ be a function.  A random group in the Gromov binomial model $\mathcal{B} (n,\RelLen,\varrho)$ is a group $\Gamma = \langle S  \vert R \rangle$, where $S = \lbrace s_1^{\pm},...,s_n^{\pm} \rbrace$ and $R$ is a set of relators of length $\RelLen$ (in $S$) chosen as follows: $R$ is obtained by selecting each cyclically reduced word of length $\RelLen$ independently with probability \(\varrho(\RelLen)\).

For a group property $P$, we say that $P$ holds asymptotically almost surely (a.a.s.) in $\mathcal{B} (n,\RelLen,\varrho)$ if 
\[
\lim_{\RelLen \rightarrow \infty} \mathbb{P} (\Gamma \in \mathcal{B} (n,\RelLen,\varrho) \text{ has property } P) = 1.
\]
\end{definition}

A group property (P) is called \textit{monotone} if it is preserved under quotients. We note that for every Banach space $\B$ and every class of Banach spaces $\mathcal{E}$, the properties $(F_{\B})$ and $(F_{\mathcal{E}})$ are both monotone.

The following proposition transfers monotone properties from the binomial model to the density model:

\begin{proposition}\cite[Proposition 10.2]{DrutuM}
\label{prop: DrutuM}
Fix $n\geq2$. For each $\RelLen\geq2$, let $(P_{\RelLen})$ be a property of
presentations on $n$ generators with cyclically reduced relators
of length $\RelLen$. Assume that, for each $\RelLen$, property $(P_{\RelLen})$ is
preserved under adding relators of length $\RelLen$.

Let $0<d_0<1$. Suppose that, for every $d\in(d_0,1)$, a random
presentation in $\mathcal{B}(n,\RelLen,\varrho_d)$, where
\[
\varrho_d(\RelLen)=(2n-1)^{-(1-d)\RelLen},
\]
has property $(P_{\RelLen})$ a.a.s.
Then, for every $d\in(d_0,1)$, a random presentation in
$\mathcal{D}(n,\RelLen,d)$ has property $(P_{\RelLen})$ a.a.s.
\end{proposition}

Below, we prove that various monotone fixed point properties hold a.a.s.\ for any $ d > \frac{1}{4}$ and for $\Gamma \in \mathcal{B} (n,\RelLen,\varrho)$ with $\varrho (\RelLen)  = (2n-1)^{-(1-d) \RelLen}$ and then use the preceding proposition to deduce these properties in the density model. We will do so by analysing the graphs $(V^{j}, E^{j}), j=1,...,4$ associated to $\Gamma$, as defined in \Cref{subsection: Graphs associated to the group presentation}, and showing that they meet the conditions of \Cref{thm: Hilbert and Lp fixed point criterion} and of \Cref{thm: uni-curv fixed point criterion} under a mild condition on the decay rate of the function controlling uniform curvature.

We first need some results about bipartite Erd\H{o}s--R\'{e}nyi random graphs, which we now define. Let $N_0,N_1 \in \mathbb{N}$ and $\eta \in (0,1)$. The bipartite Erd\H{o}s--R\'{e}nyi random graph denoted $G (N_0, N_1, \eta)$ is a bipartite graph with vertex sets $V_0, V_1$ such that $\vert V_i \vert = N_i$, constructed as follows: for every $v \in V_0, u \in V_1$, the edge $\lbrace v,u \rbrace$ is included in $G (N_0, N_1, \eta)$ independently with probability $\eta$. 

\begin{proposition}
\label{prop: bireg for ER bipartite graphs}
Let $N \in \mathbb{N}$, $N_0, N_1 \in \mathbb{N}$ and let $C \geq 1$ be a constant such that $N \leq N_0,N_1 \leq C N$. Also, let $\eta : \mathbb{N} \rightarrow (0,1)$ be a function for which there exists $0<\epsilon <1$ such that $\eta (N) = \Omega (N^{\epsilon -1})$. Then, as $N$ tends to infinity, the probability that $G \sim G (N_0,N_1, \eta)$ is $(\lfloor N_1 \eta \rfloor,\lfloor N_0 \eta \rfloor,  O (N^{- \frac{\epsilon}{4}}))$-almost biregular tends to $1$.   
\end{proposition}

\begin{proof}
Let $G = (V,E) \sim G (N_0,N_1, \eta)$ and let $m$ be the degree function of $G$. We note that for every $v \in V_0$, $m(v)$ has the binomial distribution $m (v) \sim B (N_1, \eta)$. Thus, for every $v \in V_0$, the Chernoff bound and the assumption that $\eta (N) = \Omega (N^{\epsilon -1})$ give
\[
\mathbb{P} \left( \vert m (v) - N_1 \eta \vert \geq (N_1 \eta)^{\frac{3}{4}} \right) \leq 2 \exp \left( \frac{-(N_1 \eta )^{\frac{1}{2}} }{3} \right) = \exp \left( \frac{- \Omega \left( N^{\frac{\epsilon}{2}} \right)}{3} \right) .
\]
By the union bound, 
\[
\mathbb{P} \left(\exists v \in V_0, \vert m (v) - N_1 \eta \vert \geq (N_1 \eta)^{\frac{3}{4}} \right) \leq N_0 \exp \left( \frac{- \Omega \left( N^{\frac{\epsilon}{2}} \right)}{3} \right)
\]
and this tends to $0$ as $N$ tends to infinity. By a similar computation, for every $v \in V_1$, 
\[
\mathbb{P} \left(\exists v \in V_1, \vert m (v) - N_0 \eta \vert \geq (N_0 \eta)^{\frac{3}{4}} \right) \leq N_1 \exp \left( \frac{- \Omega \left( N^{\frac{\epsilon}{2}} \right)}{3} \right)
\]
and this tends to $0$ as $N$ tends to infinity. 

Thus, for every $i=0,1$ and every $v \in V_i$, the probability that 
\begin{align*}
& \vert m (v) - N_{1-i} \eta \vert \\
& \leq (N_{1-i} \eta)^{\frac{3}{4}} \\
& = (N_{1-i} \eta)^{- \frac{1}{4}} (N_{1-i} \eta) \\
& \leq (N \eta)^{- \frac{1}{4}} (N_{1-i} \eta) \\
& = O (N^{-\frac{\epsilon}{4}}) (N_{1-i} \eta)
\end{align*}
tends to $1$ as $N$ tends to infinity. 
The error from rounding to integers can be absorbed into the error term, so a.a.s.\ the graph is $(\lfloor N_1 \eta \rfloor, \lfloor N_0 \eta \rfloor,  O (N^{-\frac{\epsilon}{4}})) $-almost biregular.
\end{proof}

Another result we will need is a bound on the second eigenvalue of a bipartite Erd\H{o}s--R\'{e}nyi random graph. The following result is proved in greater generality in \cite{Ash1}:
\begin{theorem}\cite[Theorem A]{Ash1}
\label{thm: lambda bound for ER bipartite graphs}
Let $N \in \mathbb{N}$, $N_0, N_1 \in \mathbb{N}$ and let $C \geq 1$ be a constant such that $N \leq N_0,N_1 \leq C N$. Also, let $\eta : \mathbb{N} \rightarrow (0,1)$ be a function for which there exists $0<\epsilon <1$ such that $\eta (N) = \Omega (N^{\epsilon -1})$. Then, as $N$ tends to infinity, the probability that $G = (V,E) \sim G (N_0,N_1, \eta)$ has
\[
\lambda (V,E) = O (N^{- \frac{\epsilon}{2}} )
\] 
tends to $1$.  
\end{theorem}

We will also need the following result:
\begin{proposition}
\label{prop: counting words}
Let $2 \leq k < \RelLen$ and let $w$ be a reduced word of length $k$ in the alphabet $S = \lbrace s_1^{\pm},...,s_n^{\pm} \rbrace$.  Then  
\[
\vert \lbrace w' : ww' \text{ is a cyclically reduced word of length } \RelLen \rbrace \vert = \frac{(2n-1)^{\RelLen-k+1}}{2n} + O (1),
\]
where $O(1)$ above does not depend on $\RelLen$ or on the prefix or its length.
\end{proposition}

\begin{proof}
Let $w = s w'' t$. Then $ww' $ is a cyclically reduced word of length $\RelLen$ if and only if $t w' s$ is a reduced word of length $\RelLen - k + 2$.  It is shown in \cite[Lemma 2.1]{Park} that for every $s,t \in S$,  the number of reduced words of length $\RelLen -k  + 2$ that start with $t$ and end with $s$ is $\frac{(2n-1)^{\RelLen-k+1}}{2n} + O (1)$.
\end{proof}
 
We can now analyse the graphs $(V^j, E^j), j=1,...,4$. 

\begin{lemma}
\label{lemma: double edges}
Let $\frac{1}{3} \geq d > \frac{1}{4}$ and  $\Gamma \in \mathcal{B} (n,\RelLen,\varrho)$ with $\varrho (\RelLen)  = (2n-1)^{-(1-d) \RelLen}$. For every $1 \leq j \leq 4$,  the following assertions hold a.a.s.:
\begin{enumerate}
\item There are no $v,u \in V^j$ such that $m^j (\lbrace u,v \rbrace) \geq 4$.
\item There is no $v \in V^j$ such that $\vert \lbrace u \in V^j : m^j (\lbrace u,v \rbrace) \geq 2 \rbrace \vert \geq 5$.
\end{enumerate}
\end{lemma}

\begin{proof}
We prove the assertions for $j=1$; the other cases are similar. To simplify the notation, we omit the superscript $1$ and write $(V,E) = (V^1, E^1), m = m^1$, etc. 


Fix $v = (0, w_0),  u = (1, w_1) \in V$ such that $w_0$ and $w_1$ do not have the same first letter.  By \Cref{prop: counting words}, the number of cyclically reduced words of length $\RelLen$ with the prefix $w_0^{-1} w_1$ is $\frac{(2n-1)^{\RelLen - k_0 - k_1+1}}{2n} + O (1)$.  Set $L =  \frac{(2n-1)^{\RelLen - k_0 - k_1+1}}{2n} $.  Then, for fixed $v = (0, w_0),  u = (1, w_1) \in V$ and for each fixed $r \in \mathbb{N},  r \geq 2$, there is a constant $C_0 = C_0 (n,r)$ such that
\begin{align*}
& \mathbb{P} (m (\lbrace u,v \rbrace) \geq r ) \\
& \leq (L + O(1))^r (\varrho (\RelLen))^r \\
& \leq C_0 (2n-1)^{r(\frac{\RelLen}{2}) - r (1-d) \RelLen} \\
& = C_0 (2n-1)^{(d- \frac{1}{2}) r \RelLen} \\
& \leq^{d \leq \frac{1}{3}} C_0 (2n-1)^{-\frac{1}{6} r \RelLen}.
\end{align*}

For the first assertion, we note that the number of pairs of vertices that can have an edge between them is $\leq 2n (2n-1)^{\frac{\RelLen}{2}}$.  Thus, by the union bound (using $r=4$ in the probability bound above),
\begin{align*}
& \mathbb{P} (\exists u,v \in V, m (\lbrace u,v \rbrace) \geq 4 ) \\
& \leq O ( 2n (2n-1)^{\frac{\RelLen}{2}} (2n-1)^{-\frac{2}{3} \RelLen} ) \\
& = O( 2n (2n-1)^{- \frac{1}{6} \RelLen})
\end{align*}
and this tends to $0$ as $\RelLen$ tends to infinity. 

For the second assertion, fix $v \in V$ and distinct $u_1,...,u_5 \in V$ such that $v$ and $u_1,...,u_5$ are on different sides.  The events \(\{m (\{v,u_j\})\geq2\}\), \(j=1,\ldots,5\), are mutually independent since they depend on disjoint sets of candidate relators. Therefore (taking $r=2$ in the computation above),   
\begin{align*}
& \mathbb{P} ( \forall 1 \leq i \leq 5, m (\lbrace u_i,v \rbrace) \geq 2) \\
& \leq O( \left((2n-1)^{-\frac{1}{3} \RelLen } \right)^5 ) \\
& = O( (2n-1)^{-\frac{5}{3} \RelLen } ).
\end{align*} 

For a fixed $v$,  the number of sets of $5$ vertices that can be neighbours of $v$ is $O ( (2n-1)^{\frac{5}{4} \RelLen})$.  Thus, for a fixed $v$,
\begin{align*}
& \mathbb{P} (\exists u_1,...,u_5 \in V u_i \neq u_j, \forall 1 \leq i \leq 5, m (\lbrace u_i,v \rbrace) \geq 2) \\
& \leq O( (2n-1)^{\frac{5}{4} \RelLen} (2n-1)^{-\frac{5}{3} \RelLen}) \\
& = O((2n-1)^{-\frac{5}{12} \RelLen }).
\end{align*} 
Taking a union bound over all $v \in V$ ($\vert V \vert \leq 4n (2n-1)^{\frac{\RelLen}{4} +1}$) shows that the probability that there exists $v$ with at least $5$ neighbours with a double edge is $O ((2n-1)^{-\frac{1}{6} \RelLen})$ and, in particular, tends to $0$ as $\RelLen$ tends to infinity.
\end{proof}

\begin{theorem}
\label{thm: local graphs bireg and spec}
Let $\frac{1}{3} \geq d > \frac{1}{4}$ and  $\Gamma \in \mathcal{B} (n,\RelLen,\varrho)$ with $\varrho (\RelLen)  = (2n-1)^{-(1-d) \RelLen}$. For every $1 \leq j \leq 4$ and every $s,s' \in S, s \neq s'$,  the following assertions hold a.a.s.:
\begin{enumerate}
\item There are $D_0^j, D_1^j \in \mathbb{N}$ (that do not depend on $s,s'$) such that the multigraph $(V^j_{s,s'}, E^j_{s,s'})$ is $(D_0^j, D_1^j, O_{n,d} ((2n-1)^{- \lfloor \frac{\RelLen}{4} \rfloor (d- \frac{1}{4})}))$-almost biregular.
\item $\lambda (V^j_{s,s'}, E^j_{s,s'}) = O_{n,d} ((2n-1)^{- \lfloor \frac{\RelLen}{4} \rfloor (2d- \frac{1}{2})})$. 
\end{enumerate}
\end{theorem}

\begin{proof}
In this proof, we suppress the dependence on $n,d$ in the $O, \Omega$ notation; that is, we use $O, \Omega$ to denote $O_{n,d}, \Omega_{n,d}$.

Let $\Gamma = \langle S \vert R \rangle$, where $\Gamma \in \mathcal{B} (n,\RelLen,\varrho)$.

We prove the assertions for $j=1$; the other cases are similar. To simplify the notation, we omit the superscript $1$ and write $(V,E) = (V^1, E^1), m = m^1$, etc.

We show that, a.a.s., the multigraph $(V,E)$ has a subgraph $(V',E')$ that contains ``almost all'' of the edges of $(V,E)$ and has the following property: in $(V',E')$, for every $s,s' \in S, s \neq s'$, $(V'_{s,s'}, E'_{s,s'})$ is a bipartite Erd\H{o}s--R\'{e}nyi random graph.

Let $s,s' \in S, s \neq s'$. For $v = (0,w_0) \in V_{0,s}$ and $u = (1,w_1) \in V_{1,s'}$,  we denote the number of cyclically reduced words of length $\RelLen$ with the prefix $w_0^{-1} w_1$ by $C_{v,u}$.  By \Cref{prop: counting words},  $C_{v,u} = \frac{(2n-1)^{\RelLen -k_0 - k_1 + 1}}{2n} + O (1)$ for every such $v,u$.  Define 
\[
C_{\min} = \min_{s, s' \in S, s \neq s', v \in V_{0,s},  u \in V_{1,s'}} C_{v,u}
\]
and note that there is a constant $C$ (that does not grow with $\RelLen$) such that $C_{v,u} - C_{\min}  \leq C$ for every $v,u$.  Let $\mathcal{R}_{\RelLen}$ be the set of all cyclically reduced words of length $\RelLen$.  Choose $\mathcal{R}_{\RelLen} ' \subseteq \mathcal{R}_{\RelLen}$ to be a set of relations such that for every $s,s' \in S, s \neq s'$ and every $v = (0,w_0) \in V_{0,s}$ and $u = (1,w_1) \in V_{1,s'}$,  the number of words in $\mathcal{R}_{\RelLen} '$ with a prefix $w_0^{-1} w_1$ is exactly $C_{\min}$.

Let $(V'',E'') = (V'' (R), E'' (R))$ be the sub-multigraph of $(V,E)$ given by $(V^1 (R \cap \mathcal{R}_{\RelLen} '), E^1 (R \cap \mathcal{R}_{\RelLen} '))$, i.e., we construct the multigraph $(V^1, E^1)$ for the presentation $\langle S \vert R \cap \mathcal{R}_{\RelLen} ' \rangle$. 

We will show that, a.a.s., $(V,E) = (V'',E'')$.  Indeed,
\begin{align*}
& \mathbb{P} ((V, E) \neq (V'', E'')) \\
& \leq C \vert V_{0} \vert \vert V_{1} \vert \varrho (\RelLen) \\
& \leq C (2n)^2 (2n-1)^{k_0+k_1-2} (2n-1)^{-(1-d) \RelLen} \\
& \leq^{d \leq \frac{1}{3}, k_0+ k_1 \leq \frac{\RelLen}{2}+2} C (2n)^2 (2n-1)^{- \frac{1}{6} \RelLen},
\end{align*}
and thus this probability tends to $0$ as $\RelLen$ tends to infinity.  

Next, we define a subgraph of $(V'',E'')$, denoted by $(V',E')$, and write $m'$ for its degree function: $V' = V'' (=V)$ and $E'$ is defined by
\[
m ' (\lbrace u, v \rbrace) = 
\begin{cases}
1 & \lbrace u, v \rbrace \in E'' \\
0 & \lbrace u, v \rbrace \notin E''
\end{cases}.
\]
In other words, $(V',E')$ is obtained by collapsing the multiple edges of $(V'',E'')$ into a single edge. By \Cref{lemma: double edges}, it holds a.a.s.\ that for every $v \in V$, $0 \leq m (v) - m' (v) \leq 8$ (at most $4$ triple edges incident to a vertex). We have proved above that if $m''$ is the degree function of $(V'',E'')$, then a.a.s. $m'' \equiv m$ and thus a.a.s.\ for every $v \in V''$,  $0 \leq m '' (v) - m' (v) \leq 8$. 

For every $s,s' \in S, s \neq s'$, we show that the graph $(V_{s,s'}', E_{s,s'}')$ is a $((2n-1)^{k_0-1}, (2n-1)^{k_1-1},\eta)$  bipartite Erd\H{o}s--R\'{e}nyi random graph with $\eta = \Omega ((2n-1)^{\RelLen (d - \frac{1}{2})})$. Let $s,s' \in S, s \neq s'$.  Given $v = (0,w_0) \in V_{0,s}'$ and $u = (1,w_1) \in V_{1,s'}'$,  the number of cyclically reduced words of length $\RelLen$ with the prefix $w_0^{-1} w_1$ in $\mathcal{R}_\RelLen '$ is $C_{\min} = \Omega (\frac{(2n-1)^{\RelLen - k_0 -k_1+1}}{2n})$.  The vertices $v,u$ are connected by an edge in $(V', E')$ with probability
\begin{align*}
& \eta \\
& = 1- (1- \varrho (\RelLen))^{C_{\min}} \\
& \sim C_{\min} \varrho (\RelLen) \\
& = \Omega \left(\frac{(2n-1)^{\RelLen d - k_0 -k_1+1}}{2n} \right) \\
& = \Omega ((2n-1)^{\RelLen (d - \frac{1}{2})}).
\end{align*}
Note that these edge-indicator variables are mutually independent for distinct pairs of vertices on opposite sides. 

If we set $N = (2n-1)^{\lfloor \frac{\RelLen}{4} \rfloor-1}$ and $N_i = (2n-1)^{k_i-1}$, then $N \leq N_0, N_1 \leq (2n-1) N$ and, for $\eta$ calculated above, we have 
\begin{align*}
& \eta \\
& = \Omega \left( N^{\frac{\RelLen (d - \frac{1}{2})}{\lfloor \frac{\RelLen}{4} \rfloor}} \right) \\
& = \Omega \left( N^{\frac{\RelLen (d - \frac{1}{2})}{\frac{\RelLen}{4}}} \right) \\
& = \Omega \left( N^{-1 + (4d-1)} \right).
\end{align*}
Thus, for $\epsilon = 4d-1$, we have $\eta = \Omega (N^{-1+ \epsilon})$. 

By \Cref{prop: bireg for ER bipartite graphs}, $(V_{s,s'}',E_{s,s'}')$ is a.a.s.  $(\lfloor N_1 \eta \rfloor , \lfloor N_0 \eta \rfloor, O (N^{- \frac{\epsilon}{4}}))$-almost biregular. We set $D_0 = \lfloor N_1 \eta \rfloor, D_1 = \lfloor N_0 \eta \rfloor$ and note that $D_0, D_1 = \Omega (N^{\epsilon})$. For every $v\in V$, we have $0\leq m(v)-m'(v)\leq8$. Consequently, for $i=0,1$, every $s,s'\in S$ with $s\neq s'$, and every $v\in V_i\cap V_{s,s'}$,
\begin{align*}
& \vert m_{s,s'} (v) - D_i \vert \\
& \leq 8+ \vert m_{s,s'} ' (v) - D_i \vert \\
& \leq D_i (\frac{8}{D_i} + O (N^{- \frac{\epsilon}{4}})) \\
& =^{D_i =\Omega (N^{\epsilon})} D_i O (N^{- \frac{\epsilon}{4}}).
\end{align*}
It follows that $(V_{s,s'}, E_{s,s'})$ is a.a.s. $(D_0, D_1, O (N^{- \frac{\epsilon}{4}}))$-almost biregular. We note that $D_0, D_1$ do not depend on $s,s' \in S, s \neq s'$ and that 
\[
N^{- \frac{\epsilon}{4}} = (2n-1)^{ - (\lfloor \frac{\RelLen}{4} \rfloor-1)(d-\frac{1}{4})} = O ((2n-1)^{- \lfloor \frac{\RelLen}{4} \rfloor (d-\frac{1}{4})})
\]
and thus the first assertion holds.

Keeping the notation $N, \epsilon, D_0, D_1$, we obtain from \Cref{thm: lambda bound for ER bipartite graphs} that, a.a.s., for every $s,s' \in S,s \neq s'$, 
\[
\lambda (V_{s,s'}',E_{s,s'}') = O (N^{- \frac{\epsilon}{2}}).
\]

By \Cref{lemma: double edges} and the above calculation, the following holds a.a.s.\ for every $s,s' \in S,s \neq s'$: 
\begin{align*}
& m_{s,s'} '(v) \\
& \geq m_{s,s'} (v) - 8 \\
& \geq m_{s,s'} (v) \left(1- \frac{8}{\min_{i=0,1} D_i (1-O (N^{- \frac{\epsilon}{4}}))} \right) \\
& =^{D_i =\Omega (N^{\epsilon})} m_{s,s'} (v) \left(1-O(N^{-\epsilon}) \right).
\end{align*}
Thus, by \Cref{prop: submultigraph lambda bound}, it holds a.a.s.\ that
\[
\lambda (V_{s,s'}, E_{s,s'}) \leq \lambda (V_{s,s'} ', E_{s,s'} ') + O(N^{-\epsilon}) = O (N^{- \frac{\epsilon}{2}}).
\]
With the notation above,
\[
O (N^{- \frac{\epsilon}{2}}) = O ((2n-1)^{- (\lfloor \frac{\RelLen}{4} \rfloor-1)(2d-\frac{1}{2})}) = O ((2n-1)^{- \lfloor \frac{\RelLen}{4} \rfloor (2d-\frac{1}{2})})
\]
as needed.
\end{proof}

\begin{theorem}
\label{thm: prop T binomial model}
Let $d \in (\frac{1}{4},1)$ and  $\Gamma \in \mathcal{B} (n,\RelLen,\varrho)$ with $\varrho (\RelLen)  = (2n-1)^{-(1-d) \RelLen}$. Then a.a.s. $\Gamma$ has property $(T)$.
\end{theorem}

\begin{proof}
Since property $(T)$ is monotone, we can assume without loss of generality that $\frac{1}{3} \geq d > \frac{1}{4}$.

Let $\Gamma \in \mathcal{B} (n,\RelLen,\varrho)$. By \Cref{thm: Hilbert and Lp fixed point criterion}, it suffices to show that there are constants $0 \leq \varepsilon < 1,  D_0^j, D_1^j \in \mathbb{N}, j=1,...4,  0 \leq \lambda$ such that the following conditions hold a.a.s.:
\begin{itemize}
\item For every $s, s' \in S,  s \neq s'$ and every $j=1,...,4$,  the multigraph $(V_{s,s'}^j, E_{s,s'}^j)$ is $(D_0^j,D_1^j, \varepsilon)$-almost biregular.
\item For every $s, s' \in S,  s \neq s'$ and every $j=1,...,4$,  $\lambda (V_{s,s'}^j, E_{s,s'}^j) \leq \lambda$.
\item \[
8 (2n-1) \lceil \frac{\RelLen}{4} \rceil^2 \left( \frac{\varepsilon (5 + 3\varepsilon)}{1 -\varepsilon} + (1+\varepsilon) \lambda \right) < 1.
\]
\end{itemize}

By \Cref{thm: local graphs bireg and spec}, the first two conditions hold with $\varepsilon = O_{n,d} ((2n-1)^{- \lfloor \frac{\RelLen}{4} \rfloor (d- \frac{1}{4})})$ and $\lambda = O_{n,d} ((2n-1)^{- \lfloor \frac{\RelLen}{4} \rfloor (2d- \frac{1}{2})})$. Thus 
\begin{align*}
& 8 (2n-1) \lceil \frac{\RelLen}{4} \rceil^2 \left( \frac{\varepsilon (5 + 3\varepsilon)}{1 -\varepsilon} + (1+\varepsilon) \lambda \right) \\
& = 8 (2n-1) \lceil \frac{\RelLen}{4} \rceil^2 \left( O_{n,d} ((2n-1)^{- \lfloor \frac{\RelLen}{4} \rfloor (d- \frac{1}{4})}) + O_{n,d} ((2n-1)^{- \lfloor \frac{\RelLen}{4} \rfloor (2d- \frac{1}{2})}) \right) \\
& = \lceil \frac{\RelLen}{4} \rceil^2 O_{n,d} ((2n-1)^{- \lfloor \frac{\RelLen}{4} \rfloor (d- \frac{1}{4})}) \\
& \rightarrow^{\RelLen \rightarrow \infty} 0
\end{align*}
Thus, the last condition holds a.a.s., and $\Gamma$ has property $(T)$ a.a.s.
\end{proof}

Combining this theorem with \Cref{prop: DrutuM} gives a proof of \Cref{thm:intro-property-t}:
\begin{proof}[Proof of \Cref{thm:intro-property-t}]
Let $d \in (\frac{1}{4}, 1)$. Since property $(T)$ is monotone, \Cref{prop: DrutuM} and \Cref{thm: prop T binomial model} imply that \(\Gamma\in\mathcal{D}(n,\RelLen,d)\) has property~\((T)\) a.a.s.
\end{proof}

For a group $\Gamma$, let $\mathcal{F}_L(\Gamma)$ denote the $L^p$-fixed point spectrum of $\Gamma$ as defined in Section~\ref{sec:intro-lp}.

\begin{theorem}
\label{thm: prop L-p binomial model}
Let $ d \in (\frac{1}{4},1)$ and  $\Gamma \in \mathcal{B} (n,\RelLen,\varrho)$ with $\varrho (\RelLen)  = (2n-1)^{-(1-d) \RelLen}$. There is a constant $C_{n,d} >0$ such that a.a.s. $\left[1, C_{n,d} \sqrt{\frac{\RelLen}{\log (\RelLen)}} \right]\subseteq \mathcal{F}_L(\Gamma)$. In particular, for every $1 \leq p < \infty$, $\Gamma \in \mathcal{B} (n,\RelLen,\varrho)$ has property $(F_{L^p})$ a.a.s.
\end{theorem}

\begin{proof}
Since the property $\left[1, C_{n,d} \sqrt{\frac{\RelLen}{\log (\RelLen)}} \right]\subseteq \mathcal{F}_L(\Gamma)$ is monotone, we can assume without loss of generality that $\frac{1}{3} \geq d > \frac{1}{4}$.

Let $\Gamma \in \mathcal{B} (n,\RelLen,\varrho)$. By \Cref{thm: Hilbert and Lp fixed point criterion}, for a given $2 \leq p_0 < \infty$, to show that a.a.s. $\Gamma$ has property $(F_{L^{p}})$ for every $p \in [1, p_0]$, it suffices to show that there are constants $0 \leq \varepsilon < 1,  D_0^j, D_1^j \in \mathbb{N}, j=1,...4,  0 \leq \lambda$ such that the following conditions hold a.a.s.:
\begin{itemize}
\item For every $s, s' \in S,  s \neq s'$ and every $j=1,...,4$,  the multigraph $(V_{s,s'}^j, E_{s,s'}^j)$ is $(D_0^j,D_1^j, \varepsilon)$-almost biregular.
\item For every $s, s' \in S,  s \neq s'$ and every $j=1,...,4$,  $\lambda (V_{s,s'}^j, E_{s,s'}^j) \leq \lambda$.
\item \[
2^{p_0-1} p_0^2 (2n-1) \lceil \frac{\RelLen}{4} \rceil^{p_0} \left( \frac{\varepsilon (5 + 3\varepsilon)}{1 -\varepsilon} + (1+\varepsilon) 2 \lambda^{\frac{2}{p_0}} \right) < 1.
\]
\end{itemize}
By \Cref{thm: local graphs bireg and spec}, the first two conditions hold with 
\[
\varepsilon = O_{n,d} ((2n-1)^{- \lfloor \frac{\RelLen}{4} \rfloor (d- \frac{1}{4})})
\]
and $\lambda = O_{n,d} ((2n-1)^{- \lfloor \frac{\RelLen}{4} \rfloor (2d- \frac{1}{2})})$. Assuming that $p_0 \geq 4$, we obtain a.a.s. 
\begin{align*}
& 2^{p_0-1} p_0^2 (2n-1) \lceil \frac{\RelLen}{4} \rceil^{p_0} \left( \frac{\varepsilon (5 + 3\varepsilon)}{1 -\varepsilon} + (1+\varepsilon) 2 \lambda^{\frac{2}{p_0}} \right) \\
& \leq 2^{p_0} p_0^2 \RelLen^{p_0} O_{n,d} \left( (2n-1)^{- \lfloor \frac{\RelLen}{4} \rfloor (2d- \frac{1}{2}) \frac{2}{p_0}} \right).
\end{align*}
Equivalently, there is a constant $A_{n,d}$ such that 
\begin{align*}
& 2^{p_0-1} p_0^2 (2n-1) \lceil \frac{\RelLen}{4} \rceil^{p_0} \left( \frac{\varepsilon (5 + 3\varepsilon)}{1 -\varepsilon} + (1+\varepsilon) 2 \lambda^{\frac{2}{p_0}} \right) \\
& \leq A_{n,d} 2^{p_0} p_0^2 \RelLen^{p_0} (2n-1)^{- \lfloor \frac{\RelLen}{4} \rfloor (2d- \frac{1}{2}) \frac{2}{p_0}}.
\end{align*}

Thus, it suffices to show that there is $C_{n,d}$ such that for $p_0 = C_{n,d} \sqrt{\frac{\RelLen}{\log (\RelLen)} }$ it holds a.a.s.\ that
\[
A_{n,d} 2^{p_0} p_0^2 \RelLen^{p_0}  (2n-1)^{- \lfloor \frac{\RelLen}{4} \rfloor (2d- \frac{1}{2}) \frac{2}{p_0}} < 1.
\]
Indeed, taking logarithms gives the equivalent inequality
\[
\log (A_{n,d}) + 2 \log (p_0) + p_0 (\log (2)) + p_0 \log (\RelLen) - \lfloor \frac{\RelLen}{4} \rfloor (2d- \frac{1}{2}) \frac{2}{p_0} \log (2n-1) <0.
\]
Setting $B_{n,d} = 2 (2d- \frac{1}{2})  \log (2n-1)$, this is equivalent to
\[
\frac{p_0^2 \log (\RelLen)}{\RelLen} \left(\frac{ \log (A_{n,d})}{p_0  \log (\RelLen)}  + \frac{2\log (p_0) }{p_0 \log (\RelLen)} + \frac{\log (2) }{ \log (\RelLen)}+ 1 \right)  - B_{n,d} \frac{1}{\RelLen} \lfloor \frac{\RelLen}{4} \rfloor  <0
\]
and there is $C_{n,d} >0$ such that this holds for $p_0 = C_{n,d} \sqrt{\frac{\RelLen}{\log (\RelLen)} }$ for every sufficiently large $\RelLen$.
\end{proof}

Combining this theorem with \Cref{prop: DrutuM} gives a proof of \Cref{thm:intro-lp}:
\begin{proof}[Proof of \Cref{thm:intro-lp}]
Fix $n\geq2$ and $d\in(\frac14,1)$, and set
\[
d_-=\frac12\left(d+\frac14\right).
\]
Then $\frac14<d_-<d$. Let $C>0$ be the constant supplied by
Theorem~6.9 at density $d_-$.

For each $\RelLen\geq2$, let $(P_{\RelLen})$ be the property
that the presented group $\Gamma$ satisfies
\[
\left[1,C\sqrt{\frac{\RelLen}{\log \RelLen}}\right]
\subseteq \mathcal{F}_L(\Gamma).
\]
For each fixed $\RelLen$, this property is preserved under
adding relators, since Banach fixed point properties are
preserved under quotients.

By \Cref{thm: prop L-p binomial model}, property $(P_{\RelLen})$ holds a.a.s. in the
binomial model at density $d_-$. By monotonicity of the binomial
model, it therefore holds a.a.s. at every density
$d'\in(d_-,1)$.
Applying \Cref{prop: DrutuM} with $d_0=d_-$, we conclude that
$(P_{\RelLen})$ holds a.a.s. in $\mathcal{D}(n,\RelLen,d)$.
This proves the theorem with $C_{n,d}=C$.
\end{proof}

Finally, we discuss fixed point properties for uniformly curved spaces. We need the following definition. Given a function $\ucurvfun: [0,1] \rightarrow  [0,1]$ that is monotone increasing and satisfies $\ucurvfun (0) =0$, we say that $\ucurvfun$ is \textit{logarithmically rapidly decaying} if for every $\beta \geq 1$, $\lim_{t \rightarrow 0^+} \ucurvfun (t) \cdot (\log^{\beta} (t^{-1})) =0$, or equivalently if for every $\beta \geq 1$, $\lim_{x \rightarrow \infty} \ucurvfun (e^{-x}) \cdot x^{\beta} =0$. For example, for any constant $\theta >0$, a function satisfying $\ucurvfun (t) = O (t^{\theta})$ is logarithmically rapidly decaying.

Given constants $2 \leq p < \infty, 0 < C \leq 1$ and a monotone increasing function $\ucurvfun: [0,1] \rightarrow  [0,1]$ such that $\ucurvfun (0) = 0$ and $\lim_{t \rightarrow 0^{+}} \ucurvfun (t) = 0$, we define $\mathcal{E} (p,C,\ucurvfun)$ to be the class of all real Banach spaces satisfying the following conditions: every $\B \in \mathcal{E} (p,C,\ucurvfun)$ is
\begin{itemize}
\item $(p,C)$-uniformly convex;
\item uniformly curved with $\B \in \mathcal{E}^{\text{u-curved}}_\ucurvfun$; and
\item uniformly smooth.
\end{itemize}

\begin{theorem}
\label{thm: prop FB binomial model}
Let $ d \in (\frac{1}{4},1)$ and  $\Gamma \in \mathcal{B} (n,\RelLen,\varrho)$ with $\varrho (\RelLen)  = (2n-1)^{-(1-d) \RelLen}$. Then for all constants $2 \leq p < \infty, 0 < C \leq 1$, and every monotone increasing function $\ucurvfun: [0,1] \rightarrow  [0,1]$ that is logarithmically rapidly decaying and satisfies $\ucurvfun (0) =0$, it holds a.a.s.\ that $\Gamma$ has property $(F_{\mathcal{E} (p,C,\ucurvfun)})$.
\end{theorem}

\begin{proof}
Since Banach fixed point properties are monotone, we can assume without loss of generality that $\frac{1}{3} \geq d > \frac{1}{4}$.

Let $\Gamma \in \mathcal{B} (n,\RelLen,\varrho)$. Fix $p,C, \ucurvfun$ as above. 

By \Cref{thm: uni-curv fixed point criterion}, it suffices to show that there are constants $0 \leq \varepsilon < 1,  D_0^j, D_1^j \in \mathbb{N}, j=1,...4,  0 \leq \lambda$ such that the following conditions hold a.a.s.:
\begin{itemize}
\item For every $s, s' \in S,  s \neq s'$ and every $j=1,...,4$,  the multigraph $(V_{s,s'}^j, E_{s,s'}^j)$ is $(D_0^j,D_1^j, \varepsilon)$-almost biregular.
\item For every $s, s' \in S,  s \neq s'$ and every $j=1,...,4$,  $\lambda (V_{s,s'}^j, E_{s,s'}^j) \leq \lambda$.
\item 
\[
(2n-1) \lceil \frac{\RelLen}{4} \rceil^p \frac{p}{2C} \left( \frac{\varepsilon (5 + 3\varepsilon)}{1 -\varepsilon} + (1+\varepsilon) 4 \left( \ucurvfun (\lambda) \right)^{\frac{2}{p}} \right) < 1.
\]
\end{itemize}

By \Cref{thm: local graphs bireg and spec}, the first two conditions hold with $\varepsilon = O_{n,d} ((2n-1)^{- \lfloor \frac{\RelLen}{4} \rfloor (d- \frac{1}{4})})$ and $\lambda = O_{n,d} ((2n-1)^{- \lfloor \frac{\RelLen}{4} \rfloor (2d- \frac{1}{2})})$. Thus, using the assumption that $\ucurvfun$ is logarithmically rapidly decaying, it holds a.a.s.\ for $\beta = p^2$ that
\begin{align*}
& (2n-1) \lceil \frac{\RelLen}{4} \rceil^p \frac{p}{2C} \left( \frac{\varepsilon (5 + 3\varepsilon)}{1 -\varepsilon} + (1+\varepsilon) 4 \left( \ucurvfun (\lambda) \right)^{\frac{2}{p}} \right) \\
& \leq (2n-1) \lceil \frac{\RelLen}{4} \rceil^p \frac{p}{2C} \Bigl( O_{n,d} ((2n-1)^{- \lfloor \frac{\RelLen}{4} \rfloor (d- \frac{1}{4})}) \\*
& \qquad + O \left( \ucurvfun \left( O_{n,d} ((2n-1)^{- \lfloor \frac{\RelLen}{4} \rfloor (2d- \frac{1}{2})}) \right)^{\frac{2}{p}} \right) \Bigr) \\
& \leq \RelLen^p \frac{p}{2C} \left( O_{n,d} ((2n-1)^{- \lfloor \frac{\RelLen}{4} \rfloor (d- \frac{1}{4})}) + O_{n,d, \ucurvfun} \left( \left( \frac{1}{\RelLen^{p^2}} \right)^{\frac{2}{p}} \right) \right) \\
& \leq \RelLen^p \frac{p}{2C} \left( O_{n,d} ((2n-1)^{- \lfloor \frac{\RelLen}{4} \rfloor (d- \frac{1}{4})}) + O_{n,d, \ucurvfun} \left( \frac{1}{\RelLen^{2p}} \right) \right)
\end{align*}
This expression tends to $0$ when $\RelLen$ tends to infinity and, in particular, is strictly less than $1$ for sufficiently large $\RelLen$.
\end{proof}

Combining this theorem with \Cref{prop: DrutuM} gives a proof of \Cref{thm:intro-banach}:
\begin{proof}[Proof of \Cref{thm:intro-banach}]
Let $d \in (\frac{1}{4}, 1)$. Also, let $2 \leq p < \infty, 0 < C \leq 1$ be constants, and let $\ucurvfun: [0,1] \rightarrow  [0,1]$ be a monotone increasing function  that is logarithmically rapidly decaying and satisfies $\ucurvfun (0) =0$.

Since property $(F_{\mathcal E(p,C,\ucurvfun)})$ is preserved under
quotients, \Cref{thm: prop FB binomial model} and \Cref{prop: DrutuM}, with $d_0=1/4$,
imply that a random group
$\Gamma\in\mathcal D(n,\RelLen,d)$ has property
$(F_{\mathcal E(p,C,\ucurvfun)})$ a.a.s.
\end{proof}

\begin{remark}
\label{remark: theta Hil is contained in pCf}
We note that by \Cref{prop: theta-Hilbertian uniformly curved} and \Cref{prop: theta-Hilbertian convexity and smoothness}, for every $0 < \theta_0 \leq 1$, there are $2 \leq p < \infty$ and $0 < C \leq 1$ such that $\mathcal{E} (p,C, \ucurvfun (t) = t^{\theta_0})$ contains the class of all $\theta$-Hilbertian Banach spaces for $\theta_0 \leq \theta \leq 1$. Also, by \Cref{prop: uni-curv is closed under ultraproducts}, this $\mathcal{E} (p,C, \ucurvfun (t) = t^{\theta_0})$ is closed under subspaces and quotients and thus contains all the subquotients of $\theta$-Hilbertian Banach spaces for $\theta_0 \leq \theta \leq 1$. 
\end{remark}

When considering the fixed point properties of a random group for a single Banach space, we can relax the uniform convexity and uniform smoothness requirements and instead require isomorphism to a uniformly curved Banach space with a logarithmically rapidly decaying function. To prove this, we need some auxiliary results:

\begin{lemma}[Invariant renorming]
\label{lem: invariant power convex smooth renorming}
Let $(\B,\Vert\cdot\Vert)$ be a super-reflexive real Banach space.
There exist an equivalent norm $\Vert\cdot\Vert_{\star}$ and
constants $2\leq p<\infty$ and $0<C\leq 1$ such that
$(\B,\Vert\cdot\Vert_{\star})$ is $(p,C)$-uniformly convex and
uniformly smooth, and every surjective linear isometry of
$(\B,\Vert\cdot\Vert)$ is also an isometry of
$(\B,\Vert\cdot\Vert_{\star})$.
\end{lemma}

\begin{proof}
Let $O (\B)$ be the group of all surjective linear isometries
of $(\B,\Vert\cdot\Vert)$.
We use Pisier's power-type renorming and interpolation results
\cite[Theorem~11.37, Corollary~11.45 and Proposition~11.44]{Pis16}
together with the invariant norm construction in the proof of
\cite[Proposition~2.3]{BFGM}.

First, pass to the complexification of $\B$ defined above, which we denote by $\B_{\mathbb{C}}$.
It is super-reflexive, and every element of $O (\B)$ extends
to a complex linear isometry in $O (\B_{\mathbb{C}})$; see the real-to-complex reduction
in the proof of \cite[Corollary~11.45]{Pis16}.

By \cite[Theorem~11.37]{Pis16}, we choose an equivalent $p_0$-uniformly convex norm $\Vert . \Vert_0 '$ on this
complexification, where $2\leq p_0 <\infty$, and set
\[
    N_0(\xi)=\sup_{g\in O (\B_{\mathbb{C}})} \Vert g\xi \Vert_0 '.
\]
Since $O (\B_{\mathbb{C}})$ is uniformly bounded with respect to $\Vert . \Vert_0 '$, this is an
equivalent $O (\B_{\mathbb{C}})$-invariant norm.
The modulus estimate in the proof of
\cite[Proposition~2.3]{BFGM} shows that $N_0$ is still
$p_0$-uniformly convex.

Apply the same construction to the dual space and the adjoint
group $O (\B_{\mathbb{C}}^{*}) = \lbrace (g^{-1})^* : g \in O (\B_{\mathbb{C}}) \rbrace$, and then dualize.
Dualizing this invariant norm gives an equivalent
$O(\B_{\mathbb C})$-invariant norm $N_1$ whose dual norm is
$p_1$-uniformly convex for some $p_1\geq2$.
Consequently, $N_1$ is uniformly smooth.
Complex interpolation between $N_0$ and $N_1$ at the parameter $1/2$
gives a norm that is $2p_0$-uniformly convex and uniformly smooth,
by \cite[Proposition~11.44]{Pis16}.
Every $g\in O (\B_{\mathbb{C}})$ and its inverse are contractions for
both endpoint norms, so they remain contractions for the
interpolated norm. Hence $O (\B_{\mathbb{C}})$ acts isometrically.
Restricting this norm to the original real space $\B$ proves
the assertion.
\end{proof}

\begin{corollary}
\label{corollary: prop FBstar implies FB}
Let $\B$ be a super-reflexive real Banach space. There exist an equivalent norm $\Vert\cdot\Vert_{\star}$ and
constants $2\leq p<\infty$ and $0<C\leq 1$ such that:
\begin{itemize}
\item The space $\B_{\star} =(\B,\Vert\cdot\Vert_{\star})$ is $(p,C)$-uniformly convex and uniformly smooth.
\item For every group $\Gamma$, if $\Gamma$ has property $(F_{\B_{\star}})$ for $\B_{\star} =(\B,\Vert\cdot\Vert_{\star})$, then $\Gamma$ has property $(F_{\B})$. 
\end{itemize}
\end{corollary}

\begin{proof}
Let $\Vert\cdot\Vert_{\star}$ be the norm given by \Cref{lem: invariant power convex smooth renorming}. By \Cref{lem: invariant power convex smooth renorming}, $\B_{\star} = (\B, \Vert\cdot\Vert_{\star})$ has the required uniform convexity and uniform smoothness properties. 

We prove the second assertion by contraposition. Let $\Gamma$ be a group that does not have property $(F_{\B})$. Then there is an affine isometric action $\rho$ on $\B$ without a fixed point. Thus, it is sufficient to show that $\rho$ is also an affine isometric action on $\B_{\star} = (\B, \Vert\cdot\Vert_{\star})$ (and thus there is an affine isometric action on $\B_{\star}$ that does not have a fixed point). 

Indeed, let $\pi : \Gamma \rightarrow O (\B)$ and $c: \Gamma \rightarrow \B$ be the isometric representation and the cocycle associated with $\rho$. By \Cref{lem: invariant power convex smooth renorming}, $\pi : \Gamma \rightarrow O (\B_{\star})$ is an isometric representation of $\Gamma$ and thus (since $c$ is a cocycle with respect to $\pi$), $\rho$ is an affine isometric action of $\Gamma$ on $\B_{\star} = (\B,\Vert\cdot\Vert_{\star})$. 
\end{proof}

\begin{theorem}
\label{thm:binomial-log-curved-isomorphic}
Fix $n \geq 2$ and $d \in (\frac14,1)$, and let
$\varrho(\RelLen)=(2n-1)^{-(1-d)\RelLen}$.
Let $\B$ be a fixed real Banach space isomorphic to a space in
$\mathcal{E}^{\text{u-curved}}_\ucurvfun$, where
$\ucurvfun:[0,1]\to[0,1]$ is monotone increasing, logarithmically rapidly
decaying, and satisfies $\ucurvfun(0)=0$.
Then a random group $\Gamma\in\mathcal{B}(n,\RelLen,\varrho)$ has property
$(F_{\B})$ a.a.s.
\end{theorem}

\begin{proof}
The assertion is immediate if $\B=\{0\}$, so assume that $\B\neq\{0\}$.
Fix a Banach space $\B_0\in\mathcal{E}^{\text{u-curved}}_\ucurvfun$ and an
isomorphism $U:\B\to\B_0$, and set
\[
K=\|U\|\,\|U^{-1}\|\geq 1,
\qquad
\ucurvfun' (t)=\min\{1,K\ucurvfun(t)\},\quad t\in[0,1].
\]
For every finite set $V$, every weight function $m:V\to\mathbb{R}_{>0}$,
and every fully contractive operator
$T:\ell^2(V,m)\to\ell^2(V,m)$ with
$\|T\|_{B(\ell^2(V,m))}\leq t$, conjugation by the coordinatewise
extension of $U$ gives
\begin{align*}
\|T\otimes I_{\B}\|_{B(\ell^2(V,m;\B))}
&\leq K\|T\otimes I_{\B_0}\|_{B(\ell^2(V,m;\B_0))} \\
&\leq K\ucurvfun(t).
\end{align*}
Full contractivity also bounds the left-hand side by $1$.
Thus $\B\in\mathcal{E}^{\text{u-curved}}_{\ucurvfun'}$.
The function $\ucurvfun '$ is monotone increasing, logarithmically rapidly
decaying, and satisfies $\ucurvfun '(0)=0$.
In particular, $\B$ is super-reflexive by \Cref{thm: uni. cur. are s-reflex.}.

Apply \Cref{corollary: prop FBstar implies FB} to $\B$ and let $\B_{\star} = (\B, \Vert . \Vert_{\star})$ be the $(p,C)$-uniformly convex and uniformly smooth Banach space given by this corollary. By \Cref{corollary: prop FBstar implies FB}, it is enough to prove that $\Gamma\in\mathcal{B}(n,\RelLen,\varrho)$ has property
$(F_{\B_{\star}})$ a.a.s.

The norm $\Vert . \Vert_{\star}$ is equivalent to the norm of $\B$. Thus there are constants $0 < a < b$ such that for every $\xi \in \B$, $a \Vert \xi \Vert \leq \Vert \xi \Vert_{\star} \leq b \Vert \xi \Vert$. We denote  $\ucurvfun_{\star} (t) = \min \lbrace 1, \frac{b}{a} \ucurvfun ' (t) \rbrace, \forall t \in [0,1]$. It follows that for every fully contractive operator
$T:\ell^2(V,m)\to\ell^2(V,m)$ with
$\|T\|_{B(\ell^2(V,m))}\leq t$, it holds that 
\begin{align*}
\|T\otimes I_{\B_\star}\|_{B(\ell^2(V,m;\B_\star))}
&\leq \frac{b}{a} \|T\otimes I_{\B}\|_{B(\ell^2(V,m;\B))} \\
&\leq \frac{b}{a} \ucurvfun '(t).
\end{align*}
Full contractivity also bounds the left-hand side by $1$. Thus $\B_{\star} \in \mathcal{E}^{\text{u-curved}}_{\ucurvfun_{\star}}$.

We note that by assumption $\ucurvfun$ is logarithmically rapidly
decaying and thus so are $\ucurvfun' , \ucurvfun_{\star}$. 

By the above, $\B_{\star} \in \mathcal{E} (p,C,\ucurvfun_{\star})$ with $\ucurvfun_{\star}$ logarithmically rapidly
decaying. Thus by \Cref{thm: prop FB binomial model}, $\Gamma\in\mathcal{B}(n,\RelLen,\varrho)$ has property
$(F_{\B_{\star}})$ a.a.s. and by \Cref{corollary: prop FBstar implies FB} it follows that $\Gamma\in\mathcal{B}(n,\RelLen,\varrho)$ has property
$(F_{\B})$ a.a.s. as needed.
\end{proof}

\begin{proof}[Proof of Theorem~\ref{thm:density-isomorphic-uniformly-curved}]
Fix $n,d,\B,\B_0$ as in the statement, and define
$\ucurvfun (t)= \ucurvfun_{\B_0}(t)$ for $0<t\leq1$, with $\ucurvfun(0)=0$.
Then $\ucurvfun$ is nondecreasing and logarithmically rapidly decaying,
and $\B_0\in\mathcal{E}^{\text{u-curved}}_{\ucurvfun}$.

The property $(F_{\B})$ is preserved under quotients.
By \Cref{thm:binomial-log-curved-isomorphic}, for every $d'\in(\frac14,1)$, a random group
in $\mathcal{B}(n,\RelLen,\varrho_{d'})$, where
\[
\varrho_{d'}(\RelLen)=(2n-1)^{-(1-d')\RelLen},
\]
has property $(F_{\B})$ a.a.s.
Applying \Cref{prop: DrutuM} with $d_0=\frac14$ therefore shows that
a random group $\Gamma\in\mathcal{D}(n,\RelLen,d)$ has property
$(F_{\B})$ a.a.s.
\end{proof}

\begin{remark}
As we already noted above (see \Cref{remark: theta Hil is contained in pCf}), every subquotient of a $\theta$-Hilbertian Banach space for $0 < \theta \leq 1$ is uniformly curved with $\ucurvfun (t) = t^{\theta}$ (which is logarithmically rapidly decaying). We note that by \cite[Remark 17]{Naor21}, there is no current example of a super-reflexive Banach space that is not isomorphic to a subquotient of a $\theta$-Hilbertian Banach space (although such spaces conjecturally exist). 
\end{remark}

\section*{Funding}

The author was funded by ISF grant no. 242/24. 

\section*{AI methodology}

In preparing this paper, I made meaningful use of ChatGPT 6 beyond routine grammar and language editing, including assistance in exploring arguments, checking calculations, and discussing mathematical ideas. All mathematical claims and proofs in the final manuscript were independently checked and remain the responsibility of the author.

\bibliographystyle{unsrtnat}
\bibliography{bibl}
\end{document}